\documentclass[12pt]{article}
\usepackage[utf8]{inputenc} 
\usepackage[T1]{fontenc}    

\usepackage{stmaryrd}

\usepackage{amssymb}
\usepackage{amsfonts}
\usepackage{amsmath}
\usepackage{amsthm}
\usepackage[]{algorithm2e}

\usepackage{xcolor}
\usepackage[a4paper,left=3.cm,right=3.cm,top=3.cm,bottom=3.5cm]{geometry}
\usepackage{hyperref}
\hypersetup{
    colorlinks,
    citecolor=black,
    filecolor=black,
    linkcolor=black,
    urlcolor=black,linktoc=all
}

\usepackage{enumitem}
\usepackage{pifont}

\newcommand{\E}{E}
\newcommand{\K}{K}

\newcommand{\dom}{\text{dom }}

\newtheorem{definition}{Definition}
\newtheorem{proposition}{Proposition}
\newtheorem{theoreme}{Theorem}
\newtheorem{corollaire}{Corollary}
\newtheorem{remark}{Remark}
\newtheorem{example}{Example}
\newtheorem{lemme}{Lemma}
\newcommand{\R}{\mathbb{R}}
\newcommand{\N}{\mathbb{N}}
\newcommand{\id}{\operatorname{Id}}

\newcommand{\prox}{\operatorname{prox}}
\newcommand{\proj}{\operatorname{Proj}}
\newcommand{\rprox}{\operatorname{Rprox}}
\newcommand{\ps}[2]{\langle #1, #2 \rangle}
\newcommand{\norm}[1]{\| #1 \|}
\newcommand{\be}{\begin{equation}}
\newcommand{\ee}{\end{equation}}
\newcommand{\beq}{\begin{equation}}
\newcommand{\eeq}{\end{equation}}
\newcommand{\bi}{\begin{itemize}}
\newcommand{\ei}{\end{itemize}}

\newcommand{\comment}[1]{}

\DeclareMathOperator{\argmin}{argmin}

\newcommand{\uargmin}[1]{\underset{#1}{\argmin}\;}

\newcommand{\pa}[1]{\left( #1 \right)}
\newcommand{\Uu}{\mathcal{U}}
\newcommand{\Vv}{\mathcal{V}}
\usepackage{authblk}

\title{\vspace*{2cm}Optimization with First Order Algorithms\vspace{0.5cm}}
\date{}
\author[1]{Charles Dossal} 
\author[2]{Samuel Hurault} 
\author[3]{Nicolas Papadakis}

\affil[1]{IMT, Univ. Toulouse, INSA Toulouse, Toulouse, France}
\affil[2]{ENS Paris - CNRS}
\affil[3]{Univ. Bordeaux, CNRS, INRIA, Bordeaux INP, IMB, UMR 5251, F-33400 Talence, France}
\begin{document}

\maketitle

\vspace{.5cm}
\tableofcontents

\newpage

\section{Introduction}
These notes focus on the minimization of convex functionals using first-order optimization methods, which are fundamental in many areas of applied mathematics and engineering. The primary goal of this document is to introduce and analyze the most classical first-order optimization algorithms. We aim to provide readers with both a practical and theoretical understanding in how and why these algorithms converge to minimizers of convex functions.

The main algorithms covered in these notes include gradient descent, Forward-Backward splitting, Douglas-Rachford splitting, the Alternating Direction Method of Multipliers (ADMM), and Primal-Dual algorithms. All these algorithms fall into the class of first order methods, as they only involve gradients and subdifferentials, that are first order derivatives of the functions to optimize. For each method, we provide convergence theorems, with precise assumptions and conditions under which the convergence holds, accompanied by complete proofs.

Beyond convex optimization, the final part of this manuscript extends the analysis to nonconvex problems, where we discuss the convergence behavior of these same first-order methods under broader assumptions. To contextualize the theory, we also include a selection of practical examples illustrating how these algorithms are applied in different image processing problems. 

We mention the books \cite{BauschkeCombettes,bauschke2012firmly} for more complete results and proofs, as well as the lecture notes~\cite{weiss2015elements,rondepierre2017methodes}, the handbook~\cite{garrigos2023handbook} including stochastic aspects  and the review paper~\cite{condat2023proximal} that contains many details on primal-dual algorithms.

\paragraph{Organization} In section~\ref{sec:background}, we  give the necessary background on convex functions, non expansive operators and smooth (differentiable) functions, that will used all along the document.
Section~\ref{sec:smooth_opt}  focuses on the minimization of smooth  convex functions with gradient descent schemes. Section~\ref{sec:nonsmooth_opt} is dedicated to non smooth convex functions and proximal splitting algorithms.
We introduce duality tools and primal dual algorithms in section~\ref{sec:duality}.  Section~\ref{sec:nonconvex_opt} explores the extension of previous algorithms to the optimization of nonconvex functionals. 
Finally,  examples of the application of proximal splitting algorithms to imaging problems are presented in section~\ref{sec:examples}.

\section{Definitions and background}\label{sec:background}
We recall that the objective of this document is to present algorithm computing minimizers of functions $f$. 
We first introduce the main notations and definitions in section~\ref{sec:def} in order to give condition for the existence of minimizers of $f$.
As we will consider the optimization of smooth and non smooth functions, we finally give definitions related to differentiable functions with Lipschitz gradient in section~\ref{sec:diff}.

\subsection{Existence of minimizers for convex functions}\label{sec:def}

In the whole document, $\E$ is an Euclidean vectorial space of finite dimension, equipped with  an inner product $\langle\cdot,\cdot\rangle$ and 	a norm $||.||=\langle\cdot,\cdot\rangle^{\frac12}$. We can typically consider that $\E=\R^d$.

\begin{definition}[Domain]
Let $f$ be a function defined from $\E$ to $\bar{\R}=\R\cup{+\infty}$, 
we denote the domain of  $f$ as
$dom(f)=\{x\in \E\text{ such that }f(x)\neq+\infty\}$.
\end{definition}

\begin{definition}[Coercivity]
A function  $f$  is  coercive if 
$\underset{\norm{x}\to+\infty}{\lim}f(x)=+\infty$.
\end{definition}
\begin{definition}[Lower semi-continuity, l.s.c]
A function $f$ defined from $\E$ to $\R\cup +\infty$ 
is lower semi-continuous (l.s.c) if, $\forall x\in\E$, 
$\underset{y\to x}{\lim\inf}f(y)\geqslant f(x)$. 
\end{definition}
\begin{remark}For later use, we recall interesting properties related to lower semi-continuous functions:
\begin{itemize}
    \item 
Lower  semi-continuity is stable by sum.
  \item Continuous functions are lower semi-continuous. 
    \item The characteristic function $\iota_C(x)$ of a closed convex set $C$ is lower semi-continuous.
    \item If $f$ is l.s.c, for all $ \alpha\in\R$, the sets
$$\{x\in\E,\, f(x)\leqslant \alpha\}\text{ and } \{(x,\alpha)\in \E\times \R\text{ such that }f(x)\leqslant \alpha\}$$
are closed.
\end{itemize}
\end{remark}
\begin{definition}[Proper]
A function $f$ from $\E$ to $\bar{\R}=\R\cup \{\pm\infty\}$ 
is proper if $\exists x\in \E$ such that  $f(x)<+\infty$ and  $f(x)>-\infty$, $\forall x\in \E$.
\end{definition}

\begin{definition}[Convexity]
Let $f$ be a function defined from $\E$ to $\bar{\R}$, 
$f$ is convex iff for all pairs $(x,y)\in \E\times E$ and  $\forall \lambda\in[0,1]$, 
$f(\lambda x+(1-\lambda)y)\leqslant \lambda f(x)+(1-\lambda)f(y)$. 
If the inequality is strict, then $f$ is strictly convex.
\end{definition}

\begin{proposition}\label{PropExistence}
Let $f$ be a convex, proper, l.s.c and coercive function defined on $\E$ with values in $\bar{\R}$, then 
$f$ is bounded from below and admits at least one minimizer. If $f$ is strictly convex then the minimizer is unique.
\end{proposition}

\begin{proof}
     For all $r\in\R$, the set $H_r=\{x\in\E \text{ such that  }f(x)\leqslant r\}$ 
is closed since  $f$ is l.s.c and bounded since $f$ is coercive. Then for all $r\in\R$,
the set  $H_r$ is compact.
The set  $\displaystyle H=\bigcap_{r\in\R}H_r$ is an intersection of nested compact sets. As $f$ is proper, $H=\emptyset$  and there exists $r_0\in\R$ such that
$H_{r_0}$ is empty. As a consequence, $r_0$ is a lower bound of~$f$.

Next, as  $f$ is proper, there exists $x_0\in\E$ such that $f(x_0)<+\infty$ and $f(x_0)\neq-\infty$. The set 
$H=\{x\in\E \text{ such that }f(x)\leqslant f(x_0)\}$ is closed since  $f$ is l.s.c and bounded as $f$ is coercive.
$H$ is thus compact as the dimension of $\E$ is finite.
We have $\underset{x\in\E}{\inf}f(x)=\underset{x\in H}{\inf}f(x)>-\infty$ since  $f$ is lower bounded. Let $(x_n)_{n\in\N}$ be a minimizing sequence of elements of 
$H$ such that  $\underset{n\to\infty}{\lim}f(x_n)=\underset{x\in H}{\inf}f(x)$. This sequence admits a subsequence converging to an element $x^{\infty}$ of $H$. Since $f$ is l.s.c, we get  on a 
$$\underset{n\to \infty}{\lim\inf}f(x_n)\geqslant f(x^{\infty}).$$
As  $\underset{n\to\infty}{\lim}f(x_n)=\underset{x\in H}{\inf}f(x)$, we deduce that  
$f(x^{\infty})=\underset{x\in H}{\inf}f(x)=\underset{x\in \E}{\inf}f(x)$.
The uniqueness of the minimizer finally comes with the strict convexity. 
\end{proof}

Using the notion of convexity, we can also define {\it strong and weak convexity}. Strong convexity allows better local control of the function to minimize.
\begin{definition}[$\alpha$-strong  and weak convexity]\label{def:strong_weak_convex}
Let $f$ be a function defined from $\E$ to $\R\cup\{+\infty\}$ and $\alpha>0$. We say that the function $f$ is $\alpha-$strongly convex, or $\alpha-$convex (resp. $\alpha-$weakly convex) if the function $g$ defined by $g(x)=f(x)-\frac{\alpha}{2}\norm{x}^2$ (resp. $g(x)=f(x)+\frac{\alpha}{2}\norm{x}^2$) is convex.
\end{definition} 
By definition, if $f$ is convex and $y$ is an element of $E$, the function $x\mapsto f(x)+\frac{1}{2\gamma}\norm{x-y}^2$ is $\frac{1}{\gamma}$-strongly convex. Strong convexity involves strict convexity and ensures the uniqueness of the  minimizer.

\subsection{$L$-smoothness}\label{sec:diff}
In this document, we consider the minimization of smooth and non-smooth functions. We will refer to a smooth function as soon as it is differentiable.

\begin{definition}[Differentiability]
A function $f$ defined from $\E$ to $\R\cup{+\infty}$, is  differentiable at point $x\in \E$ if there exists a unique point $l_x\in \E$ such that $$\lim_{h\in\E, ||h||\rightarrow 0} \frac{f(x+h)-f(x)-\langle l_x,h\rangle}{||h||}=0.$$
The gradient of $f$ at point $x$ is then denoted as $\nabla f(x)=l_x$. 

A function is said differentiable as soon as it is differentiable $\forall x\in\E$.
A function $f$ is called {\bf non-smooth} when it is not a differentiable function.
\end{definition}

\begin{definition}[Lipschitz continuity]\label{def:lip}
A function $T$ defined from $\E$ to $\E$ is said to be $L-$Lipschitz (or $L-$Lipschitz continuous) if for all $(x,y)\in \E^2$, we have
\begin{equation*}
\norm{T(x)-T(y)}\leqslant L\norm{x-y}.
\end{equation*} 
\end{definition}

\begin{definition}[$L$-smoothness]
A smooth  function $f$ with $L$-Lipschitz gradient $\nabla f$ is called a $L$-smooth function. 
\end{definition}

A classical theoretical framework ensuring $L-$smoothness is the case of twice differentiable functions $f$. If we are able to bound the operator norm of the Hessian matrix of $f$ by $L$, then we deduce that the gradient of $f$ is $L-$Lipschitz (see Remark~\ref{rem:Lsmooth} in the next section).

\section{Optimization of smooth convex functions}\label{sec:smooth_opt}
We consider the  general problem:
\begin{equation}\label{eq1}
\min_{x\in E}f(x)
\end{equation} 
where $f$ is a real valued, convex, coercive, $L$-smooth function defined on a closed, convex set $\E$.  As $f$ is differentiable, it is lower semi-continuous and proper by definition. Following Proposition~\ref{PropExistence}, these hypotheses ensure that the optimization problem \eqref{eq1} admits at least one solution. Additionally, if $f$ is strictly convex, then the solution is unique.

The minimizers $x^*$ of a convex, differentiable function $f$ defined on $\E$ are simply characterized by the Euler equation:
\begin{equation}\label{eq1_opt}
\nabla f(x^*)=0.
\end{equation} 
A different way to formulate this is that all the critical points of a convex function are global minima.
To estimate such critical points, we will rely on the gradient descent algorithm:
\begin{equation}\label{algo:gd0}
x_{n+1}=x_n-\gamma \nabla f(x_n).
\end{equation} 
This algorithm   consists in applying recursively from a point $x_0\in \E$ the gradient descent operator $\id-\gamma\nabla f$, defined for a strictly positive real number $\gamma$ called {\it stepsize}.\\

The organization of this section is as follows. We first give general properties on the gradient of differentiable convex functions $f$ in section~\ref{sec:nabla}. Then we study in section~\ref{sec:gd_conv} the properties of the gradient descent operator $\id-\gamma\nabla f$ and exploit this operator to find fixed points of~\eqref{algo:gd0} satisfying the optimality condition~\eqref{eq1_opt} and thus minimizing problem~\eqref{eq1}. We provide an analysis of the convergence rate of the gradient descent scheme in section~\ref{sec:convergence_rate} and finally discuss several variants and accelerations of this algorithm in section~\ref{sec:other_gd}.

\subsection{Properties of gradient of smooth functions}\label{sec:nabla}
We here recall important properties related to differentiable and $L$-smooth functions.
\begin{proposition}\label{prop:affine}
Let $f$ be a convex, differentiable function defined from $\E$ to $\R$. Then for all $(x,y)\in\E^2$, we have
\begin{equation} \label{eq:low_bound_affine}
f(y)\geqslant f(x)+\ps{\nabla f(x)}{y-x}
\end{equation}
\end{proposition}
\begin{proof}
By convexity of $f$, we have for all $\lambda\in]0,1[$,
\begin{equation*}
f(x+\lambda(y-x))\leqslant (1-\lambda)f(x)+\lambda f(y).
\end{equation*}
which can be rewritten
\begin{equation*}
\frac{f(x+\lambda(y-x))-f(x)}{\lambda}\leqslant f(y)-f(x).
\end{equation*}
By letting $\lambda$ go to zero we obtain
\begin{equation*}
\ps{\nabla f(x)}{y-x}\leqslant f(y)-f(x)
\end{equation*}
\end{proof}

By applying this proposition to two different points $x$ and $y$, we deduce that the gradient of $f$ is monotone.
\begin{definition}[Monotonicity]

Let $T$ be a function defined from  $\E$ to  $\E$. We say that $T$ is {\it monotone} if for all $(x,y)\in\E^2$,
\begin{equation*}
\ps{T(x)-T(y)}{x-y}\geqslant 0.
\end{equation*}
\end{definition} 

\begin{corollaire}\label{cor:grad_mono}

Let $f$ be a convex, differentiable function defined from $\E$ to $\R$. Then the gradient of $f$ is monotone, for all $(x,y)\in\E^2$:
\begin{equation*}
\ps{\nabla f(x)-\nabla f(y)}{x-y}\geqslant 0.
\end{equation*}
\end{corollaire} 
\begin{proof}
The result is given by summing the inequality in  Proposition~\ref{prop:affine} taken in $(x,y)$ and $(y,x)$.
\end{proof}
\begin{remark} \label{rem:bauschke}
Note that Corollary~\ref{cor:grad_mono} is in fact an equivalence \cite[Proposition 17.10]{BauschkeCombettes}. As the monoticity result is directly obtained by the inequality in  Proposition~\ref{prop:affine}, we get that, for $f : E \to \R$ differentiable, the following points are equivalent:
\begin{itemize}
    \item[(i)] $f$ is convex.
    \item[(ii)]$\forall (x,y)\in\E^2$, we have $f(y)\geqslant f(x)+\ps{\nabla f(x)}{y-x}$
    \item[(iii)] $\nabla f$ is monotone.
\end{itemize}
Moreover, if $f$ is twice-differentiable, this is also equivalent to \cite[Proposition 17.10]{BauschkeCombettes}
\begin{itemize}
\item[(iv)]  for all $(x,y)\in\E^2$, we have $\langle \nabla^2 f(x) y, y \rangle \geq 0$
\end{itemize}
\end{remark}
We just saw that convex and differentiable functions are bounded from below by their affine approximations. In addition, if the gradient of $f$ is $L$-Lipschitz, we can obtain an upper bound of $f$ regardless of its convexity.
\begin{lemme}[Descent Lemma]\label{LemmeMajQuad}
Let $f$ be a $L$-smooth function (i.e. differentiable with $L$-Lipschitz gradient) defined from $\E$ to $\E$. Then for all $(y,z)\in\E^2$, we have
\begin{equation}\label{eqGradLip}
f(z)\leqslant f(y)+\ps{\nabla f(y)}{z-y}+\frac{L}{2}\norm{z-y}^2.
\end{equation} 
\end{lemme} 
\begin{proof}
Let $g$ be a differentiable function on $\R$ such that $g'$ is $K-$Lipschitz (see Definition~\ref{def:lip}). Then we have
\begin{equation*}
g(1)= g(0)+\int_0^1g'(t)dt=g(0)+g'(0)+\int_0^1(g'(t)-g'(0))dt\leqslant g(0)+g'(0)+\frac{K}{2}.
\end{equation*}
Let $(y,z)\in\E^2$. For all  $t\in[0,1]$, we set $v=z-y$, $y_t=y+t(z-y)$ and $g(t)=f(y_t)$. The function $g$ is differentiable and $g'(t)=\ps{\nabla f(y_t)}{v}$. According to the hypotheses on $f$, $g'$ is $K$-Lipschitz with $K=L\norm{v}^2$.\\ From this we deduce that for all $(y,z)\in\E^2$,
\begin{equation*}
f(z)\leqslant f(y)+\ps{\nabla f(y)}{z-y}+\frac{L}{2}\norm{z-y}^2.
\end{equation*} 
\end{proof}
\begin{remark} \label{rem:Lsmooth}
Using Remark \ref{rem:bauschke}, we have that the descent lemma equation~\eqref{eqGradLip} is equivalent to $x \to \frac{L}{2}\norm{x}^2 - f$ being convex. And if $f$ is twice differentiable, this is also equivalent to: for all $(x,y)\in\E^2$, $\langle \nabla^2f(x) y, y \rangle \leq L $.
\end{remark}
Next we recall a useful inequality for convex and $L$-smooth functions.
\begin{lemme}[Co-coercivity]\label{coco}
Let $f$ be a convex $L$-smooth function, then 
\be\label{eq:coco}
\forall (x,y)\in\E^2,\quad \ps{\nabla f(x)-\nabla f(y)}{x-y}\geqslant \frac{1}{L}\norm{\nabla f(x)-\nabla f(y)}^2.
\ee
This property is called co-coercivity of the function $\nabla f$.
\end{lemme}
\begin{proof}
Using inequality~\eqref{eqGradLip} with $z=x-\dfrac{1}{L}\nabla f(x)$ and $y=x$, we obtain
$$\frac{1}{2L}\norm{\nabla f(x)}^2\leqslant f(x)-f\left(x-\frac{1}{L}\nabla f(x)\right)\leqslant f(x)-f(x^*).$$
where $x^*$ is a minimizer of $f$. We also have, for all $x\in\E$,
\be\label{eqnablaf}
\frac{1}{2L}\norm{\nabla f(x)}^2\leqslant f(x)-f(x^*).
\ee
Let $(x,y)\in\E^2$ and introduce the functions $h_1(z)=f(z)-\ps{\nabla f(x)}{z}$ and $h_2(z)=f(z)-\ps{\nabla f(y)}{z}$. These two functions are convex and  respectively admit the minimizers $x$ and $y$. Applying inequality~\eqref{eqnablaf} to these two functions and obtain
$$h_1(x)\leqslant h_1(y)-\frac{1}{2L}\norm{\nabla h_1(y)}^2\text{ and }h_2(y)\leqslant h_2(x)-\frac{1}{2L}\norm{\nabla h_2(y)}^2.$$
By adding these two inequalities we obtain the result of the lemma.
\end{proof}
Let us finally mention that we can obtain a lower bound for $f$ when the function $f$ is also $\alpha$-strongly (or weakly) convex, 

\begin{lemme}\label{LemmeMinQuadSC}
Let $\alpha \in \R$, let $f$ be a smooth and $\alpha$-convex function defined from $\E$ to $\E$. Then for all $(y,z)\in\E^2$, we have
\begin{equation}\label{eqIneqSC}
f(y)\geq f(x)+\langle\nabla f(x),y-x\rangle+\frac{\alpha}{2}\norm{y-x}^2
\end{equation} 
so that 
\begin{equation}\label{eq:strong}\langle\nabla f(x)-\nabla f(y),y-x\rangle\geq \alpha\norm{y-x}^2.\end{equation}
\end{lemme} 

\begin{proof}
The function $f(x)-\frac{\alpha}{2}\norm{x}^2$ is smooth and convex, from the Definition~\ref{def:strong_weak_convex} of $\alpha$-strong convexity. Using Proposition~\ref{prop:affine}, we have for all $(x,y)\in\E^2$
\begin{equation}
f(y)-\frac{\alpha}{2}\norm{y}^2\geqslant f(x)-\frac{\alpha}{2}\norm{x}^2+\ps{\nabla f(x)-\alpha x}{y-x},
\end{equation}
which gives~\eqref{eqIneqSC}.
Switching the roles of  $x$ and $y$ in~\eqref{eqIneqSC} and summing both relations, we obtain~\eqref{eq:strong}.

\end{proof}

\subsection{Gradient descent algorithm}\label{sec:gd_conv}
In this section we describe the most classical method using  the gradient of a convex and  differentiable function $f$, to approach a solution of problem~\eqref{eq1}:
\begin{equation}\label{eq1bis}
\min_{x\in \E}f(x).
\end{equation}

\subsubsection{Gradient descent operator}\label{sec:gd}
In order to compute minimizers of problem~\eqref{eq1bis}, we now introduce the gradient descent operator $\id-\gamma \nabla f$, defined for  a strictly positive real number $\gamma$ called {\it stepsize}.  In this section, we analyze the properties of this operator using the characteristics of $L$-smooth functions studied in section~\ref{sec:nabla}.
First of all, notice that the gradient descent operator allows to  characterize minimizers of~\eqref{eq1}.
\begin{proposition}\label{PropGDArgmin}
Let $f$ be a convex, proper and differentiable function defined in $\E$, then for all $\gamma\neq 0$: 
$$\text{Fix}\, (\id-\gamma \nabla f)=\argmin f.$$
\end{proposition}
\begin{proof}
$x=(\id-\gamma \nabla f)(x)\Leftrightarrow \nabla f(x)=0
\Leftrightarrow x\in\argmin f$.
\end{proof}

The idea behind gradient descent algorithms is  to find fixed points of the gradient descent operator, as these fixed points correspond to minimizers of $f$.
In that perspective, it is interesting to note that for a suitable choice of the stepsize $\gamma$, the gradient descent operator enjoys non expansive properties.
\begin{definition}[Non expansive operators and contractions]
We recall that a function $T$ defined from $\E$ to $\E$ is said to be $L-$Lipschitz (or $L-$Lipschitz continuous) if for all $(x,y)\in \E^2$, we have
\begin{equation*}
\norm{T(x)-T(y)}\leqslant L\norm{x-y}.
\end{equation*} 
If $T$ is 1-Lipschitz we say that $T$ is non expansive and if $T$ is $L-$Lipschitz with $L<1$ we say that $T$ is $L$-contracting or just a contraction.
\end{definition}

\begin{proposition}\label{LemmeGradLip}
Let $f$ be a convex, differentiable function with a $L-$Lipschitz gradient. For $\gamma>0$ we note $T=\id-\gamma \nabla f$. If $\gamma\leqslant \frac{2}{L}$, the operator $T$ is 1-Lipschitz.
\end{proposition} 

\begin{proof}
Taking $(x,y)\in\E^2$ and using Lemma~\ref{coco} on the co-coercivity of $\nabla f$, we get 
\begin{align*}
&\;||x-\gamma\nabla f(x)-(y-\gamma\nabla f(y))||^2\\=&\;||x-y||^2+\gamma^2 ||\nabla f(x)-\nabla f(y)||^2-2\gamma\langle x-y,\nabla f(x)-\nabla f(y)\rangle\\
\leq&\; ||x-y||^2+\gamma \left(\gamma-\frac{2}L\right)||\nabla f(x)-\nabla f(y)||^2.
\end{align*}
We deduce that $\id-\nabla f$ is $1$-Lipschitz for $\gamma\leq \frac2L$.
\end{proof}

\subsubsection{Gradient descent with fixed stepsize}\label{sec:gd_standard}
The gradient descent method with fixed stepsize consists in applying recursively the gradient descent operator $\id-\gamma\nabla f$. Considering a strictly positive real number $\gamma$ called {\it stepsize} and an element $x_0\in \E$, the gradient descent algorithm writes $\forall n\in\mathbb{N}$:
\begin{equation}\label{algo:gd}
x_{n+1}=x_n-\gamma \nabla f(x_n).
\end{equation}
If we note $T$ the operator defined from $\E$ to $\E$ by
$$Tx=x-\gamma \nabla f(x)=(\id-\gamma \nabla f)(x),$$
the sequence $(x_n)_{n\in\N}$ is simply defined by $x_{n+1}=Tx_n$.

If no specific hypothesis is made on the gradient of $f$, the method~\eqref{algo:gd} may not converge to a minimizer of $f$. The algorithm may indeed oscillate and diverge if $\gamma$ is chosen too large. On the other hand, if the gradient of $f$ is  $L$-Lipschitz, we now show that the algorithm~\eqref{algo:gd} converges to a minimizer of $f$ for a correct choice of~$\gamma$.

\begin{theoreme}[Gradient descent algorithm]\label{TheoGradPasFixe}
Let $f$ be a convex, differentiable function with a $L-$Lipschitz gradient, admitting a minimizer. If $\gamma<\frac{2}{L}$ then for all $x_0\in \E$, the sequence defined  $\forall n\in\N$ by $x_{n+1}=(\id-\gamma \nabla f)x_n$ converges to a minimizer of $f$.
\end{theoreme}
\noindent There exist several proofs for this convergence result. The one we propose is based on the nonexpansiveness of the operator $T$ and the following  lemma. Such a sketch of proof  will be useful to prove the convergence of other algorithms in this document. 
\begin{lemme}\label{LemmeConvergence}
Let $T$ be a 1-Lipschitz operator defined from $\E$ to $\E$ admitting a fixed point. Let $x_0\in \E$ and let $(x_n)_{n\in\N}$ be the sequenced defined by $x_{n+1}=Tx_n$. If $\lim_{n\to\infty}\norm{x_{n+1}-x_n}=0$ then the sequence $(x_n)_{n\in\N}$ converges to a fixed point of $T$.
\end{lemme}
\begin{proof}
Let $y$ be a fixed point of $T$. As $T$ is 1-Lipschitz, 
$$ \norm{x_{n+1}-y} =  \norm{T(x_n)-T(y)} \leq \norm{x_{n}-y}, $$
the sequence $(\norm{x_n-y})_{n\in\N}$ is decreasing and therefore bounded. The sequence $(x_n)_{n\in\N}$ is thus bounded.
As $\E$ is of finite dimension, the sequence $(x_n)_{n\in\N}$ admits an adherent point (or closure point) $z\in \E$. As  ${\lim_{n\to\infty}\norm{x_{n+1}-x_n}=0}$, this adherent point satisfies $Tz=z$ and it is  a fixed point of $T$.
The sequence $\norm{z-x_n}$ is thus decreasing and admits a subsequence that goes to zero. Hence the sequence converges to zero. 
\end{proof}

\begin{proof}[Proof of Theorem~\ref{TheoGradPasFixe}]
We just have to show that the operator $T=\id-\gamma \nabla f$ satisfies all the hypotheses of Lemma~\ref{LemmeConvergence}. The fact that $T$ is 1-Lipschitz is the result of Proposition~\ref{LemmeGradLip}. Furthermore, we may note from Proposition~\ref{PropGDArgmin} that there is an equivalence between being a fixed point of $T$ and being a minimizer of $f$, as $Tx=x$ is equivalent to $\nabla f(x)=0$.
Theorem \ref{TheoGradPasFixe} assumes the existence of such a minimizer and thus the existence of a fixed point of $T$. To finish the proof, it is sufficient to show that $\lim_{n\to \infty}\norm{x_{n+1}-x_n}=0$. This result is obtained by applying the inequality~\eqref{eqGradLip} satisfied by $L$-smooth functions (Lemma~\ref{LemmeMajQuad}) to the points $y=x_n$ and $z=x_{n+1}$. We then have $z-y=\gamma \nabla f(x_n)$ and thus
\begin{equation}
f(x_{n+1})\leqslant f(x_n)-\frac{1}{\gamma}\norm{x_{n+1}-x_n}^2+\frac{L}{2}\norm{x_{n+1}-x_n}^2.
\end{equation}
If $\gamma<\frac{2}{L}$, we get
\begin{equation} \label{eq:SDC_GD}
f(x_{n+1})+\left(\frac{2-\gamma L}{2\gamma}\right)\norm{x_{n+1}-x_n}^2\leqslant f(x_n).
\end{equation}
As the sequence $(f(x_n))_{n\in\N}$ is bounded from below, we deduce that $\lim_{n\to \infty}\norm{x_{n+1}-x_n}=0$, which concludes the proof of the theorem. 
\end{proof}

The gradient descent method with fixed stepsize is a {\it descent method}, i.e. the value of $f(x_n)$ decreases. In the next section, we will see that it is possible to control the rate at which $f(x_n)-f(x^*)$ decreases towards zero, where $x^*$ refers to an arbitrary minimizer of $f$.

\subsubsection{Strongly convex function}

If the function $f$ is $L$-smooth and $\alpha$-strongly convex, the  following result shows that the gradient descent operator is a contraction for $\tau<\frac1L$, 
so that Banach's Theorem directly ensures the convergence of the gradient descent method~\eqref{algo:gd} to a fixed point of the gradient descent operator, which is a global minimizer of problem~\eqref{eq1bis}.

\begin{proposition}\label{prop:contratSC}
Let $f$ be a $\alpha$-strongly  convex function. If $f$ is $L$-smooth, then $\id-\gamma \nabla f$ is a $\sqrt{1-\gamma \alpha}$-Lipschitz application for $\tau<\frac1L$.
\end{proposition}
\begin{proof}
We here rely on Lemmas~\ref{LemmeMajQuad} and~\ref{LemmeMinQuadSC}.
If $f$ is $\alpha$-strongly convex, then relation~\eqref{eqIneqSC} states that $\forall x,y\in\E^2$ 
\begin{equation}\label{eq:tmp1}f(y)\geq f(x)+\langle\nabla f(x),y-x\rangle+\frac{\alpha}{2}\norm{y-x}^2.\end{equation}
If $f$ is $L$-smooth, we also have from relation~\eqref{eqGradLip} that 
$$f(y)\leq f(x)+\langle\nabla f(x),y- x\rangle+\frac{L}{2}\norm{y-x}^2,$$
so that $\alpha<L$. 
Combining the co-coercivity property~\eqref{eq:coco} with the strong convexity proper~\eqref{eq:strong}, we get:

\begin{equation}\label{eq:tempo}\begin{split}
&||x-y||^2-||(\id-\gamma \nabla f)(x)-(\id-\gamma\nabla f)(y)||^2\\=\;&2\gamma\langle  \nabla f(x)-\nabla f(y),x-y\rangle-\gamma^2\norm{\nabla f(x)-\nabla f(y)}^2\\=\;&\gamma\hspace{-1pt}\left(\hspace{-1pt}\langle  \nabla f(x)\hspace{-1pt}-\hspace{-1.5pt}\nabla f(y),x\hspace{-1pt}-\hspace{-1pt}y\rangle\hspace{-1pt}-\hspace{-1pt}\gamma\norm{\nabla f(x)\hspace{-1pt}-\hspace{-1pt}\nabla f(y)}^2\hspace{-1.5pt}+\hspace{-1.5pt}\langle  \nabla f(x)\hspace{-1.5pt}-\hspace{-1pt}\nabla f(y),x\hspace{-1pt}-\hspace{-1pt}y\rangle\hspace{-1pt}\right)\\
\geq\;&\gamma\left(\frac1{L}-\gamma\right)\norm{\nabla f(x)-\nabla f(y)}^2 +\gamma\alpha\norm{y-x}^2.
\end{split}\end{equation} 
We conclude that for all $\gamma<\frac1L$:
$$||(\id-\gamma \nabla f)(x)-(\id-\gamma \nabla f)(y)||^2\leq (1-\gamma\alpha) ||x-y||^2,$$
so that $\id-\gamma \nabla f$ is a $\sqrt{1-\gamma \alpha}$-Lipschitz application.
One can also check that since  $\alpha<L$, it implies that $\gamma \alpha<1$ for $\gamma<\frac1L$.
\end{proof}

\subsection{Convergence rates of the gradient descent algorithm}\label{sec:convergence_rate}
The decay rate on $f$ of the gradient descent depends on the hypothesis made on $f$. If $f$ is convex and differentiable, the decay is $O(\frac{1}{n})$. If $f$ is strongly convex or even if $f$ satisfies a quadratic growth condition, then the decay may be exponential. We give here two theorems based on the inequality~\eqref{eq:SDC_GD} with $\gamma=\frac{1}{L}$ proved in section \ref{sec:gd}:
\begin{equation} \label{eq:SDC_GD2}
f(x_{n+1})+\frac{L}{2}\norm{x_{n+1}-x_n}^2\leqslant f(x_n),
\end{equation}
that is satisfied by the sequence defined by $x_{n+1}=x_n-\frac{1}{L} \nabla f(x_n)$. 
The first theorem provides a decay of $O(\frac{1}{n})$ for differentiable convex functions. 
\begin{theoreme}\label{th:gdrate}
If $f$ is convex differentiable and $\nabla f$ is $L-$Lipschitz, having at least one minimizer $x^*$. The sequence $x_n$ defined by $x_0\in E $ and 
\begin{equation}
    x_{n+1}=x_n-\frac{1}{L}\nabla f(x_n)
\end{equation}
satisfies 
\begin{equation}
f(x_n)-f(x^*)\leqslant \frac{L}{2n}\norm{x_0-x^*}^2.
\end{equation}
\end{theoreme}
\begin{proof}
The proof uses a Lyapunov analysis, i.e. a sequence $S_n$ is defined such that it is bounded along the trajectory. As it is usually the case for Lyapunov analysis, the sequence we define is non-negative and non-increasing: 
\begin{equation}\label{def:En}
    S_n:=n(f(x_n)-f(x^*))+\frac{L}{2}\norm{x_n-x^*}_2^2.
\end{equation}
Notice that in practice, $x^*$ and $f(x^*)$ may be unknown and the Lyapunov sequence can not be  computed. We first show that there exists $n_0$ such that $(S_n)_{n\geqslant n_0}$ is non-increasing :
\begin{equation*}
    S_{n+1}-S_n=(n+1)(f(x_{n+1})-f(x_n))+f(x_n)-f(x^*)
    +\frac{L}{2}\norm{x_{n+1}-x^*}_2^2-\frac{L}{2}\norm{x_n-x^*}_2^2.
\end{equation*}
Using \eqref{eq:SDC_GD2} we get 
\begin{equation*}
    S_{n+1}-S_n\leqslant- (n+1)\frac{L}{2}\norm{x_{n+1}-x_n}^2+f(x_n)-f(x^*)
    +\frac{L}{2}\ps{x_{n+1}-x_{n}}{x_{n+1}+x_n-2x^*} 
\end{equation*}
and then 
\begin{equation*}
    S_{n+1}-S_n\leqslant- n\frac{L}{2}\norm{x_{n+1}-x_n}^2+f(x_n)-f(x^*)
    +L\ps{x_{n+1}-x_{n}}{x_n-x^*}.
\end{equation*}
Using the fact that $x_{n+1}-x_n=-\frac{1}{L}\nabla f(x_n)$, we get
\begin{equation*}
    S_{n+1}-S_n\leqslant- n\frac{L}{2}\norm{x_{n+1}-x_n}^2+f(x_n)-f(x^*)
    -\ps{\nabla f(x_n)}{x_n-x^*}.
\end{equation*}
'Using the convexity of $f$ we finally obtain that $S_{n+1}-S_n\leqslant 0$. Then we deduce that for all $n\in\N$, $S_n\leqslant S_0$ and thus 
\begin{equation}
f(x_n)-f(x^*)\leqslant \frac{L}{2n}\norm{x_0-x^*}^2.
\end{equation}
\end{proof}
If additional assumptions, such as strong convexity, are made on $f$, a better convergence can be achieved. To prove this, we need the following lemma. 
\begin{lemme}\label{lem:strongconv}
If $f$ is $\alpha$-strongly convex and differentiable, we have for all $x\in \E$
\begin{equation}
f(x)-f(x^*)\leqslant \frac{1}{2\alpha}\norm{\nabla f(x)}^2, 
\end{equation}
where $x^*$ is the unique minimizer of $f$.
\end{lemme}
\begin{proof}
For any $x\in\E$, we define the function 
\begin{equation*}
    \phi_{x}(y):=f(x)+\ps{\nabla f(x)}{y-x}+\frac{\alpha}{2}\norm{y-x}^2. 
\end{equation*}
Lemma \ref{LemmeMinQuadSC} ensures that for any $(x,y)\in \E^2$, 
\begin{equation*}
    f(y)\geqslant \phi_{x}(y).
\end{equation*}
For any $x$, $\phi_x$ is a quadratic function whose minimizer is $x-\frac{1}{\alpha}\nabla f(x)$. Taking the minimum value in the above inequality, we get 
\begin{equation}
    f(x^*)\geqslant \phi_x(x-\frac{1}{\alpha}\nabla f(x))=
    f(x)-\frac{1}{2\alpha}\norm{\nabla f(x)}^2,
\end{equation}
which ends the proof of the Lemma.
\end{proof}
With this lemma we can now state the following Theorem.
\begin{theoreme}\label{thm:rate_sc}
If $f$ is a $L$-smooth and 
$\alpha$-strongly convex function, then the sequence generated by the gradient descend with $\gamma=\frac{1}{L}$ satisfies 
\begin{equation}\label{rate_sc}
    f(x_n)-f(x^*)\leqslant \left(1-\frac{\alpha}{L}\right)^n(f(x_0)-f(x^*)).
\end{equation}
\end{theoreme}
\begin{proof}
The inequality \eqref{eq:SDC_GD2} can be written 
\begin{equation}
    f(x_{n+1})-f(x^*)\leqslant f(x_n)-f(x^*)-\frac{1}{2L}\norm{\nabla f(x_n)}^2.
\end{equation}
Using Lemma \ref{lem:strongconv} we have 
\begin{equation}
    f(x_{n+1})-f(x^*)\leqslant f(x_n)-f(x^*)-\frac{\alpha}{L}(f(x_n)-f(x^*)).
\end{equation}
which concludes the proof of the Theorem.
\end{proof}
Notice that the strong convexity assumption does not improve the convergence rate already given by  the so-called quadratic growth condition \eqref{eq:SDC_GD2}.%

\subsection{Other gradient descent algorithms}\label{sec:other_gd}
We now describe some variants and accelerations of the gradient descent algorithm.
\subsubsection{Gradient descent with optimal stepsize}
The gradient descent with optimal stepsize consists in adapting the descent stepsize $\gamma$ at each iteration in order to have a maximal decrease of the value of the functional $f$. We define a sequence for $x_0\in\E$ and for all $n\in\N$:
\begin{equation}
x_{n+1}=x_n-\gamma_n\nabla f(x_n)\text{ with }
\gamma_n=\arg\min_{\gamma>0}f(x_n-\gamma \nabla f(x_n)) 
\end{equation} 
This method converges to the unique minimizer of $f$ if $f$ is $\alpha-$strongly convex.
%
%
It has the advantage of converging faster than the gradient descent method with fixed stepsize in terms of number of iterations, but each iteration requires the resolution of an optimization problem that is not simple to solve in general. A close form expression of the optimal parameter $\gamma$ is available in few cases, including quadratic functions~$f$.

\subsubsection{Newton's method}
Newton's method was originally intended for approching a regular zero of a function $F$ of class $C^2$ defined from $\E$ to $\E$. We say that $x$ is a {\it regular zero} of $F$ if
\begin{equation}
F(x)=0\text{ and }F'(x) \text{ is invertible.}
\end{equation}
Newton's method is based on a first order Taylor development in the neighbourhood of a zero of $F$: $F(x_{n+1})=F(x_n)+\langle F'(x_n),x_{n+1}-x_n\rangle$. In order to find $x_{n+1}$ that is a zero of $F$, it builds,  from an element $x_0\in \E$, a sequence defined for all $n\in\N$ by
\begin{equation}\label{eqNewton}
x_{n+1}=x_n-(F'(x_n))^{-1}F(x_n).
\end{equation}
In order for this sequence to be defined  the matrix $F'(x_n)$ has to be invertible. This is guaranteed as soon as the departure point $x_0$ is close enough to the zero of $F$.

In the case of the minimization of a smooth function $f$
\be\label{eq:pbconv}\min_x f(x),
\ee
the Newton's method corresponds to a gradient descent algorithm with an optimized (non scalar) step.
Recalling that an optimum $x^*$ of problem~\eqref{eq:pbconv} is characterized by $\nabla f(x^*)=0$, Newton's method can  be applied to find a zero of $\nabla f$.  Assuming that $f$ is  $\mathcal{C}^2$, the algorithm writes $$x_{n+1}=x_{n}-H_f^{-1}(x_n)\nabla f(x_n),$$ where $H_f(x_n)$ is the Hessian matrix of $f$ taken at point $x_n$.
If $f$ is convex, such a sequence is guarantee to converge to a global minima of $f$.

\subsubsection{Gradient descent with backtracking}
We showed in Theorem~\ref{TheoGradPasFixe} that the gradient descent algorithm $x_{n+1}=x_n-\gamma \nabla f(x_n)$ applied to the minimization of a convex $L$-smooth function $f$ converges for $\gamma<2/L$. This global upper bound may be suboptimal locally: there exists  points $x_n$ for which  larger stepsizes $\gamma$ can be safely considered. Moreover, the constant $L$ may be unknown in real use cases. The idea of backtracking is  to consider a gradient descent with a potentially too large stepsize, and to reduce it until a sufficient  decrease condition is met.
The decrease condition is built from relation~\eqref{eqGradLip} that writes, for $x=x_n$ and $y=x_n-\gamma\nabla f(x_n)$:
\be\label{bt}f(x_n)-f\left(x_n-\gamma\nabla f(x_n)\right)\geq \gamma\left(1-\frac{\gamma L}{2}\right)||\nabla f(x_n)|| ^2.\ee

Given a parameter a large stepsize $\gamma^0>2/L$ and a parameter $\beta\in]0;1[$, backtracking consists in performing,  at each point $x_n$ a gradient descent with stepsize $\gamma^0$ and to update it with $\gamma^{\ell+1}=\gamma^\ell\beta$ until
$$ f(x_n)-f\left(x_n-\gamma^\ell\nabla f(x_n)\right)>\frac{\gamma^\ell}2||\nabla f(x_n)||^2.$$ 
In the worst case, we obtain $\gamma^\ell<2/L$ and backtracking gets back to standard gradient descent.
This simple approach generally performs well. When the constant $L$ is large, gradient descent relies on a small parameter $\gamma<2/L$. With backtracking,  the number of iterations $n$ required by $x_n$ to reach the neighborhood of a global minimizer $x^*$ may be drastically reduced.  Backtracking may also be considered to optimize smooth functions whose gradients are not Lipschitz. In this case, the choice of the initial point is essential.

\subsubsection{Implicit gradient descent with fixed stepsize}
The implicit gradient descent with fixed stepsize can be expressed  as the explicit method, but it raises a practical problem of solving an implicit problem at each step.

We consider a strictly positive real number $\gamma$ called {\it stepsize} and an element $x_0\in \E$. For all $n\in\mathbb{N}$ we define:
\begin{equation}
x_{n+1}=x_n-\gamma \nabla f(x_{n+1}).
\end{equation}
Without  hypothesis on $f$ this relation does not guarantee the existence and uniqueness of $x_{n+1}$, and the resolution of this implicit equation can be practically complicated.
The convexity of $f$ ensures the existence and uniqueness of $x_{n+1}$, as we may note that $x_{n+1}$ is the unique minimizer of a strongly convex and coercive function:
\begin{equation}
x_{n+1}=\arg\min_{x\in \E}f(x)+\frac{1}{2}\norm{x_n-x}^2
\end{equation}
Thus $x_{n+1}$ is the unique element of $\E$ such that $x_{n+1}+\gamma \nabla f(x_{n+1})=x_n$, which allows us to write $x_{n+1}=(\id+\gamma \nabla f)^{-1}(x_n)$, becoming $x_{n+1}=Tx_n$ by introducing $T=(\id+\gamma \nabla f)^{-1}$.

We will later see that this second definition of $x_{n+1}$s that does not take into account the gradient at all, can be generalized to any convex, proper, lower semi-continuous function without making any hypothesis of differentiability.\\
 
We will show that the implicit gradient descent method constructs a minimizing sequence, regardless of the choice of $x_0$ and $\gamma>0$. This method has the advantage over the explicit method of not having a condition on $\gamma$, but the disadvantage of requiring the resolution of a possibly difficult problem at each iteration.

\begin{theoreme}\label{TheoGradImplicite}
Let $f$ be a convex, differentiable function, admitting a minimizer. Then for all $x_0\in \E$, the sequence defined for all $n\in\N$ by $x_{n+1}=(\id+\gamma \nabla f)^{-1}x_n$ converges to a minimizer of $f$.
\end{theoreme}
We do not prove this result as it is a particular case of more general algorithms, called Proximal Point and Forward-Backward algorithms, that will be treated in  sections~\ref{ssec:ppa} and~\ref{ssec:fb}.

\subsubsection{Projected gradient}
We now consider the following constrained optimization problem:
\begin{equation}\label{eq1cons}
\min_{x\in C}f(x),
\end{equation}
where $K$ is a closed, convex set and $f$ is differentiable with an $L-$Lipschitz gradient. 
In this case, we characterize the minimizers in the following way.
\begin{theoreme}\label{thm:grad_proj}
Let $f$ be a convex and smooth function defined from $\E$ to $\R$ and $C\subset\E$ a closed, convex set. Then $x^*$ is a solution to~\eqref{eq1cons} if and only if for all $y\in C$,
\begin{equation}\label{eqCondOpCont}
\ps{\nabla f(x^*)}{y-x^*}\geqslant 0.
\end{equation}
\end{theoreme} 

\begin{proof}
$\Rightarrow$ Let $(x^*,y)\in C^2$. For all $t\in[0;1]$, $x^*+t(y-x^*)\in C$ as $C$ is convex. If $x^*$ minimizes  $f$ on $C$, then 
\begin{align*}
f (x^* + t(y-x^*))-f (x^*) &\geq  0    \\
\lim_{t\to 0}\frac{f (x^* + t(y-x^*))-f (x^*) }t &\geq  0 \\
\ps{\nabla f(x^*)}{y-x^*}&\geq 0.
\end{align*}
$\Leftarrow$ From Proposition~\ref{prop:affine}, we know that for a convex function
$$f(y)\geq f(x^*)+\ps{\nabla f(x^*)}{y-x^*}.$$
Assuming that $\ps{\nabla f(x^*)}{y-x^*}\geq 0$, we deduce that for all $y\in C$, 
$$f(y)\geq f(x^*)+\ps{\nabla f(x^*)}{y-x^*}\geq f(x^*),$$
which implies that $x^*\in K$ solves problem~\eqref{eq1cons}.
\end{proof}

The result below states that it is possible to adapt the gradient descent method with fixed step~\eqref{algo:gd} to solve the constrained problem~\eqref{eq1cons}.
\begin{theoreme}
Let $f$ be a convex, differentiable function such that $\nabla f$ is $L-$Lipschitz. Let $K$ be a closed, convex set and $\gamma<\frac{2}{L}$. Then for all $x_0\in\E$, the sequence defined~by
\begin{equation}
x_{n+1}=\proj_{K}(x_n-\gamma \nabla f(x_n))
\end{equation}
converges to a solution of problem~\eqref{eq1cons}.
\end{theoreme}  

We do not prove this result for two reasons. First,  it is a simple adaptation of the convergence proof for the gradient descent method with fixed step. Second, the projected gradient method is a particular case of the more general Forward-Backward algorithm, that will be treated in section~\ref{ssec:fb}.

\section{Optimization of non smooth convex functions}\label{sec:nonsmooth_opt}
In this part we consider non smooth convex optimization problems, i.e. the function that we want to minimize is non-differentiable. A particular case that we will consider is the case of the sum of two convex functions where at least one (let say $g$ is non-differentiable:
\begin{equation}\label{eq:pb}
\min_{x\in\E}f(x)+g(x).
\end{equation}
The problem of minimizing a differentiable function $f$ under convex constraints~\eqref{eq1cons} can be rewritten in this form, by using the indicator function of the convex set $\K$:
\begin{equation}
\min_{x\in C}f(x)\Longleftrightarrow \min_{x\in\E}f(x)+i_{K}(x),
\end{equation}    
where $$i_{K}(x)=\left\{\begin{array}{ll}0&\textrm{if }x\in K\\
+\infty&\textrm{otherwise.}\end{array}\right.$$
However, this is not the only situation we will consider. For example, in image processing, it is  common to minimize functionals of the form of a sum of two terms; a {\em data fidelity} part and a {\em regularization} part, the latter being chosen to promote certain characteristics to the image to be reconstructed. Such a regularity prior is  often non-differentiable, as in the case of an $\ell_1$ regularization, that promotes sparsity, or the total variation regularization, that promotes piecewise constant images.

In this section, we will introduce  {\it proximal splitting algorithms}, that are designed to solve problem~\eqref{eq:pb}. 
We first recall the key concept of subdifferentials $\partial f$ of convex functions $f$ in section~\ref{sec:sousdiff}.  In section~\ref{sec:prox}, we introduce and study the properties of the proximal operator $(\id+\gamma \partial f)^{-1}$, which can be seen as the implicit equivalent of the explicit gradient descent operator. In section~\ref{sec:prox_algo}, we finally present different proximal splitting algorithms to solve problem~\eqref{eq:pb}. Such algorithms minimize the sum of (non-differentiable) functions by alternating elementary explicit $(\id+\gamma\nabla f)$ and/or implicit $(\id+\partial f)^{-1}$ gradient operations on each function separately. In particular, we will analyze the Forward-Backward algorithm and its acceleration, the Alternating Direction Method of Multipliers and the {Douglas-Rachford} algorithm.
\subsection{Properties of subdifferentials}\label{sec:sousdiff}

\begin{definition}[Subdifferential] \label{def:subdif}
Let $f$ be a function from $E$ to $\R\cup \{+\infty\}$. The subdifferential of $f$ is the multivalued operator that to each $x\in E$ associates the set of slopes
\begin{equation}\label{eq:sousdiff}\partial f(x)=\{u\in\E\text{ such that }\forall y\in\E,\,\,\ps{y-x}{u}+f(x)\leqslant f(y)\}.\end{equation}
\end{definition}
\begin{remark}
We recall the following properties related to subdiffrentials:
\bi
\item An element $p\in \partial f$ of the subdifferential is called a {\bf subgradient}.
\item The subdifferential of a function at a point $x$ can be either empty or a convex set, possibly reduced to a singleton.
\item If $f$ is convex and differentiable in each point $x \in E$, the subdifferential is reduced to the singleton $\{\nabla f(x)\}$, the gradient of $f$ in $x$.
\item The subdifferential of the function $x\mapsto -x^2$ is empty in each point of $E$, even though $f$ is differentiable in each point. In this example $f$ is not convex.
\item It can be shown that in each point $x$ of the relative interior of the domain of a {\bf convex} function $f$, the subdifferential $\partial f (x)$ is non-empty.
\item On the boundary of the domain things are a somewhat more complicated. The function $f$ defined from $\R$ to $\R\cap +\infty$ by
$$f(x)=\left\{
\begin{array}{cc}
-\sqrt{x}&\text{ si }x\geqslant 0\\
+\infty& \text{ si }x<0
\end{array}
\right.$$ 
is convex, takes a finite value in 0 but does not admit a subdifferential in zero.
\ei
\end{remark}
The importance of the subdifferential is made clear through Fermat's rule, as defined by the following theorem.
\begin{theoreme} \label{thm:fermat}
Let $f$ be a proper function defined from $E$ to $\R\cup{+\infty}$. Then
$$\argmin f=\text{zer }\partial f=\{x\in\E\text{ such that }0\in\partial f(x)\}.$$
\end{theoreme}
\begin{proof}\vspace{-.2cm}
$$x\in\argmin f\Leftrightarrow \forall y\in\E,\,\ps{y-x}{0}+f(x)\leqslant f(y)\Leftrightarrow 0\in\partial f(x).$$
\end{proof}
In order to minimize a functional we thus have to find a zero of its subdifferential. If by lower semi-continuity and coercivity, we know that a functional admits at least one minimum, then we know that in this minimum, the subdifferential is non-empty and contains 0.
This rule is an extension of the well known fact that in the global minimum of a differentiable function, the gradient is zero.

We end this paragraph by some rules on the sum of subdifferentials.
\begin{lemme}[Proof in Rockafellar~\cite{RockOnTheMax}]\label{LSSD}
Let $(f_i)_{i\leqslant p}$ be a family of proper, lower semi-continuous functions and let $x\in\E$. Then
$$\partial (\sum_{i=1}^pf_i)(x)\supset \sum_{i=1}^p\partial f_i(x).$$
Additionally, if
$$\bigcap_{1\leqslant i\leqslant p}\operatorname{inter}(\operatorname{dom}(f_i))\neq \emptyset$$
then we have equality between the subdifferential of the sum and the sum of subdifferentials.
\end{lemme}
The hypothesis ensuring equality is not very constraining. In the majority of interesting practical cases, it is satisified.
\begin{remark}
If $J$ is the sum of a convex differentiable function $f$ and a convex function $g$, then under the hypotheses of Lemma \ref{LSSD},
$$\partial J=\nabla f+\partial g.$$
\end{remark}
\begin{remark}
The constrained optimization problem~\eqref{eq1cons} can be reformulated as 
$$\min_x f(x)+\iota_C(x),$$
where $f$ is a smooth and convex function and $\iota_C$ is the characteristic function of a close convex set $C$, that is to say $\iota_C(x)=0$ if $x\in C$ and $+\infty$ otherwise. The optimality condition thus gives $$0\in \partial (f+\iota_C)(x^*)\Leftrightarrow -\nabla f(x^*)\in\partial \iota_C(x^*).$$
We deduce from the subdifferential definition~\eqref{eq:sousdiff} that  $\langle \nabla f(x^*),x-x^*\rangle \geq 0$\, $\forall x\in C$, which corresponds to the characterization of minimizers shown in Theorem~\ref{thm:grad_proj}.
\end{remark}
\subsection{Proximal operator}\label{sec:prox}
In order to minimize non smooth convex functions, we now define the proximal (or proximity) operator of a convex function.

\begin{definition}[Proximal operator]
Let $g$ be a convex,  lower semi-continuous and proper function from $\E$ to $\R\cup +\infty$. The proximity operator of $g$, denoted $\prox_g$ and also called "$\prox$ of $g$" is the operator defined from $\E$ to $\E$ by
\begin{equation}\label{eq:prox}
\prox_g(x)=\underset{z\in\E}{\arg\min}\ g(z)+\frac{1}{2}\norm{x-z}_2^2
\end{equation}
\end{definition}
\begin{remark} The function $z\mapsto g(z)+\frac{1}{2}\norm{x-z}_2^2$ that defines the $\prox$ is convex, proper, lower semi-continuous and coercive, and thus admits at least one minimizer. The uniqueness of the minimizer comes from the strict convexity of the function $y\mapsto \frac{1}{2}\norm{x-z}^2$.
\end{remark}

The proximal operator can be interpreted as a generalization of the concept of projecting onto a convex set in the sense that if $f$ is the indicator of a closed convex set $C$ then $\prox_f(x)$ is the projection of $x$ onto $C$.
From now on we will say that {\bf a function $f$ is simple when its proximity operator can be computed easily}.\\

If $f$ is not convex, the proximity operator may not be defined or it may be multivalued. The convexity of $f$ is a sufficient condition for the existence and uniqueness of the proximity operator, but it is not a necessary condition.

\begin{example}
We here present the proximal operator of functions  used in many problems coming from statistics and image processing.
\begin{itemize}
\item Separable functions. If a function $f$ is separable, the proximity operator may be computed by blocks or even component by component.
Hence if $x=(x_1,x_2)$ and if $f(x)=f_1(x_1)+f_2(x_2)$ then 
\begin{align*}
\prox_{f}(x)&=\underset{z=(z_1,z_2)\in\mathbb{R}^n}{\arg\min}f_1(z_1)+f_2(z_2)+\frac{1}{2}\norm{z-x}_2^2\\&=\underset{z=(z_1,z_2)\in\mathbb{R}^n}{\arg\min}f_1(z_1)
+\frac{1}{2}\norm{z_1-x_1}_2^2+f_2(z_2)+\frac{1}{2}\norm{z_2-x_2}_2^2 
\end{align*}
it follows that 
\begin{equation*}\label{Pprox}
\prox_{f}(x)=(\prox_{f_1}(x_1),\prox_{f_2}(x_2))
\end{equation*}
This remark also applies to a sum of $M$ functions $f(x)=\sum_{i=1}^Mf_i(x_i)$, the proximity operator of $f$ at the point $x$ is a vector whose components $p_i$ are $\prox_{f_i}(x_i)$. A classical application of this rule is the proximal operator of $f(x)=\lambda\norm{x}_1$ which is given by the component-wise "soft thresholding" operation with threshold $\lambda$:
$$\prox_{f}(x)=\left\{\begin{matrix}
0& \text{if }|x|\leq\lambda\\
x-\lambda&\text{if }x>\lambda\\
x+\lambda&\text{if }x<-\lambda
\end{matrix}\right.$$
    \item Quadratic functions. Standard data fidelity terms in imaging write 
$f(x)=\frac{\lambda}{2}\norm{Ax-b}_2^2$. There is an explicit form for the proximal operator of such functions:
\begin{equation}
\prox_{f}(x)=(\id+\lambda A^*A)^{-1}(x+\lambda A^*b).
\end{equation}
Thus the computation of this proximity operator amounts to invert of a system of linear equations. In any cases the matrix $\id+\lambda A^*A$ is real and symmetric  and a conjugate gradient may be used to solve such a problem. It can also happen that the matrix $A$ is not available, and we may only be able to apply $A$ and $A^*$. It also exist situations in imaging where the linear problem can be solved exactly:
\begin{itemize}
\item If $A$ is a circular convolution by a filter $h$. In the Fourier domain $(\id+\lambda A^*A)^{-1}$ is a simple multiplication by $(1+\lambda|\hat h(k)|^2)^{-1}$.
\item If $A$ is a derivative operator like a discrete gradient, then if the derivative is circular, that may create strange boundary effect, it actually is a convolution and one may apply the previous formula. 
If the derivative is not shift invariant, the inversion in the Fourier domain is only an approximation and a conjugate gradient may provide more precise results.
\item If $A$ is a masking operator the operator $(\id+\lambda A^*A)$ is diagonal in the domain and the inversion of the system is a simple pointwise division.
\end{itemize}
\item Indicator functions of vectorial subspaces
Let $A$ a linear operator defined from $\R^n$ to $\R^d$ and $K=\{(x_1,x_2)\text{ such that }Ax_1=x_2\}$ a vectorial subspace of $\R^{n+d}$. The proximity operator of the indicator function $\iota_K$ of $K$ is defined by 
\begin{equation}
(p_1,p_2)=\prox_{\iota_K}(x_1,x_2)=\underset{(z_1,z_2)}{\arg\min }\frac{1}{2}\norm{z_1-x_1}_2^2+\frac{1}{2}\norm{z_2-x_2}_2^2+\iota_{K}(z_1,z_2)
\end{equation}
It follows that $p_2=Ap_1$ with 
\begin{equation}\label{EqProj}
p_1=\underset{z_1}{\arg\min }\frac{1}{2}\norm{z_1-x_1}_2^2+\frac{1}{2}\norm{Az_1-x_2}_2^2=(\id+A^*A)^{-1}(x_1+A^*x_2)
\end{equation}
Once again, the computation of this proximity operator may be fast or not depending on $A$. For example if $A=\id$ then $p_1=\frac{1}{2}(x_1+x_2)$ and $p_2=p_1$. If we are able to diagonalize the matrix $I+A^*A$, for circular convolutions or circular derivative operators, the computation may be fast. If not, one needs to use an algorithm to solve this linear system, like conjugate gradient.
\end{itemize}
\end{example}

The proximal operator has many remarkable properties that make it useful  when minimizing non-differentiable convex functions.
\begin{proposition}\label{PropProxSD}
Let $f$ be a convex function on $E$. Then for all $(x,p)\in \E^2$, we have
\be\label{charac_sg}
p=\prox_f(x)\Leftrightarrow \forall y\in\E,\,
\ps{y-p}{x-p}+f(p)\leqslant f(y).
\ee
Thus $p=\prox_f(x)$ is the unique vector  $p\in \E$ such that $x-p\in\partial f(p)$.
\end{proposition}
This last characterization implies that the decomposition of  $x=p+z$, as the sum of one element $p$ from $E$ and one element from the subdifferential of $f$ at the point $p$, is unique. This  property is the origin of the notation

$$\prox_f(x)=(\id+\partial f)^{-1}(x)$$ 
to which we can give a sense even though the map $\partial f$ is multivalued.
\begin{proof}

$\Rightarrow$ Suppose that $p=\prox_f(x)$. Let $\alpha\in]0,1[$ and $p_{\alpha}=\alpha y+(1-\alpha)p$. Then, by using the definition of $\prox_f(x)$ and the convexity of $f$, we obtain
\begin{align*}
f(p)&\leqslant f(p_{\alpha})+\frac{1}{2}\norm{x-p_{\alpha}}^2-\frac{1}{2}\norm{x-p}^2\\
&\leqslant \alpha f(y)+(1-\alpha)f(p)-\alpha\ps{y-p}{x-p}
+\frac{\alpha^2}{2}\norm{y-p}^2
\end{align*}
and thus
$$\ps{y-p}{x-p}+f(p)\leqslant f(y)+\frac{\alpha}{2}\norm{y-p}^2.$$
By letting $\alpha$ go to zero we obtain the direct implication.\\
$\Leftarrow$  Conversely, suppose that $\forall y$ we have $\ps{y-p}{x-p}+f(p)\leqslant f(y)$. Then
\begin{align*}
f(p)+\frac{1}{2}\norm{x-p}^2&\leqslant 
f(y)+\frac{1}{2}\norm{x-p}^2+\ps{x-p}{p-y}+\frac{1}{2}\norm{p-y}^2\\&\leqslant 
f(y)+\frac{1}{2}\norm{x-y}^2    
\end{align*}
\noindent
from which we deduce that $p=\prox_f(x)$ by using the definition of $\prox_f(x)$.
\end{proof}

An important property of  fixed points of the proximity operator is the following:
\begin{proposition}\label{PropProxArgmin}
Let $f$ be a proper, convex function defined on $E$. Then
$$\operatorname{Fix}\, \prox_f=\argmin f.$$
\end{proposition}
\begin{proof}

$$x=\prox_f(x)\Leftrightarrow \forall y\in\E,\,\ps{y-x}{x-x}+f(x)\leqslant f(y)
\Leftrightarrow x\in\argmin f.$$
\end{proof}
Thanks to this characterization, one can consider an iterative algorithm involving the proximity operator of $f$ to minimize a convex functional $f$.
\subsubsection{Firm nonexpansiveness}
To establish the convergence of proximal algorithms, we first show that the proximity operator of convex functions is $1$-Lipschitz. To do so, we rely on the notion of firm nonexpansiveness, which implies $1$-Lipschitzity.

\begin{definition}[Firm nonexpansiveness]\label{def:FNE}
A function $T$ is said to be {\it firmly nonexpansive} if for all $(x,y)\in\E^2$,
\begin{equation*}\label{fne}
\norm{T(x)-T(y)}^2+\norm{x-T(x)-y+T(y)}^2\leqslant \norm{x-y}^2.
\end{equation*} 
\end{definition}
A classic example of a firmly nonexpansive function is the case of a projection onto a closed convex set. nonexpansiveness will be a key ingredient to show the convergence of the Douglas-Rachford algorithm. 

\begin{proposition}\label{PropProxLip}
If $f$ is a proper, convex function on $E$ then the maps $\prox_f$ and $\id-\prox_f$ are firmly nonexpansive; that is, for all $(x,y)\in \E^2$,
$$\norm{\prox_f(x)-\prox_f(y)}^2+\norm{(x-\prox_f(x))-(y-\prox_f(y))}^2\leqslant 
\norm{x-y}^2$$ 
In particular, these two operators are nonexpansive (1-Lipschitz).
\end{proposition}
\begin{proof}
Let $(x,y)\in \E^2$. By noting $p=\prox_f(x)$ and $q=\prox_f(y)$, we have, by definition of $p$ and $q$:
$$\ps{q-p}{x-p}+f(p)\leqslant f(q)\text{ and }\ps{p-q}{y-q}+f(q)\leqslant f(p).$$
Adding these two inequalities gives
$$0\leqslant \ps{p-q}{(x-p)-(y-q)}$$
and we conclude by observing that
\begin{align*}
\norm{x-y}^2&=\norm{p-q+(x-p)-(y-q)}^2\\
&=\norm{p-q}^2+\norm{(x-p)-(y-q)}^2+
 2\ps{p-q}{(x-p)-(y-q)}.
\end{align*}
 
\end{proof}

We underline that the gradient descent operator of a smooth convex function, which is $1$-Lipschitz for $\gamma <\frac2{L}$ (see Proposition \ref{LemmeGradLip}), is also nonexpansive with an additional restriction on the stepsize $\gamma$.
\begin{proposition}\label{PropGradLip}
Let $f$ be a differentiable, convex function defined on $\E$ with an $L-$Lipschitz gradient. If $\gamma<\dfrac{1}{L}$, then the map $\id-\gamma \nabla f$ is firmly nonexpansive.
\end{proposition}
\begin{proof}
The proof is based on determining the sign of a scalar product. 
Let $(x,y)\in\E^2$. 
We set  $u=\gamma(\nabla f(x)-\nabla f(y))$ and $v=x-\gamma\nabla f(x)-(y-\gamma\nabla f(y))$. Hence we have 
$$\norm{x-y}^2=\norm{u+v}^2=\norm{u}^2+\norm{v}^2+2\ps{u}{v}.
$$
Showing that $\id-\gamma \nabla f$ is firmly nonexpansive amounts to demonstrate that $\ps{u}{v}$ is non negative. The result directly comes from the co-coercivity of $\nabla f$ shown in Lemma~\ref{coco} and the fact that  $\gamma<\frac{1}{L}$:
\begin{align*}\ps{u}{v}&=\gamma \ps{x-y}{\nabla f(x)-\nabla f(y)}-\gamma^2||\nabla f(x)-\nabla f(y)||^2\\
&\geq \gamma\left(\frac1L-\gamma\right) ||\nabla f(x)-\nabla f(y)||^2.\end{align*}
\end{proof}

\subsubsection{Strongly convex functions}
The proximal operator of a strongly convex function is a contraction.
\begin{proposition}\label{prop:prox_strong}
Let $f$ be a $\alpha$-strongly  convex function, then  $\prox_{\gamma f}$ is a 
$\frac{1}{1+\alpha\gamma}$-lipschitz application.
\end{proposition}

\begin{proof}
We recall that a $\alpha$-strongly  convex function $f$ satisfies for all $x$, $y$ and $\alpha\in[0;1]$: 
$$f(\alpha x+(1-\alpha) y)\leq \alpha f(x)+(1-\alpha) f(y)-\frac{\alpha(1-\alpha)\alpha}{2}||x-y|||^2.$$
Let $p=\prox_{\gamma f}(x)$, $\alpha\in(0,1)$ and denote
$p_{\alpha}=\alpha z+(1-\alpha)p$ for some $z\in\mathbb{R}^n$.
From the definition~\eqref{eq:prox} of $\prox_f(x)$ and the $\alpha$-strong  convexity of $f$, we have:
\begin{align*}
&f(p)\\
\leq &f(p_{\alpha})+\frac{1}{2\gamma}\norm{x-p_{\alpha}}^2-\frac{1}{2\gamma}\norm{x-p}^2\\
\leq &  \alpha f(z)+(1-\alpha)f(p-\frac{\alpha(1-\alpha)\alpha}2||z-p||^2\hspace{-0.5pt}-\hspace{-0.5pt}\frac\alpha\gamma\langle z-p,x-p\rangle
\hspace{-0.5pt}+\hspace{-0.5pt}\frac{\alpha^2}{2\gamma}\norm{z-p}^2,
\end{align*}
and thus
$$\frac1\gamma\langle z-p,x-p\rangle+f(p\leq f(z)+\frac{1}{2}\left(\frac{\alpha}\gamma -(1-\alpha)\alpha\right)\norm{z-p}^2.
$$
For $\alpha\to 0$, we deduce that
\begin{align*}
\frac1\gamma\langle z-p,x-p\rangle +f(p)&\leq f(z)-\frac\alpha{2}\norm{z-p}^2.
\end{align*}
Taking $z=q:=\prox_{\gamma f}(y)$ we get:
$$\frac1\gamma\langle q-p,x-p\rangle +f(p)\leq f(q)-\frac\alpha{2}\norm{q-p}^2$$
Switching the roles of $x$ and $y$ and summing both expressions we obtain:
\begin{equation}\label{eq:strongcarac}\langle p-q,(x-p)-(y-q\rangle \geq \gamma \alpha \norm{q-p}^2,\end{equation}
and we get the co-coercivity property :
\begin{equation}\label{eq:strongcarac2}\begin{split}
\langle p-q,x-y\rangle &\geq (1+\gamma \alpha) \norm{p-q}^2\\
\end{split}\end{equation}
From \eqref{eq:strongcarac2}, and recalling that $p=\prox_f(x)$ and $q=\prox_f(y)$, we get that $\prox_{\gamma f}$ is $\frac1{1+\gamma\alpha}$-Lipschitz  and thus a contraction. 
\end{proof}
\newpage

\subsection{Proximal algorithms}\label{sec:prox_algo}
We now present proximal optimization algorithms to solve nonsmooth convex minimization problems.

\subsubsection{Proximal Point algorithm}\label{ssec:ppa}
In order to  solve the problem
\begin{equation*}
\min_{x\in\E}g(x),
\end{equation*}
for a (nonsmooth) proper convex function $g$,
the proximal point algorithm applies recursively the proximity operator of the function $\gamma g$.
Choosing $\gamma>0$ and $x_0\in E$, this algorithms defines the sequence $(x_n)_{n\in\N}$ as
\begin{equation}\label{algo_prox}
x_{n+1}=Tx_n=\prox_{\gamma g}(x_n).
\end{equation}
The fixed points of the proximity operator are minimizers of $g$ (Proposition~\ref{PropProxArgmin})  and this operator is 1-Lipschitz (Proposition~\ref{PropProxLip}). By construction,~\eqref{algo_prox} gives that 
\begin{equation*}
g(x_{n+1})+\frac{1}{2\gamma}\norm{x_n-x_{n+1}}^2\leqslant g(x_n),
\end{equation*} 
which implies that the sequence  $\norm{x_n-x_{n+1}}^2$ goes to zero.
Applying Lemma \ref{LemmeConvergence}, we  get that the sequence generated by the algorithm~\eqref{algo_prox} converges to a minimizer of~$g$.

In practice, this algorithm is barely used, in the sense that computing the proximity operator may be just as difficult as minimizing $g$ directly (by taking $\gamma\to 0$). In what follows, we will rather use the proximity operator as an ingredient in more sophisticated {\em splitting} algorithms.

\subsubsection{Forward-Backward algorithm (FB)}\label{ssec:fb}
The Forward-Backward, also known as proximal Gradient Descent (PGD) is an algorithm designed to solve the following optimization problem:
\begin{equation}\label{pb:FB}
\min_{x\in\E}\mathcal{J}(x)=\min_{x\in\E}f(x)+g(x)
\end{equation}
where $f$ is a convex, differentiable function with an $L-$Lipschitz gradient and $g$ is a convex function. This algorithm consists in alternating an explicit gradient descent step on $f$ and a proximity operator on $g$.

The Forward-Backward algorithm is defined by an initial point $x_0\in \E$ and a parameter $\gamma>0$ as follows:
\begin{equation}\label{algo:FB}
\forall n\in\N\quad x_{n+1}=Tx_n=\prox_{\gamma g}(x_n-\gamma \nabla f(x_n))
\end{equation}

The algorithm results from the following proposition:
\begin{proposition}
Let $\mathcal{J}=f+g$ be a functional defined from $\E$ to $\R\cup+\infty$ having the form of a sum of two convex, proper, lower semi-continuous functions satisfying the hypotheses of Lemma \ref{LSSD} (on the sum of subdifferentials), such that $f$ is differentiable. Let $\gamma>0$, then
\be
\operatorname{zeros}(\partial \mathcal{J})=\operatorname{Fix}(\prox_{\gamma g}(\id-\gamma \nabla f)).
\ee   
\end{proposition}
\begin{proof}
\begin{align*}
0\in \partial \mathcal{J}(x)&\Leftrightarrow 0\in\partial \gamma \mathcal{J}(x)\\ 
&\Leftrightarrow 0\in\nabla \gamma f(x)+\partial \gamma g(x)\\
& \Leftrightarrow -\gamma \nabla f(x)\in \partial \gamma g(x)\\
& \Leftrightarrow x-\gamma\nabla  f(x)\in (\id+\partial \gamma g)(x)\\
& \Leftrightarrow x=\prox_{\gamma g}(x-\gamma \nabla f(x)).
\end{align*}
\end{proof}

Having an operator $T$ whose fixed points are the minimizers of $\mathcal{J}$ does not guarantee that the sequence $(x_n)_{n\in\N}$ converges to one of these minimizers if $T$ is not a contraction (take for instance $x_{n+1}=-x_n$ as a counterexample).

For convex functions $f$ and $g$, the Forward-Backward operator is only nonexpansive, and thus not a contraction, as the composition of $1$-Lipschitz operators. We indeed know from Propositions~\ref{LemmeGradLip} and~\ref{PropProxLip} that $\id-\tau\nabla f$ is $1$-Lipschitz for $\tau<\frac2L$ and $\prox_{\tau g}$ is $1$-Lipschitz for all $\tau>0$.

In the strongly convex case, convergence may directly be ensured.
If $f$ (resp. $g$) is strongly convex, then Proposition~\ref{prop:contratSC} (resp. Proposition~\ref{prop:prox_strong}) states that $\id-\tau\nabla f$ (resp. $\prox_{\tau g}$ is a contraction for $\tau<\frac1L$ (resp. $\tau>0$). Combining a contraction with a $1$-Lipschitz operator preserves the contraction. Hence Banach's Theorem  directly gives the convergence of the sequence generated by the Forward-Backward algorithm in case $f$ or $g$ are strongly convex.\\

To show convergence of the iterates in the general case, we introduce the concept of {\it surrogate functions}. The following lemma shows that at each iteration of the algorithm, the value of the functional $\mathcal{J}$ decreases. Additionally, this decrease ensures that the sequence with the general term $\norm{x_{n+1}-x_n}^2$ is summable and thus goes to zero when $n$ goes to $+\infty$. 

\begin{lemme}\label{lemmasurrogate}
The sequence $(x_n)_{n\in\N}$ defined by $x_{n+1}=\prox_{\gamma g}(x_n-\gamma\nabla f(x_n))$ satisfies the following relation:
\be\label{eqdecJ}
\mathcal{J}(x_{n+1})+\left(\frac{1}{\gamma}-\frac{L}{2}\right)\norm{x_{n+1}-x_n}^2\leqslant \mathcal{J}(x_n).
\ee
\end{lemme}
\noindent
We underline that relation~\eqref{eqdecJ}, ensured by the Forward-Backward algorithm  when minimizing $\mathcal{J}=f+g$, is the extension of the property~\eqref{eq:SDC_GD} satisfied by the iterates of the gradient descent algorithm when minimizing a differentiable function~$f$.

\begin{proof}
By definition of $x_{n+1}$, we have
$$x_{n+1}=\uargmin{x\in\E}\,\gamma g(x)+\frac{1}{2}\norm{x-x_n+\gamma \nabla f(x_n)}^2.$$
Note that $x_{n+1}$ is also the minimizer of the following functional:
\be\label{eqxnp2}
x_{n+1}=\uargmin{x\in\E}\, g(x)+f(x_n)+\ps{\nabla f(x_n)}{x-x_n}+\frac{1}{2\gamma}\norm{x-x_n}^2.
\ee
According to the inequality \eqref{eqGradLip}, we have, for all $x\in E$,
\be\label{eqxnp1}
f(x)\leqslant f(x_n)+\ps{\nabla f(x_n)}{x-x_n}+\frac{L}{2}\norm{x-x_n}^2.
\ee
The right hand-side is a function that bounds $\mathcal{J}$ from above, and it is equal to $\mathcal{J}$ in the point $x_n$. We call this function a {\it surrogate function}.
The function in~\eqref{eqxnp2} being $\frac{1}{\gamma}$-strongly convex, we deduce from \eqref{eqxnp2} that\vspace{-0.1cm}
\begin{align*}
&\;g(x_{n+1})+f(x_n)+\ps{\nabla f(x_n)}{x_{n+1}-x_n}+\frac{1}{2\gamma}\norm{x_{n+1}-x_n}^2\\\leqslant&\; \mathcal{J}(x_n)-\frac{1}{2\gamma}\norm{x_{n+1}-x_n}^2.
\end{align*}
By applying \eqref{eqxnp1} to $x_{n+1}$ and adding the previous inequality we obtain
$$\mathcal{J}(x_{n+1})+\frac{1}{\gamma}\norm{x_{n+1}-x_n}^2\leqslant \mathcal{J}(x_n)+\frac{L}{2}\norm{x_{n+1}-x_n}^2.\vspace*{-0.75cm}$$
\end{proof} 

With all these elements, we can now show the convergence of the Forward-Backward algorithm, also known in the literature as the Proximal Gradient Descent (PGD) algorithm.

\begin{theoreme}[Forward-Backward algorithm]\label{ThFB2}
Let $\mathcal{J}=f+g$ be a sum of two convex, coercive, lower semi-continuous functions that are bounded from below. We suppose that $f$ is differentiable with an $L-$Lipschitz gradient. Let $\gamma<\dfrac{2}{L}$ and $x_0\in\E$, and let $(x_n)_{n\in\N}$ be the sequence defined for all $n\in \N$ by
\be
x_{n+1}=\prox_{\gamma g}(\id-\gamma \nabla f)(x_n).
\ee
Then the sequence $(x_n)_{\in\N}$ converges to a minimizer of $\mathcal{J}$.
\end{theoreme}
\begin{proof}

As for the proof of the gradient descent method, we use Lemma~\ref{LemmeConvergence} to prove the theorem.  We first show that the operator $T$ is $1-$Lipschitz and then demonstrate that the sequence $\norm{x_n-x_{n+1}}$ goes to zero.

The operator $T$ is $1-$Lipschitz as a composition of a proximity operator that is 1-Lipschitz (Proposition~\ref{PropProxLip}) and an operator of the form $\id-\gamma\nabla f$ that is also 1-Lipschitz under the condition that $\gamma\leqslant \frac{2}{L}$ (Proposition~\ref{LemmeGradLip}).
The second point is then given by Lemma~\ref{lemmasurrogate}.
\end{proof}
In the following section we will see that it is possible to control the speed of the functional's decrease towards its minimum.

\subsubsection{Fast Iterative Shrinkage Thresholding Algorithm (FISTA)}
In the early 60' Polyak introduced inertial methods in~\cite{polyak1964some} to improve the convergence rate of gradient descent algorithm for $C^2$ strongly convex functions. The goal of inertial methods is to accelerate the gradient descent algorithm without any second order information on $f$, that is the Hessian of $f$. 

Such an inertial algorithm  uses the memory of the past trajectory at points $x_n$ and $x_{n-1}$ to compute $x_{n+1}$. It is also known as Heavy Ball because the optimization scheme can    be seen as a discretization of an ODE describing the position of a ball submitted to a force field $\nabla f$ with a friction proportional to its speed:
\begin{equation}\label{eq:odehb0}
\ddot{x}(t)+a \dot{x}(t)+\nabla f(x(t))=0.
\end{equation} 
If $f$ is $L$-smooth and $\alpha$-strongly convex, we have seen in Theorem~\ref{thm:rate_sc} (relation~\eqref{rate_sc}) that the sequence $(x_n)_{n\in\N}$ provided by the gradient descent satisfies
\begin{equation*}
f(x_n)-f(x^*)=O(e^{-\frac{\alpha}{L}n}).
\end{equation*}
If we also assume that $f$ is $C^2$, the Heavy Ball method  provides a sequence $(x_n)_{n\in\N}$ ensuring the following rate, that is better if $\frac{\alpha}{L}<<1$, which is often the case for large scale problems:
\begin{equation*}
f(x_n)-f(x^*)=O(e^{-4\sqrt{\frac{\alpha}{L}}n}).\vspace{0.1cm}
\end{equation*}

In this section we focus on two inertial algorithms inspired by the work of Polyak to minimize  composite functions:
\begin{equation}
\min_{x\in\E}\mathcal{J}(x)=\min_{x\in\E}f(x)+g(x).
\end{equation} The first one is FISTA  by Beck and Teboulle~\cite{BeckTeboulle}. This algorithm, based on an original idea by Nesterov~\cite{Nes83}, is dedicated to convex functions. The second one, called V-FISTA~\cite{beck2017first}, can be seen as an adaptation of the original Heavy ball algorithm to composite strongly convex functions.
\paragraph{FISTA for convex functions}
To minimize a convex function $\mathcal{J}=f+g$ where $f$ is convex and  $L-$smooth and $g$ is convex, the FISTA  (Fast Iterative Soft Thresholding Algorithm) algorithm~\cite{BeckTeboulle} writes
\begin{equation}\label{eq:defFISTA0}
\left\{
\begin{array}{lll}
y_n&=x_n+\frac{t_n-1}{t_{n+1}}(x_n-x_{n-1})\\
x_{n+1}&=\prox_{\gamma g}(y_n-\gamma \nabla f(y_n)),
\end{array}
\right.
\end{equation}
where $\gamma\leqslant \frac{1}{L}$, $t_1=1$ and $t_{n+1}=\frac{1+\sqrt{1+4t_n^2}}{2}$.
The following result shows that this sequence generally performs  better than Forward-Backward.
\begin{theoreme}\label{theo:FISTA}
The sequence generated by~\eqref{eq:defFISTA0} satisfies 
\begin{equation}
\mathcal{J}(x_n)-\mathcal{J}(x^*)\leqslant \frac{2}{\gamma(n+1)^2}\norm{x_0-x^*}^2.
\end{equation}
\end{theoreme}

Several remarks can be done before going any further.
\begin{itemize}
\item The convergence rate of FISTA, $O(\frac{1}{n^2})$ is better than the rate $O(\frac{1}{n})$ of Forward-Backward (FB). Numerically it performs better, try it.
\item FISTA should have been called Fast Forward Backward, but the authors prefer a reference to ISTA (Iterative Soft Thresholding Algorithm) which is a particular case of FB to get a better name... 
\item Y. Nesterov proposed the same scheme with $T=Id-\gamma \nabla f$ to minimize a differentiable function in 1983 with the same rate~\cite{Nes83}. For this reason, FISTA is also called a Nesterov's acceleration of FB. 
\item The decay is optimal in the sense that no first order method (using only gradient or subgradient) may provide a rate $O(\frac{1}{n^\delta})$ with $\delta>2$~\cite{Nes83}. Using extra previous iterates $x_k$ with $k<n-1$ does not help to get a better rate.
\item With the original choice of the sequence $(t_n)_{n\in\N^*}$, there is no known proof of the convergence of the sequence $(x_n)_{n\in\N}$ generated by FISTA.
\end{itemize}
The analysis of this algorithm with the  sequence $(t_n)_{n\in\N^*}$ defined in~\eqref{eq:defFISTA0} was  not understood until the works of~\cite{chambolle2015convergence} and~\cite{su2016differential} who proposed a complete  proof of convergence for a modified sequence $(t_n)_{n\in\N^*}$, while noticing that FISTA can be seen as a discretization of the following  ODE 
\begin{equation}\label{eq:odehb}
\ddot{x}(t)+\frac{3}{t} \dot{x}(t)+\nabla f(x(t))=0.
\end{equation} 

More precisely, the proof of Theorem \ref{theo:FISTA} can be obtained by relying on a Lyapunov analysis, by showing that the following sequence
\begin{equation}
S_n=t_n^2(f(x_n)-f(x^*))+\frac{1}{2\gamma}\norm{t_n(x_n-x_{n-1})+2(x_n-x^*)}^2
\end{equation} 
is non-increasing if 
\begin{equation}\label{eq:decFISTA}
t_{n+1}^2-t_{n+1}+t_{n}^2\leqslant 0.
\end{equation}
The choice in~\cite{Nes83,BeckTeboulle} is to maximize $t_n$ at each step, to get the best decay. Other definition of $t_n$ satisfying \eqref{eq:decFISTA} can nevertheless be considered, to ensure both similar asymptotic decay and  weak convergence of the sequence $(x_n)_{n\in\N^*}$. As proposed in~\cite{chambolle2015convergence},  the following version of FISTA
\begin{equation}\label{eq:defFISTA}
\left\{
\begin{array}{lll}
y_n&=x_n+\frac{n}{n+\beta}(x_n-x_{n-1})\\
x_{n+1}&=\prox_{\gamma g}(y_n-\gamma \nabla f(y_n))
\end{array}
\right.
\end{equation}
ensures that 
\begin{equation}
\mathcal{J}(x_n)-\mathcal{J}(x^*)\leqslant \frac{\beta+1}{2\gamma(n+1)^2}\norm{x_0-x^*}^2,
\end{equation}
for $\gamma\leqslant \frac{1}{L}$ and $\beta>3$.
Unfortunately, there is no rule to choose $\beta$. In practice, if no additional assumptions are made on $\mathcal{J}$, i.e. $\mathcal{J}$ is only assumed convex, any value $\beta>3$ seems to be a good choice. When $\mathcal{J}$ is strongly convex, $\beta$ must be chosen depending on $\varepsilon$ if  $\varepsilon$ is the numerical precision we want to reach~\cite{aujol2023fista}.   
\paragraph{V-FISTA for strongly convex functions.}
When the function $\mathcal{J}=f+g$ is $\alpha-$strongly convex and $\alpha$ is known, there exist several ways to improve the convergence reached by the Forward-Backward algorithm. To avoid many complex definitions, we only focus here on the strongly convex hypothesis, but this assumption could be weaken.  The analysis of Polyak in his seminal work~\cite{polyak1964some} deals with $C^2$ functions. The Heavy Ball has been extended to composite functions in~\cite{beck2017first}, with an inertial scheme that can be seen as a discretization of a Heavy Ball ODE~\eqref{eq:odehb}  with a different friction parameter $a$ and a different stepsize $\gamma=\frac{1}{L}$:
\begin{equation}\label{eq:defFISTA2}
\left\{
\begin{array}{lll}
y_n&=&x_n+\frac{\sqrt{L}-\sqrt{\alpha}}{\sqrt{L}+\sqrt{\alpha}}(x_n-x_{n-1})\\
x_{n+1}&=&T(y_n):=\prox_{\gamma g}(y_n-\gamma \nabla f(y_n)).
\end{array}
\right.
\end{equation}
This scheme provides a sequence such that 
\begin{equation}
\mathcal{J}(x_n)-\mathcal{J}(x^*)\leqslant \left(1-\sqrt{\frac{\alpha}{L}}\right)^n(\mathcal{J}(x_0)-\mathcal{J}(x^*)).
\end{equation}
We finally make two remarks on this result.
\begin{itemize}
\item 
The strong convexity parameter $\alpha$ must be known to fix the inertial parameter, which is a difference between FISTA and V-FISTA. 
\item 
If $\frac{\alpha}{L}<<1$, this rate is much better than Forward-Backward.
\end{itemize}
\subsubsection{Douglas-Rachford algorithm (DR)} \label{sec:DR}
We now target  the problem
\begin{equation*} 
\min_{x\in\E}\mathcal{J}(x)=\min_{x\in\E}f(x)+g(x),
\end{equation*}
where the functions $f$ and $g$ are convex, proper, lower semi-continuous and where $\mathcal{J}$ is bounded from below. The difference with the previous setting is that we do not make any assumptions on the differentiablity of $f$. We also assume that we are able to compute the proximity operators of both $f$ and $g$. The first thing to notice is that under these assumptions, there exists a minimizer of $\mathcal{J}$.\\

As before, we can identify an operator that is $1-$Lipschitz and whose fixed points are associated to the minimizers of $\mathcal{J}$. In the case of Forward-Backward, these fixed point were the minimizers of $\mathcal{J}$, whereas for the Douglas-Rachford algorithm, the minimizers are the images of these fixed points through an operator. 
To introduce this operator, we first need to define the {\it reflected proximity operator}.

\begin{definition}\label{def:rprox}
 The {\it reflected prox} or $\rprox$ of  a proper and convex function $f$ is 
$$\rprox_f=2\prox_f-\id.$$
\end{definition}
The Douglas-Rachford algorithm is then based on the following proposition.
\begin{proposition}\label{prop:fixed_DR}
Let $\mathcal{J}=f+g$ be a functional defined from $E$ to $\R\cup+\infty$, where the two functions $f$ and $g$ are both convex, proper, lower semi-continuous and satisfying the hypotheses of Lemma~\ref{LSSD} (on the sum of subdifferentials), and let $\gamma$ be a strictly positive real number. Then we have
\be
\text{zeros }(\partial \mathcal{J})=\prox_{\gamma g}\left(Fix(\rprox_{\gamma f}\rprox_{\gamma g})\right).
\ee
\end{proposition}
\begin{proof}
Under the hypotheses of Lemma \ref{LSSD}, we have
\begin{align*}
0\in \partial \mathcal{J}(x)&\Leftrightarrow 0\in\partial \gamma \mathcal{J}(x)\\ 
&\Leftrightarrow 0\in\partial \gamma f(x)+\partial \gamma g(x)\\
&\Leftrightarrow \exists z\in\E \text{ such that }-z\in\partial \gamma f(x)
\text{ and } z\in\partial \gamma g(x)\\
&\Leftrightarrow \exists y\in\E \text{ such that }x-y\in\partial \gamma f(x)
\text{ and } y-x\in\partial \gamma g(x)
\end{align*}
We rewrite $x-y\in\partial \gamma f(x)$ as  $2x-y\in(\id+\partial\gamma f)(x)$. The relation $y-x\in\partial \gamma g(x)$ can also be rewritten as $y\in (\id+\gamma \partial g)(x)$ and thus $x=\prox_{\gamma g}(y)$, which gives
$$0\in \partial \mathcal{J}(x)\Leftrightarrow \exists y\in\E \text{ such that } 2x-y\in(\id+\partial\gamma f)(x)\text{ and }x=\prox_{\gamma g}(y).$$
By using the definition of the operator $\rprox$ we obtain:\vspace{-0.1cm}
\begin{align*}
&0\in \partial \mathcal{J}(x)\\\Leftrightarrow &\exists y\in\E \text{ such that } x=\prox_{\gamma f}(\rprox_{\gamma g}y)\text{ and }x=\prox_{\gamma g}(y)\\
\Leftrightarrow& \exists y\in\E \text{ such that } y=2x-\rprox_{\gamma g}y=\rprox_{\gamma f}(\rprox_{\gamma g} y)
\text{ and }x=\prox_{\gamma g}(y)\\
\Leftrightarrow& \exists y\in\E \text{ such that } y\in Fix\left(\rprox_{\gamma f}\rprox_{\gamma g}\right)\text{ and }x=\prox_{\gamma g}(y).
\end{align*} 
\end{proof}

From this proposition we can define a minimization algorithm for $\mathcal{J}$ that estimates a fixed point of the operator $T=\rprox_{\gamma f}\rprox_{\gamma g}$ and then applies 
$\prox_{\gamma f}$ to this fixed point to get a minimizer of $\mathcal{J}$. In order to estimate such a fixed point, 
we first give a key property of $\rprox$ operators related to firm nonexpansiveness (see Definition~\ref{def:FNE}).


\begin{lemme}\label{lemma:rprox}
Let $T$ be an operator defined from $E$ to $E$. The following two statements are equivalent:
\bi
\item $T$ is firmly nonexpansive
\item $R=2T-\id$ is nonexpansive (1-Lipschitz)
\ei
\end{lemme}

\begin{proof}
We  show that these two properties are equivalent to the non negativeness of a scalar product. Let $(x,y)\in\E^2$. By setting $u=Tx-Ty$ and $v=Tx-x-(Ty-y)$, we have
$$\norm{x-y}^2=\norm{u-v}^2=\norm{u}^2+\norm{v}^2-2\ps{u}{v}\vspace{0.1cm}$$
thus $T$ is firmly nonexpansive if and only if $\ps{u}{v}\geqslant 0$. Additionally,
$$\norm{Rx-Ry}^2=\norm{u+v}^2=\norm{u}^2+\norm{v}^2+2\ps{u}{v}
=\norm{x-y}^2+4\ps{u}{v}$$
which shows that the nonexpansiveness of $R$  is equivalent to $\ps{u}{v}\leqslant 0$.
\end{proof}
\vspace*{0.3cm}

\begin{corollaire}\label{cor:rprox}
For any convex function $f$ and parameter $\tau>0$, the operator $\rprox_{\tau f}$ is nonexpansive.
\end{corollaire}
\begin{proof}
The result follows from Lemma~\ref{lemma:rprox}, given the definition of the $\rprox$ operator (Definition~\ref{def:rprox}) and the fact that the proximal operator of a convex function is firmly nonexpansive (Proposition~\ref{PropProxLip}). 
\end{proof}

The $\rprox$ operator is thus nonexpansive ($1$-Lipschitz), but this is not sufficient to show the convergence of $x_{n+1}=\rprox_{\gamma f}\rprox_{\gamma g}(x_n)$. Indeed, contrary to gradient descent and Forward-Backward algorithms, we do not have a descent Lemma ensuring that $\lim_{n\to\infty}\norm{x_{n+1}-x_n}=0$ to apply Lemma~\ref{LemmeConvergence}. 
Hence we rather rely on the Krasnosel'skii-Mann Algorithm that considers an averaging process to target fixed-points of nonexpansive operators. 
\begin{theoreme}[Krasnosel'skii-Mann Algorithm]\label{thm:KM}
Let $D$ be a non-empty, closed and convex subset of $E$ and let $T,\,D\to D$ be a 1-Lipschitz operator such that the set of fixed points of $T$ is non-empty. Let $(\lambda_n)_{n\in\N}$ be a sequence of real numbers in $[0,1]$ such that $\sum_{n\in\N}\lambda_n(1-\lambda_n)=+\infty$, and take $x_0\in D$. We define
$$\forall n\in\N,\quad x_{n+1}=x_n+\lambda_n(Tx_n-x_n).$$
\end{theoreme}

\begin{proof}
Let $y$ be a fixed point point of $T$ (i.e. $Ty=y$) and $n\in\N$, we have \begin{align*}
\norm{x_{n+1}-y}^2&=\norm{(1-\lambda_n)(x_n-y)+\lambda_n(Tx_n-y)}^2\\
&=(1\hspace{-.5pt}-\hspace{-.5pt}\lambda_n)\norm{x_n\hspace{-.5pt}-\hspace{-.5pt}y}^2\hspace{-.5pt}+\hspace{-.5pt}\lambda_n\norm{Tx_n\hspace{-.5pt}-\hspace{-.5pt}Ty}^2\hspace{-.5pt}-\hspace{-.5pt}\lambda_n(1\hspace{-.5pt}-\hspace{-.5pt}\lambda_n)\norm{Tx_n\hspace{-.5pt}-\hspace{-.5pt}x_n}^2\\
&\leqslant \norm{x_n-y}^2-\lambda_n(1-\lambda_n)\norm{Tx_n-x_n}^2,
\end{align*} 
where we used the following relation:
\begin{align*}
2\langle x_n-y, Tx_n-y\rangle
= \norm{Tx_n-y}^2- \norm{Tx_n-x_n}^2+\norm{x_n-y}^2,
\end{align*} 
as well as the fact that $T$ is non expansive  for the last inequality.   
We deduce that
\be\label{eqKM}
\sum_{n\in\N}\lambda_n(1-\lambda_n)\norm{Tx_n-x_n}^2\leqslant \norm{x_0-y}^2.
\ee
Also observe that
\begin{align*}
\norm{Tx_{n+1}-x_{n+1}}&=\norm{Tx_{n+1}-Tx_n+(1-\lambda_n)(Tx_n-x_n)}\\
&\leqslant \norm{x_{n+1}-x_n}+(1-\lambda_n)\norm{Tx_n-x_n}\\
&\leqslant \norm{Tx_n-x_n}
\end{align*}
As $\sum_{n\in\N}\lambda_n(1-\lambda_n)=+\infty$ and as the sequence $(\norm{Tx_n-x_n})_{n\in\N}$ is non increasing we deduce from \eqref{eqKM} that $(Tx_n-x_n)_{n\in\N}$ converges to 0.
The sequence $\norm{x_n-y}$ being non increasing, we get that all elements from the sequence $(x_n)_{n\in N}$
belong to a close ball of center $y$ and radius $\norm{y-x_0}$, which is a compact since $\E$ is finite-dimensional. 
Hence we can extract a subsequence $(x_{n_k})_{k\in\N}$ that converges to a point  $x\in \E$.
\newpage

As $Tx_n-x_n$ converges to $0$, the sequence $(Tx_{n_k}-x_{n_k})_{k\in\N}$ goes to $0$ too. We deduce that the sequence $(Tx_{n_k})_{k\in\N}$ also converges to $x$. Since this sequence also converges to $Tx$, we get that  $x=Tx$ is a fixed point of $T$.
Finally, since $x$ is a fixed point of $T$, the sequence
 $(\norm{x_n-x})_{n\in\N}$ is non increasing. We can extract a subsequence that converges to $0$, which implies that $(x_n)_{n\in\N}$ converges to $x$ a fixed point of~$T$. 
\end{proof}
We can now show the convergence of the Douglas-Rachford algorithm.
\begin{theoreme}[Douglas-Rachford algorithm]
Let $f$ and $g$ be two convex, proper, lower semi-continuous functions, bounded from below. Let $(\mu_n)_{n\in\N}$ be a sequence of elements in $[0,2]$ such that $\sum_{n\in\N}\mu_n(2-\mu_n)=+\infty$. Let $\gamma>0$ and $x_0\in \E$. Let $(x_n)_{n\in\N},\,(y_n)_{n\in\N}$ and $(z_n)_{n\in\N}$ be the sequences defined by
\be
\forall n\in\N \left\{\begin{array}{ll}
y_n&=\prox_{\gamma g}(x_n),\\
z_n&=\prox_{\gamma f}(2y_n-x_n),\\
x_{n+1}&=x_n+\mu_n(z_n-y_n),
\end{array}\right.
\ee
which is equivalent to $$x_{n+1}=x_n+\dfrac{\mu_n}{2}(\rprox_{\gamma f}\rprox_{\gamma g}(x_n) -x_n).$$
Then there exists $x\in E$ minimizing $(f+g)$ such that  $(x_n)_{n\in\N}$ converges to $x$.
\end{theoreme}

\begin{proof}
First observe that 
\begin{align*}
    x_{n+1}&=x_n+\mu_n(z_n-y_n)\\
    &=x_n+\mu_n\left(\prox_{\gamma f}(2y_n-x_n)-\prox_{\gamma g}(x_n)\right)\\
    &=x_n+\mu_n\left(\prox_{\gamma f}(\rprox_{\gamma g}(x_n))-\prox_{\gamma g}(x_n)\right)\\
    &=x_n+\frac{\mu_n}2\left(2\prox_{\gamma f}(\rprox_{\gamma g}(x_n))-2\prox_{\gamma g}(x_n)+x_n-x_n\right)\\
    &=x_n+\frac{\mu_n}2\left(\rprox_{\gamma f}(\rprox_{\gamma g}(x_n))-\rprox_{\gamma g}(x_n)-x_n\right)\\
    &=x_n+\frac{\mu_n}2\left(\rprox_{\gamma f}(\rprox_{\gamma g}(x_n))-x_n\right).
\end{align*}

Corollary~\ref{cor:rprox}  states that the $\rprox$ operator is  1-Lipschitz.
We now define the operator $T=\rprox_{\gamma f}\rprox_{\gamma g}$. We know that this operator is 1-Lipschitz as a composition of two operators that are 1-Lipschitz. 
From Proposition~\ref{prop:fixed_DR}, we also know that  the set of fixed points of $T$ is equal to the image by the $1$-Lipschitz operator $\prox_{\gamma f}$ of the set of minimizers of $f+g$, which is non-empty according to the hypotheses on $f$ and $g$. We note that $x_{n+1}=x_n+\dfrac{\mu_n}{2}(Tx_n -x_n)$. We conclude by applying Krasnoselsky-Mann's algorithm presented in Theorem~\ref{thm:KM}.
\end{proof}

\newpage
\paragraph{Other formulations of Douglas-Rachford} The Douglas-Rachford algorithm can be expressed in many forms in the literature. 
For example, it is frequent for the parameters $\mu_n$ to be fixed to $1$. The algorithm is then expressed in the following way:
\be\label{algoDR}
\forall n\in\N \left\{\begin{array}{ll}
y_n&=\prox_{\gamma g}(x_n),\\
z_n&=\prox_{\gamma f}(2y_n-x_n),\\
x_{n+1}&=x_n+z_n-y_n
\end{array}\right.
\ee

By omitting the variable $z_n$ we obtain the following description of the sequences $(x_n)_{n\in\N}$ and $(y_n)_{n\in\N}$:
\be\label{algoDR00}
\forall n\geqslant 1 \left\{\begin{array}{cl}
x_{n+1}=&x_{n}+\prox_{\gamma f}(2y_{n}-x_{n})-y_{n},\\
y_{n+1}=&\prox_{\gamma g}(x_{n+1}).
\end{array}\right.
\ee
We also mention the change of variables $w_n=y_n-x_n$ that appears in the literature:
\be\label{eqDR4}
\forall n\geqslant 1 \left\{\begin{array}{cl}
u_{n+1}=&\prox_{\gamma f}(y_{n}+w_{n}) \\
y_{n+1}=&\prox_{\gamma g}(u_{n+1}-w_{n}),\\
w_{n+1}=&w_{n}+y_{n+1}-u_{n+1}.
\end{array}\right.
\ee 
This algorithm can also be adapted to the case where the minimization problem is of the form:
\begin{equation*}
\min_{x\in\E}f(Ax)+g(x)
\end{equation*}
where $A$ is a linear operator from $\E$ to $F$.
Indeed, one can consider the equivalent problem:
\begin{equation*}
\min_{x,y\in\E\times F}f(y)+g(x)+\iota_{Ax=y}(x,y),
\end{equation*}

\noindent
and apply the Douglas-Rachford algorithm to minimize the sum of the two convex functions functions $f(y)+g(x)$ and $\iota_{Ax=y}(x,y)$, with respect to the augmented  variable  $v=(x,y)$.

\subsubsection{Parallel ProXimal Algorithm (PPXA)}\label{sec:PPXA}
The PPXA algorithm (see Chapter 10 in \cite{CombettesPesquet}) adapts the Douglas-Rachford algorithm  to solve:
\begin{equation}\label{eq:PPXA}
\min_{x\in E}\sum_{i=1}^M f_i(x),
\end{equation}
where the $M$ function $f_i$  are convex proper and l.s.c. 
The PPXA algorithm increases the dimension of the problem and it allows for the parallel computation of the proximal operators $\prox_{f_i}$.
Let us denote as $X$ a vector in  $E^M$
that writes $X=(x_1,x_2,\cdots,x_M)$ with $x_i\in \E$ for all $i\leqslant M$.
PPXA then consists in reformulating the problem~\eqref{eq:PPXA} as an optimization problem in $\E^M$:
\begin{equation}\label{eq:PPXA2}
\min_{X=(x_i)_{i\leqslant n}\in E^M}\sum_{i=1}^M f_i(x_i)+\iota_{D}(X)
\end{equation}
where $D=\{X=(x_n)_{k\leqslant M}\in E^M\text{ such that }x_n=x_1,\,\forall k\leqslant M\}$.

Assuming that the optimization problem~\eqref{eq:PPXA} has at least one solution, a minimizer of~\eqref{eq:PPXA} can be obtained by applying the Douglas-Rachford  algorithm to the problem $\min_{X\in\E^M}\tilde f(X)+\tilde g(X)$ with 
$\tilde f(X)=\sum_{i=1}^M f_i(x_i)$ and $\tilde g(X)=\iota_{D}(X)$.
The proximity operator of $\tilde f$ can be computed in a parallel way using the separability of $\tilde f$, following Example~\eqref{Pprox}.
One can also show that the proximity operator of the indicator of the set $D$ is a simple mean.

We finally provide the following observations about PPXA:
\begin{itemize}
\item Even when only $M=2$ functions are involved, applying PPXA is different from applying Douglas-Rachford. The computation of proximal operators is parallel with PPXA, and it may be faster to use PPXA than Douglas-Rachford  for $M=2$ functions. 
\item If the functions $f_i$ for $i\geqslant 2$ are composed with  linear operators $L_i$:
\begin{equation}
\min_{x\in E}f_1(x)+\sum_{i=2}^M f_i(L_ix)
\end{equation}
a similar reformulation can be considered:
\begin{equation}\label{ppxa2}
\min_{X=(x_i)_{i\leqslant M}\in E^n}\sum_{i=1}^M f_i(x_i)+\iota_{D_2}(X)
\end{equation}
where $D_2=\{X=(x_i)_{i\leqslant M}\in E^M\text{ such that }L_ix_1=x_i,\,\forall i\leqslant M\}$.
To solve problem~\eqref{ppxa2} with the Douglas-Rachford algorithm, one needs to perform the projection on $D_2$. It exists a closed form for this projection, which is similar to \eqref{EqProj}. If we denote 
$(p_1,p_2,\cdots,p_n)=Proj_{D_2}(X)$ then $\forall k\geqslant 2$ we have $p_k=L_kp_1$ and 
\begin{equation*}
p_1=\left(\id+\sum_{i=2}^M L_i^*L_i\right)^{-1}\left(x_1+\sum_{i=2}^M L_i^*x_i\right).
\end{equation*}
To perform  the  inversion in a general case (without assumptions on the operators $L_i$),  a conjugate gradient method can be considered. 

\end{itemize}

\subsubsection{Alternating Direction Method of Multipliers (ADMM)} \label{sec:ADMM}
The ADMM algorithm~\cite{GabayMercier,GlowinskiMarroco} is designed to solve optimization problems of the form:
\begin{equation}\label{eqADMM}
\min_{(x,y)\in\E\times F,\,Ax+By=b}f(x)+g(y)
\end{equation}

where $A$ and $B$ are two linear operators taking values from $\E$ and  $F$ to $G$, $b$ is a vector in $G$ and $f$ and $g$ are two convex, proper, lower semi-continuous functions. It is a general formulation, containing  cases such as $y=x$ or  $y=Ax$. 

The Lagrangian associated to~\eqref{eqADMM} to account for the constraint $Ax+By=b$ with a multiplier variable $z\in G$ writes 
\be
	L(x,y,z)= f(x) + g(y) + \ps{z}{Ax+By-b}, 
\ee
and the augmented Lagrangian for $\gamma > 0$ is 
\be
	L_\gamma(x,y,z) = L(x,y,z)  + \frac{\gamma}{2}\norm{ Ax+By-b}^2.
\ee

Let $(x_0,y_0)\in \E\times F$, $\gamma>0$ and $z^0\in G$. The ADMM introduces the sequences $(x_n)_{n\in\N},\, (y_n)_{n\in\N}$ and $(z_n)_{n\in\N}$ defined  $\forall n\geqslant 0$ as:
\be\label{algo:ADMM}
\left\{\begin{array}{ll}
x_{n+1}&=\uargmin{x}f(x)+\ps{z_{n}}{Ax}+\frac{\gamma}{2}\norm{Ax+By_{n}-b}^2\\
y_{n+1}&=\uargmin{y}g(y)+\ps{z_{n}}{By}+\frac{\gamma}{2}\norm{Ax_{n+1}+By-b}^2\\
z_{n+1}&=z_{n}+\gamma(Ax_{n+1}+By_{n+1}-b)
\end{array}\right.
\ee 

The ADMM is so called because it can be seen as a variant of an algorithm known as the Augmented Lagrangian Method. If we replace the iterative updates of $x$ and $y$ by a joint update step: 
\begin{equation*}
(x_{n+1},y_{n+1})=\arg\min_{(x,y)\in E\times F}f(x)+g(y)+\ps{z_{n}}{Ax+By}+\frac{\gamma}{2}\norm{Ax+B y-b}^2,
\end{equation*}
we obtain the Augmented Lagrangian method that consists in penalizing the constraints with a Lagrange multiplier $z$ and a quadratic term. One of the problems with this method is that joint minimization is generally difficult to perform. The ADMM separates this problem by optimizing with respect to the first variable $x$ first, and then the second variable $y$, while the variable $z$ can be interpreted as a Lagrange multiplier that is updated at each iteration.

\begin{theoreme}[ADMM algorithm]\label{thm:admm}
Let $f$ and $g$ be two proper, convex, coercive, lower semi-continuous functions. Let the sequences $(x_n)_{n\in\N}$, $(x_n)_{n\in\N}$ and $(z_n)_{n\in\N}$ be defined as in the ADMM algorithm~\eqref{algo:ADMM}, then
\begin{enumerate}
\item The sequence $(f(x_n)+g(y_n))_{n\in\N}$ converges to the minimum value of $f+g$.
\item The sequences $(x_n)_{n\in\N}$ and $(y_n)_{n\in\N}$ converge.
\item The sequence $(Ax_n+By_n-b)_{n\in\N}$ goes to zero.
\end{enumerate} 
\end{theoreme}
In the specific case where 
 $A=\id$, $B=-\id$ and $b=0$ (i.e. $x=y$), the ADMM algorithm~\eqref{algo:ADMM} can be expressed using proximity operators:
\be
\label{eq:ADMM_simple}
\left\{\begin{array}{cl}
x_{n+1}&=\prox_{f/\gamma}(y_{n}-z_{n}/\gamma)\\
y_{n+1}&=\prox_{g/\gamma}(x_{n+1}+ z_{n}/\gamma)\\
z_{n+1}&=z_{n}+ \gamma(x_{n+1}-y_{n+1}),
\end{array}\right.
\ee  
which, for $\gamma=1$, is one of the forms of Douglas-Rachford, see \eqref{eqDR4}.

To prove Theorem~\ref{thm:admm} in the general case, it is possible to show that applying  ADMM is equivalent to applying the Douglas-Rachford algorithm on a different problem, called the dual problem of \eqref{eqADMM}, and this independently of the choice of the operators $A$ and $B$.
We will show an equivalence between both algorithms in section~\ref{sec:equivADMMDR}, once  the adequate material on duality has been introduced.

\section{Duality}\label{sec:duality}

In this section we present optimization algorithms that rely on the  dual formulations of convex problems. 
In section~\ref{sec:convconj}, we first introduce the key dual notion of convex conjugate $f^*$ of a function $f$, and detail relevant properties of this transform.
Then, we study primal-dual algorithms in section~\ref{sec:pd}.  In section~\ref{sec:equiv}, we finally show equivalence, in particular settings, between the Chambolle-Pock primal dual algorithm, the Douglas-Rachford algorithm and the ADMM.

\subsection{Properties of convex conjugates}\label{sec:convconj}
There exist close relationships between convex conjugate, subdifferential and proximal operators. These links will be useful to show the equivalence 
between primal and dual problems.
\begin{definition}[Convex conjugate]
Let $f$ be a function defined from 
$\E$ to $\bar{\R}$, the convex (or Fenchel or Legendre-Fenchel) conjugate of $f$ is the function defined from  $\E$ to $\bar{\R}$ by
\be\label{def:conconj}
f^*(u)=\underset{x\in\E}{\sup}(\ps{x}{u}-f(x)).
\ee
\end{definition}
\begin{remark}
The function $\ps{x}{u}-f(x)$ is linear (and thus convex) in $u$. The supremum of convex functions being convex, $f^*$ is necessarily a convex function.
\end{remark}
Conjugation offers a geometric interpretation of convex, proper and l.s.c functions: 
at any point, $f$ is equal to the supremum of all affine functions that are lower bounds of $f$.
\begin{example}
Let us give examples of convex conjugates
\begin{itemize}
\item If $f(x)=\ps{a}{x}-b$ then 
$$f^*(u)=\left\{
\begin{array}{ll}
b&\text{if }u=a\\
+\infty& \text{if }u\neq a.
\end{array}
\right.$$
\item If $f(x)=\norm{x}_p$ where $\norm{\cdot}_p$ is the $\ell_p$ norm, then $f^*$ is the indicator function of the unit ball for the dual norm $\ell_q$ (such that $1/p+1/q=1$), i.e $f^*(u)=i_{\mathcal{B}_{q}}(u)$, for $\mathcal{B}_{q}=\{u,\, \norm{u}_{q}\leq 1$\}. 
Thus, for $f(x)=\norm{x}_1$, we have
$$f^*(u)=\left\{
\begin{array}{ll}
0&\text{if }\norm{u}_{\infty}\leqslant 1\\
+\infty& \text{if }\norm{u}_{\infty}>1.
\end{array}
\right.$$
\end{itemize}
 
\end{example}

\noindent Next lemma gives a main relation between subdifferentials and convex conjugates.
\begin{lemme}\label{LSD}
Let $f$ be a function defined on $\E$ and $x\in\E$ , then
\be
u\in\partial f(x)\Leftrightarrow f^{*}(u)+f(x)=\ps{u}{x}
\ee
\end{lemme}
\begin{proof}
\begin{align*}
u\in\partial f(x)& \Leftrightarrow \forall y\in\E,\; f(y)
\geqslant f(x)+\ps{y-x}{u}.\\
&\Leftrightarrow \forall y\in\E,\;\ps{x}{u}-f(x)
\geqslant \ps{y}{u}-f(y).\\
&\Leftrightarrow \ps{x}{u}-f(x)=f^{*}(u)\\
&\Leftrightarrow f^{*}(u)+f(x)=\ps{x}{u}.
\end{align*}
\end{proof}
One essential property of convex functions is equality to its biconjugate. First we show a property of the biconjugate for a general, nonconvex function, and then prove the equality between $f$ and $f**$ is $f$ is convex.
\begin{lemme}\label{Lbiconj}
Let $f$ be a function defined from $\E$ to $[-\infty,+\infty]$, for all $x\in\E$, $f^{**}(x)\leqslant f(x)$. Moreover, $f^{**}$ is the largest convex lower semi-continuous envelope of the function satisfying $f^{**}(x)\leq f(x)$, for all $x\in\E$
\end{lemme}
\begin{proof}
Let $\Sigma\subset \E \times \R$ the set of pairs $(u,\alpha)$ representing affine functions $x\mapsto \langle u,x\rangle-\alpha$ that are upper bounded by $f$:
\begin{align*}
    (u,\alpha)\in \Sigma&\Leftrightarrow f(x)\geq \langle u,x\rangle-\alpha,\;\forall x\in\E\\
    &\Leftrightarrow \alpha\geq \sup_{x\in\E}\langle u,x\rangle-f(x)\\
    &\Leftrightarrow\alpha\geq f^*(u),\;\textrm{and }u\in \dom(f^*).
\end{align*}
Then we obtain, for all $\in\E$:
\begin{align*}
    \sup_{(u,\alpha)\in\Sigma} \langle u,x\rangle -\alpha&=\sup_{\begin{array}{c}u\in\dom(f^*)\\-\alpha\leq -f^*(u) \end{array}} \langle u,x\rangle -\alpha\\
    &=\sup_{u\in\dom(f^*)} \langle u,x\rangle-f^*(u)
    &=f^{**}(x).
\end{align*}
\end{proof}

\begin{theoreme}[Biconjugation]\label{Tbiconj}
If $f$ and is a proper l.s.c convex function defined from  $\E$ to $[-\infty,+\infty]$, then $f^{**}=f$.
\end{theoreme}
\begin{proof}
    If $f$ is convex proper and l.s.c then Lemma~\ref{Lbiconj} ensures that $f$ is equal to its largest convex lower semi-continuous envelope, that is to say $f^{**}=f$.
\end{proof}

We deduce from this theorem that a convex,  proper and l.s.c function $f$ satisfies  $\forall x\in\E$
\be\label{conj2}
f(x)=\underset{u\in\E}{\text{sup}}(\ps{x}{u}-f^{*}(u))
\ee   
\begin{remark}
Computations in the proof of Lemma~\ref{Lbiconj} show that the supremum in~\eqref{def:conconj} is reached at all points $u\subset \partial f(x)$. Hence, if the subdifferential $\partial f(x)$ is not empty, which is the case in most practical applications, the supremum ({\it sup}) is a maximum ({\it max}) in relation~\eqref{def:conconj}. 
\end{remark}
\begin{remark}
Relation~\eqref{conj2} is useful when $f$ is composed with an operator. For instance, if $K$ is a linear application from 
 $\E$ to an another space $F$, we have 
$\forall x\in E$
\be
f(Kx)=\underset{u\in F}{\text{sup}}(\ps{Kx}{u}-f^{*}(u))=
\underset{u\in F}{\text{sup}}(\ps{x}{K^{*}u}-f^{*}(u)).
\ee
With such an expression, the operator $K$ can be decoupled from the convex function $f$. This is  a key ingredient of the primal-dual algorithms detailed below.
\end{remark}

Let us give a relevant properties between subgradients of conjugate functions.
\begin{proposition}
If $f$ is a convex, proper and l.s.c function, then 
\be\label{eqProxSD}
u\in\partial f(x)\Leftrightarrow x\in \partial f^{*}(u).
\ee  
\end{proposition}
\begin{proof}
Let  $(x,u)\in\E^2$ such that  $u\in \partial f(x)$. Lemma~\ref{LSD} ensures that
$$u\in\partial f(x)\Leftrightarrow f^{*}(u)+f(x)=\ps{u}{x}.$$
As $f$ is convex, proper and l.s.c, we have $f=f^{**}$ so that
$$u\in\partial f(x)\Leftrightarrow f^{*}(u)+f^{**}(x)=\ps{u}{x}$$
which implies $u\in\partial f(x)\Leftrightarrow x\in \partial f^{*}(u)$.
\end{proof}
\noindent
We now give a useful theorem involving convex conjugates and proximal operators.
\begin{corollaire}[Moreau's identity~\cite{Moreau1965}]\label{Moreau}
If $f$ is a convex, proper and l.s.c function, then for all $\gamma>0$
\be\label{eq-moreau-identity}
\prox_{\gamma f}(x)+\gamma \prox_{f^*/\gamma}(x/\gamma) = x.
\ee
that is to say, for $\gamma=1$, $\prox_f+\prox_{f^*}=\id$.
\end{corollaire}
\begin{proof}
Using Proposition~\ref{PropProxSD} and  relation~\eqref{eqProxSD}, we have
\begin{align*}
p=\prox_{\gamma f}(x)&\Leftrightarrow (x-p)\in\gamma \partial f(p)\\
&\Leftrightarrow p \in \partial f^{*}((x-p)/\gamma)\\
&\Leftrightarrow p/\gamma \in \frac1\gamma\partial f^{*}((x-p)/\gamma)\\
&\Leftrightarrow x/\gamma-(x-p)/\gamma \in\partial (f^{*}/\gamma)((x-p)/\gamma)\\
&\Leftrightarrow (x-p)/\gamma =\prox_{f^*/\gamma }(x/\gamma)\\
&\Leftrightarrow \prox_{\gamma f}+\gamma \prox_{f^*/\gamma}(./\gamma)=\id.
\end{align*}   
\end{proof}

We finally give an equivalence between $L$-smoothness and strong convexity of conjugate functions.
\begin{proposition} 
Let $f$ be a convex proper and l.s.c function, then the two following assertions are equivalent: (i) $f$ is $1/L$-strongly ; (ii) $f^*$ is $L$-smooth.
\end{proposition}

\begin{proof}
$(i)\Rightarrow(ii)$: If $f$ is $1/L$- strongly convex, then we can extend relation~\eqref{eq:strong} from Proposition~\eqref{prop:contratSC} to non smooth function using the definition of subdifferential~\eqref{eq:sousdiff}. For all $(x,y)\in\E^2$, $p\in \partial f(x)$  and $q\in\partial f(y)$, we obtain:
$$||p-q||\geq \frac1L||x-y||.\vspace{-0.1cm}$$
Using relation~\eqref{eqProxSD} it implies:
$$||p-q||\geq \frac1L||\partial f^*(p)-\partial f^*(q)||,$$
so that $\partial f^*$ has a $L$-Lipschitz continuous gradient.\\
$(ii)\Rightarrow(i)$: If $f^*$ is $L$-smooth, then the co-coercivity property~\eqref{eq:coco} gives $\forall(p,q)\in\E^2$:
$$\langle \nabla f^*(p)-\nabla f^*(q),p-q\rangle \geq \frac1L||\nabla f^*(p)-\nabla f^*(q)||^2.$$
Denoting as $x=\nabla f^*(p)$ and $y=\nabla f^*(q)$, relation~\eqref{eqProxSD} leads to $p=\partial f(x)$ and  $q=\partial f(y)$. Hence we have\vspace{-0.2cm}  
\begin{align*}   
\langle x-y,\partial f(x)-\partial f(y)\rangle&\geq \frac1L||x-y||^2\\
||\partial f(x)-\partial f(y)||&\geq \frac1L||x-y||,
\end{align*}
and $f^*$ is $1/L$-strongly convex.
\end{proof}

\subsection{Primal-Dual Algorithms}\label{sec:pd}
We now consider the optimization problem:
\be\label{eqPD1}
\min_{x\in E}f(Kx)+g(x),
\ee
where $K$ is a linear operator from $\E$ to $F$ and $f$ and $g$ are two convex proper l.s.c functions defined from $F$ (resp. $\E$) to $[-\infty,+\infty]$.
In the following, we assume that the proximal operators of both $f$ and $g$ are simple to compute, but the proximal operator of $f\circ K$ is not easily accessible.
To tackle this issue, we  consider the  {\it primal dual} formulation of problem~\eqref{eqPD1}. We first detail the corresponding saddle point problem and then present algorithms dedicated to its resolution.

\subsubsection{Saddle-point problem}
If  $f$ is convex proper and l.s.c, we use conjugation
\be
\forall z\in F,\quad f(z)=\sup_{y\in F}(\ps{y}{z}-f^{*}(y))\vspace{-0.1cm}
\ee
to rewrite the term $f(Kx)$ as
\be
\forall x\in E,\quad f(Kx)=\sup_{y\in F}(\ps{y}{Kx}-f^{*}(y))=\sup_{y\in F}(\ps{K^{*}y}{x}-f^{*}(y)),
\ee
where $K^*$ denotes the adjoint operator of $K$ defined from $F$ to $\E$. \newpage

Introducing this reformulation into the initial problem~\eqref{eqPD1}, we obtain the primal-dual problem
\be\label{eqPD2}
\min_{x\in\E}\sup_{y\in F}\,\ps{K^{*}y}{x}-f^{*}(y)+g(x),
\ee
defined for the primal unknown $x$ and the dual variable $y$.

In the following, we assume that the above supremum is a maximum. 
As soon as  the subdifferential of $f$ at point $z$ is non empty, the supremum is reached at a point $y$ of the subdifferential  $\partial f(z)$ 
and the supremum is indeed a maximum. We refer the reader to the book of R. Rockafellar~\cite{rockafellar1997convex} for further details. 
Having made this assumption, we return to the following {\bf saddle point} problem:
\be\label{eqPD3}
\min_{x\in\E}\max_{y\in F}\,h(x,y):=\ps{K^{*}y}{x}-f^{*}(y)+g(x), 
\ee
whose saddle-points $(x^*,y^*)$ provide minimizers $x^*$ of problem~\eqref{eqPD1}. In the following we assume that such saddle-points exist.
Problems~\eqref{eqPD1} and~\eqref{eqPD3} are equivalent in the following sense: 
if $(x,y)$ solves~\eqref{eqPD3} then $x$ is a minimizer of problem~\eqref{eqPD1}.

\begin{proposition}\label{prop_carac_pd}
Let $f$ and $g$ be convex and proper functions respectively defined from $F$ and $\E$ to $\R$, and $K$ be an operator from $\E$ to $F$. If $(x^*,y^*)$ is a saddle point of~\eqref{eqPD3}, then $x^*$ is a solution  of the primal problem~\eqref{eqPD1}.
Moreover, the saddle points $(x^*,y^*)$ of~\eqref{eqPD3} satisfy
\begin{equation}\label{eq:saddlePD}
\left\{\begin{array}{ll}
\hspace{-2pt}\forall \tau>0,&\hspace{-2pt}x^*=\prox_{\tau g}(x^*-\tau K^*y^*)\\
\hspace{-2pt}\forall \sigma>0,&\hspace{-2pt}y^*=\prox_{\sigma f^*}(y^*+\sigma Kx^*).
\end{array}\right.
\end{equation}
\end{proposition}
\begin{proof}
Let us recall that $h$ is the function defined over  $\E\times F$ by
\be\label{eqdefh}
h(x,y)=\ps{Kx}{y}+g(x)-f^*(y).
\ee
For $(x^*,y^*)\in \E\times F$, we introduce $h_1$ and $h_2$ the partial functions respectively defined on $\E$ and $F$ by 
$h_1(x)=h(x,y^*)$ and $h_2(y)=h(x^*,y)$. 
For any $x^*$ and $y^*$, 
$h_1$ is convex in $x$ and $h_2$ is concave in $y$ (i.e. $-h_2$ is convex). 

Hence, one can observe that
\begin{align}
\hspace*{-7pt}(x^*,y^*)\text{ is a saddle-point of  }h&\Leftrightarrow\nonumber
\hspace{-2pt}\left\{\begin{array}{ll}\hspace{-2pt}
\forall x\in\E,&h(x^*,y^*)\leqslant h(x,y^*)\\
\hspace{-2pt}\forall y\in F,&h(x^*,y)\leqslant h(x^*,y^*)\\
\end{array}\right.\\
&\Leftrightarrow\nonumber
\hspace{-2pt}\left\{\begin{array}{ll}
\hspace{-2pt}\forall x\in\E,&h_1(x^*)\leqslant h_1(x)\\
\hspace{-2pt}\forall y\in F,&-h_2(y^*)\leqslant -h_2(y)\\
\end{array}\right.\\
&\Leftrightarrow\nonumber
\hspace{-2pt}\left\{\begin{array}{l}
\hspace{-2pt}0\in\partial h_1(x^*)\\
\hspace{-2pt} 0\in\partial h_2(y^*)\\
\end{array}\right.\\
&\Leftrightarrow\nonumber
\hspace{-2pt}\left\{\begin{array}{l}
\hspace{-2pt}0\in(K^*y^*+\partial g(x^*))\\
\hspace{-2pt}0\in(-Kx^*+\partial f^*(y^*))\\
\end{array}\right.\\
\hspace{-2pt}&\Leftrightarrow\label{eqPFPD}
\left\{\begin{array}{ll}
\hspace{-2pt}-K^*y^*&\hspace{-2pt}\in\partial g(x^*)\\
\hspace{-2pt}Kx^*&\hspace{-2pt}\in\partial f^*(y^*)\\
\end{array}\right.\\
\hspace{-2pt}&\Leftrightarrow\nonumber
\left\{\begin{array}{ll}
\hspace{-2pt}\forall \tau>0,&\hspace{-2pt}x^*=\prox_{\tau g}(x^*-\tau K^*y^*)\\
\hspace{-2pt}\forall \sigma>0,&\hspace{-2pt}y^*=\prox_{\sigma f^*}(y^*+\sigma Kx^*).
\end{array}\right.
\end{align}

To conclude on the fact that $x^*$ is a minimizer of problem~\eqref{eqPD1}, 
we first provide the following lemma.
\begin{lemme}
Let $f$ be a convex and close function defined from $F$  to $\R$  and $K$ be an operator from $\E$ to $F$, then  $\forall y\in \partial f(Kx)$, $K^*y\in\partial (f\circ K)(x)$.\end{lemme}
\begin{proof}
Let $y\in \partial f(Kx)$, then from definition~\eqref{eq:sousdiff}, we have  $\forall x'\in\E$:  
\begin{align*}
    f(Kx')&\geq \langle y,Kx'-Kx\rangle + f(Kx)\\
    (f\circ K)(x')&\geq  \langle K^*y,x'-x\rangle + (f\circ K)(x),
\end{align*}
which is also equivalent to $K^*y\in \partial (f\circ K)(x)$ from definition~\eqref{eq:sousdiff}.
\end{proof}

Using relation~\eqref{eqProxSD} to switch from $f^*$ to $f$, as well as the previous lemma, we obtain from~\eqref{eqPFPD}:
\begin{equation*}\begin{split}
&\; \exists y^*\in F \textrm{ s. t. }  \left\{\begin{array}{ll}
-K^*y^*&\in \partial g(x^*)\\
Kx^*&\in \partial f^*(y^*)\end{array}\right.\\
\Leftrightarrow&\; \exists y^*\in F \textrm{ s. t. } \left\{\begin{array}{ll}
-K^*y^*&\in \partial g(x^*)\\y^*&\in \partial f(Kx^*)
\end{array}\right.\\
\Rightarrow&\; \exists y^*\in F \textrm{ s. t. }  \left\{\begin{array}{ll}
-K^*y^*&\in \partial g(x^*)\\
K^*y^*&\in \partial (f\circ K)(x^*)\end{array}\right.\\
 \Leftrightarrow &\; 0\in \partial(f\circ K)(x^*)+\partial g(x^*)\\
 \Leftrightarrow &\;x^*\in \argmin_{x\in \E} f(Kx)+g(x).
\end{split}
\end{equation*}
\end{proof}

As a consequence from Proposition~\eqref{prop_carac_pd}, algorithms dedicated to primal-dual problems are looking for pairs of variables $(x,y)$ satisfying relations~\eqref{eq:saddlePD}.\\
In particular, the Arrow-Hurwicz algorithm computes, for a given $x_0\in\E$,  the following pair of sequences $(x_n)_{n\in\N}$ and
$(y_n)_{n\in\N}$:
 \be\label{eqAlgoAH}
\left\{\begin{array}{lll}
y_{n+1}&=&\prox_{\sigma f^*}(y_n+\sigma K x_n)\\
x_{n+1}&=&\prox_{\tau g}(x_n-\tau K^*y_{n+1}).
\end{array}
\right.
\ee

The proof of convergence of this algorithm to a saddle-point of~\eqref{eqPD3} requires additional assumptions, such as the strong convexity of $f$ or $g$ to make the proximal operators contractions (see Proposition~\ref{prop:prox_strong}). 
We now focus on generalizations of this algorithm that involve minimal conditions.
\subsubsection{Chambolle-Pock algorithm (CP)}
Before presenting the algorithm proposed by Chambolle and Pock in~\cite{ChambollePock}, we first give a fundamental property of saddle-points of problem~\eqref{eqPD3}, by introducing the {\it partial primal-dual gap}:
\be
\mathcal{G}_{B_1\times B_2}(x,y)=\max_{y'\in B_2}\,\ps{y'}{Kx}-f^{*}(y')+g(x)-\min_{x'\in B_1}\,\ps{y}{Kx'}+g(x')-f^*(y),
\ee
that is defined on bounded sets $(B_1\times B_2)\subset \E\times F$.
\begin{proposition}\label{prop_gap}
Let $(x^*,y^*)\in B_1\times B_2$ be a solution to problem~\eqref{eqPD3}, then, $\forall (x,y)\in B_1\times B_2$, we have  $\mathcal{G}_{B_1\times B_2}(x,y)\geq 0$, and this gap is zero if and only if $(x,y)$ is a solution of~\eqref{eqPD3}.
\end{proposition}
\begin{proof}
We recall that the function $h$ is defined in~\eqref{eqPD3} as $h(x,y)=\ps{Kx}{y}+g(x)-f^*(y)$. A saddle point $(x^*,y^*)$ of problem~\eqref{eqPD3} thus satisfies:
$h(x^*,y^*)=\min_{x'\in B_1}h(x',y^*)=\max_{y'\in B_2}h(x^*,y')$, which implies that 
\begin{equation}\label{gap_sad}
\mathcal{G}_{B_1\times B_2}(x^*,y^*)=\max_{y'\in B_2}h(x^*,y')-\min_{x'\in B_1}h(x',y^*)=h(x^*,y^*)-h(x^*,y^*)=0.
\end{equation}
Next we have
\begin{align}\nonumber
\mathcal{G}_{B_1\times B_2}(x,y)
&=\max_{y'\in B_2}\,h(x,y')-\min_{x'\in B_1}\,h(x',y)\nonumber\\
&\geqslant h(x,y^*)-h(x^*,y)\nonumber\\
&\geqslant h(x,y^*)-h(x^*,y^*)+h(x^*,y^*)-h(x^*,y)\nonumber\\
&\geqslant h(x,y^*)-\min_{x'\in B_1}h(x',y^*)+\max_{y'\in B_2}h(x^*,y')-h(x^*,y)\nonumber\\
&\geqslant 0.\label{eq_pd_gap}
\end{align}

\end{proof}

We can now present the Chambolle-Pock algorithm~\cite{ChambollePock}, that is based on a 
modification of the scheme~\eqref{eqAlgoAH} 
with a nonconvex update of $x_n$.
\begin{theoreme}[Chambolle-Pock algorithm]\label{TPD}
Let $f$ and $g$ be two convex, proper, lower semi-continuous functions, with $f$ defined from $F$ to $\R$, $g$ defined from $\E$ to $\R$ and $K$ a linear operator from $\E$ to $F$. Assume that the problem \eqref{eqPD3} admits a solution $(x^*,y^*)$. We  denote $L=\norm{K}=\sqrt{||K^*K||}$, and take $\sigma$ and $\tau$ such that $\tau \sigma L^2<1$. We choose $(x_0,y_0)\in\E\times F$ and set $\bar{x}_0=x_0$ and define the sequences $(x_n)_{n\in\N},\,(y_n)_{n\in\N},\,\text{and }(\bar{x}_n)_{n\in\N}$ by
\be\label{eqAlgoPD}
\left\{\begin{array}{lll}
y_{n+1}&=&\prox_{\sigma f^*}(y_n+\sigma K\bar{x}_n)\\
x_{n+1}&=&\prox_{\tau g}(x_n-\tau K^*y_{n+1})\\
\bar{x}_{n+1}&=&2x_{n+1}-x_n.
\end{array}
\right.
\ee   
Setting $x^N=(\sum_{n=1}^N x_n)/N$ and $y^N=(\sum_{n=1}^N y_n)/N$, then for each bounded set $(B_1\times B_2)\subset \E\times F$, the {\it partial primal-dual gap} satisfies
\begin{align*}
0\leq\mathcal{G}_{B_1\times B_2}(x^N,y^N)&\leqslant \frac1{N}\left(\frac{||x_0-x^*||^2}{2\tau}+\frac{||y_0-y^*||^2}{2\sigma}-\langle K(x_0-x^*),y_0-y^*\rangle\right).
\end{align*}

\end{theoreme}
This theorem illustrates the ergodic convergence rate of the primal-dual algorithm~\cite{chambolle2016ergodic}. 
For completeness, we provide in Appendix~\ref{sec:app:CP} the  proof of~\cite{ChambollePock}, that shows the convergence of the iterates $(x_n,y_n)$ of algorithm~\eqref{eqAlgoPD} to a saddle-point $(x^*,y^*)$ of problem~\eqref{eqPD3}.


\begin{proof}
This demonstration is based on the works of~\cite{he2012convergence,lu2023unified}. 
Let us first switch the step order in algorithm~\eqref{eqAlgoPD}:
\be\label{eqAlgoPD2}
\left\{\begin{array}{lll}
x_{n+1}&=&\prox_{\tau g}(x_n-\tau K^*y_{n+1})\\
\bar{x}_{n+1}&=&2x_{n+1}-x_n\\
y_{n+1}&=&\prox_{\sigma f^*}(y_n+\sigma K\bar{x}_{n+1}).
\end{array}
\right.
\ee   
Using the characterization of proximal operators in Proposition~\ref{PropProxSD}, we have:
\begin{align}\nonumber
&\left\{\begin{array}{lll}
-(\frac{x_{n+1}-x_n}\tau +K^*y _n)&\in \partial g(x_{n+1})\\
-( \frac{y_{n+1}-y_n}\sigma -K(2x_{n+1}-x_n))&\in \partial f^*(x_{n+1})\\
\end{array}
\right.\\\nonumber
\Leftrightarrow &
0\in \left\{\begin{array}{lll}
\frac{x_{n+1}-x_n}\tau +K^*y _n+\partial g(x_{n+1})\\
 \frac{y_{n+1}-y_n}\sigma -K(2x_{n+1}-x_n)+\partial f^*(x_{n+1})\\
\end{array}
\right.\\\nonumber
\Leftrightarrow &
0\in  
\begin{bmatrix}
    \frac{\id}\tau&0\\0&\frac{\id}\sigma
\end{bmatrix}\begin{bmatrix}
    x_{n+1}-x_n\\ y_{n+1}-y_n
\end{bmatrix}
-\begin{bmatrix}
    0&K^*\\K&0
\end{bmatrix}\begin{bmatrix}
    x_{n+1}-x_n\\ y_{n+1}-y_n
\end{bmatrix}+\begin{bmatrix}
    \partial g&K^*\\-K&\partial f^*
\end{bmatrix}\begin{bmatrix}
    x_{n+1}\\ y_{n+1}
\end{bmatrix}\\\label{eq:PD_CP}
\Leftrightarrow & 
0\in
\begin{bmatrix}
    \frac{\id}\tau&-K^*\\-K&\frac{\id}\sigma
\end{bmatrix}
\begin{bmatrix}
    x_{n+1}-x_n\\ y_{n+1}-y_n
\end{bmatrix}
+
\begin{bmatrix}
    \partial g&K^*\\-K&\partial f^*
\end{bmatrix}
\begin{bmatrix}
    x_{n+1}\\ y_{n+1}
\end{bmatrix}.
\end{align}
Recalling that $$h(x,y)=\ps{Kx}{y}+g(x)-f^*(y),$$
and introducing the matrix 
$$M=\begin{bmatrix}
    \frac{\id}\tau&-K^*\\-K&\frac{\id}\sigma
\end{bmatrix},$$
relation~\eqref{eq:PD_CP} can be written as 
\begin{equation}\label{eq_tmp}
M \begin{bmatrix}
    x_{n}-x_{n+1}\\ y_{n}-y_{n+1}
\end{bmatrix}\in 
\begin{bmatrix}
    \partial_x h(x_{n+1},y_{n+1})\\
    -\partial_y h(x_{n+1},y_{n+1})
\end{bmatrix}.
\end{equation}
In the following, we consider $\sigma\tau||K^*K||=\sigma\tau L^2<1$ to ensure that the matrix $M$ is  positive-definite. 
Since $h$ is convex  (resp. concave) with respect to $x$ (resp. $y$), we get from the definition of subdifferential~\eqref{def:subdif} that:
\begin{align*}
   h(x_{n+1},y_{n+1}) +\langle \partial_x h(x_{n+1},y_{n+1}),x-x_{n+1}\rangle& \leq  h(x,y_{n+1})\\
   -h(x_{n+1},y_{n+1})-\langle \partial_y h(x_{n+1},y_{n+1}),y-y_{n+1} \rangle & \leq -  h(x_{n+1},y).
\end{align*}

Summing both expressions and introducing $w_{n+1}=M(z_n-z_{n+1})$, we get from  relation~\eqref{eq_tmp} 
\begin{align*}
    h(x_{n+1},y)-h(x,y_{n+1})&\leq \langle w_{n+1},z_{n+1}-z\rangle=\langle z_n-z_{n+1},z_{n+1}-z\rangle_M\\
   &\leq \frac12\left( ||z_n-z ||_M^2- ||z_{n+1}-z ||_M^2- ||z_{n+1}-z_n ||_M^2\right)\\
   &\leq \frac12\left( ||z_n-z ||_M^2- ||z_{n+1}-z ||_M^2\right),
\end{align*}
where we used the fact that $||.||^2_M=\langle M.,.\rangle$ is a semi-norm since $M$ is positive-definite.
Summing the previous expression from $n=1$ to $N$ and introducing $x^N=(\sum_{n=1}^N x_n)/N$ and $y^N=(\sum_{n=1}^N y_n)/N$,  we obtain:

$$ h(x^N,y)-h(x,y^N)\leq \frac1N \sum_{n=1}^N \left(h(x_{n},y)-h(x,y_{n})\right)\leq \frac1{2N} ||z_{0}-z ||_M^2.$$
Considering previous expression for a saddle point $z^*=(x^*,y^*)$, we get from Proposition~\ref{prop_gap} that 
$$0\leq \mathcal{G}_{B_1\times B_2}(x^N,y^N)\leqslant \frac1{2N} ||z_{0}-z^* ||_M^2.\vspace*{-0.5cm}$$
\end{proof}

\begin{remark}
The work in~\cite{he2012convergence} offers an alternative to the proof of convergence of the primal-dual algorithm provided in Appendix~\ref{sec:app:CP}. 
The convergence of primal-dual can indeed be obtained from the fact that the scheme~\eqref{eq:PD_CP} is an
instance of the proximal point algorithm of~\cite{rockafellar1976monotone} applied to the variable $z=(x,y)$ in which  $\begin{bmatrix}
    \partial g&K^*\\-K&\partial f^*
\end{bmatrix}$ corresponds to a maximal monotone operator (a monotone operator $T:\mathcal{H}\to\mathcal{H}$  satisfies $\langle T(z)-T(z'),z-z'\rangle \geq 0$ for all $z,z'\in\mathcal{H}^2$ and it is maximal if and only if $\text{range}(T+\id)=\mathcal{H}$, with $\text{range}(T) = \{z \in\mathcal{H},\,\text{such that } \exists w\in\mathcal{H},\,\text{satisfying }z=T(w)\}$. The iterates $z_n=(x_n,y_n)$ then converge if the matrix $M$ is  positive-definite, which is the case  as soon as $\sigma\tau||K^*K||=\sigma\tau L^2<1$.
\end{remark}

\subsubsection{Condat algorithm}
The Condat algorithm generalizes the primal dual one to the minimization of problems involving  more convex functions:
\begin{equation}\label{pb:condat}
\min_{x\in \E}f(x)+g(x)+\sum_{i=1}^M h_i(L_ix).
\end{equation} 
where $f$ is a convex differentiable function whose gradient is $L$-Lipschitz and the functions $g$ and $h_i$ are convex, and $L_k$ are linear operators.  
For any initial points $x_0\in E$ and $(u_0^{i})_{i\leqslant M}$ initial dual variables, the Condat algorithm introduced in~\cite{Condat13} defines the sequences $(\tilde  x_n)_{n\in\N}$, $(x_n)_{n\in\N}$, $(\tilde  u_n)_{n\in\N}$, $(u_n)_{n\in\N}$ as

\begin{equation}\label{algo:condat}
    \left\{\begin{array}{lll}
         \tilde x_{n+1}&=\prox_{\tau g}\left(x_n-\tau \nabla f(x_n)-\tau \sum_{i=1}L^*_m u_n^i\right) \\
         x_{n+1}&=\rho \tilde x_{n+1}+(1-\rho)x_n\\
         \tilde u_{n+1}^i&=\prox_{\sigma h^*_i}\left(u_k^i+\sigma L_i(2\tilde x_{n+1}-x_n)\right)&\text{ for }i=1\ldots M\\
         u_{n+1}^i&=\rho \tilde u_{n+1}^i+(1-\rho)u_k^i&\text{ for }i=1\ldots M.\\
    \end{array}\right.
\end{equation}
For  $\rho \in (0,1]$ and $\sigma>0$, $\tau>0$ such that 
\begin{equation}
\tau\left(\frac{L}{2}+\sigma \left|\left|\left|\sum_{i=1}^ML_i^*L_i\right|\right|\right|\right)<1,
\end{equation}
it is shown in~\cite{Condat13} (the proof is not reproduced here) that the sequence $(x_n)_{n\in\N}$ defined in~\eqref{algo:condat} weakly converges to a minimizer of~\eqref{pb:condat}.\newpage
 
 \noindent
The Condat  algorithm have the following properties:
\begin{itemize}
\item It allows to use an explicit gradient descend on the differentiable part, as the Forward-Backward algorithm;
\item It  separates functions and operators, as done in the Chambolle-Pock algorithm;
\item The computations of $(u_n^i)_{i\leqslant M}$ are independent and can be done separately and allow parallel computing as with PPXA.
\end{itemize}
\subsection{Equivalence between proximal splitting algorithms}\label{sec:equiv}
We now show that in particular cases, some proximal splitting algorithms are equivalent. To that end, we consider the general problem
\be\label{eq-addm-dr-primal}
	\min_{x\in \E} f(Kx) + g(x) ,    
\ee
where the linear operator $K$ is defined from $\E$ to $F$. 
For convex proper and l.s.c functions $f$ and $g$, this problem can be solved with the Douglas-Rachford algorithm~\eqref{algoDR00}: 
\begin{equation}\label{Algodr_equiv}
\left\{\begin{array}{ll}
w_{n+1} &=w_{n}  + \prox_{\gamma f \circ K} (2v_{n} -w_{n} )-v_n,\\
v_{n+1} &= \prox_{\gamma g}(w_{n+1}).
\end{array}\right.
\end{equation}
We now detail the equivalence of this algorithm with the Chambolle-Pock and the ADMM ones.
\subsubsection{Chambolle-Pock and Douglas-Rachford equivalence}
We  show the equivalence between Chambolle-Pock and Douglas-Rachford algorithms when the linear operator is the identity.
\begin{proposition}
	Considering the primal problem~\eqref{eq-addm-dr-primal}, in the case $K=Id$ and setting
	\be\label{eq-chg-var-dr-pd}
		\sigma = \frac{1}{\gamma}\text{ and }
		\tau=\gamma, 
	\ee
	the Douglas-Rachford iterations~\eqref{Algodr_equiv} are recovered from the Chambolle-Pock iterations~\eqref{eqAlgoPD} using
	\begin{equation}\label{change_var}
		\left\{\begin{array}{ll}
  x_{n+1}&=v_{n+1},\\
        \gamma y_{n+1}&=v_n-w_{n+1}. 
        \end{array}\right.
	\end{equation}
\end{proposition}
\begin{proof}
The Chambolle-Pock algorithm with the above parameter reads:
\begin{align*}
&\left\{\begin{array}{ll}
	y_{n+1} &= \prox_{{f^*}/\gamma}( y_{n} +    \bar x_n /\gamma)=y_n +    \bar x_n /\gamma -\frac{1}\gamma\prox_{\gamma{f}}( \gamma y_n +    \bar x_n) \\
	x_{n+1} &= \prox_{\tau g}(  x_n - \gamma y_{n+1}   ), \\
	\bar x_{n+1} &= 2x_{n+1} -x_n,
\end{array}\right.\\
\Leftrightarrow &
\left\{\begin{array}{ll}
	-\gamma y_{n+1}+\gamma y_n + 2x_{n}-x_{n-1} &= \prox_{\gamma{f}}( \gamma y_n +    2x_{n}-x_{n-1}) \\
	x_{n+1} &= \prox_{\tau g}(  x_n - \gamma y_{n+1}   )\\
\end{array}\right.
\end{align*}
Incorporating the change of variables~\eqref{change_var}, we obtain
\begin{align*}
\left\{\begin{array}{ll}
	w_{n+1}-w_n+v_n &= \prox_{\gamma{f}}( 2v_n-w_n ) \\
	v_{n+1} &= \prox_{\tau g}(  w_{n+1}   ),
\end{array}\right.
\end{align*}
that corresponds to  the Douglas-Rachford iterations~\eqref{Algodr_equiv}.
\end{proof}

\subsubsection{ADMM and Douglas-Rachford  equivalence}\label{sec:equivADMMDR}

To show the equivalence between Douglas-Rachford and ADMM, we consider the Fenchel-Rockafeller dual to the primal problem~\eqref{eq-addm-dr-primal} that writes: 
\be\label{eq-addm-dr-primal2}
\max_{z\in F} - \left(g^*(-K^*z)+f^*(z)\right)
\Leftrightarrow \min_{z\in F} f^*(z) + g^*(-K^*z),   
\ee
where we assume that the adjoint operator $K^*$ defined from $F$ to $\E$ is  injective. 
This problem can be reformulated as 
\be\label{eq-addm-dr-primal3}
	\min_{(x,y)\in F \times \E,\, K^*x+y=0} f^*(x) + g^*(y).   
\ee
This corresponds to a particular form of the problem~\eqref{eqADMM}  for the two convex conjugate $f^*$ and $g^*$ defined on $F$ and $\E$,  the operators  $A=K^*$, $B=\id$ and the vector $b=0$. It can be solved with the Alternating Direction Method of Multipliers (ADMM) algorithm presented in~\eqref{algo:ADMM}:
\be\label{eq-admm-iter}
\left\{\begin{array}{ll}
x_{n+1}&=\uargmin{x}f^*(x)+\ps{z_{n}}{K^*x}+\frac{\gamma}{2}\norm{y_{n}+K^*x}^2\\
y_{n+1}&=\uargmin{y}g^*(y) + \ps{z_{n}}{y}+\frac{\gamma}{2}\norm{K^*x_{n+1}+y}^2\\
z_{n+1}&=z_n+\gamma(y_{n+1}+K^*x_{n+1}).
\end{array}\right.
\ee
Introducing the proximal operator $\prox_{\gamma f}^{A}$   with a metric induced by an injective map~$A$ 
\be\label{eq-def-proxB}
	\prox_{\gamma f}^A(x) = \uargmin{s} \frac{1}{2\gamma} \norm{A s - x}^2 + f(s).
\ee
 the ADMM iterations~\eqref{eq-admm-iter} can be rewritten using proximal maps
\begin{equation}\label{eq-admm-prox-iters}
\left\{\begin{array}{ll}
	x_{n+1} &= \prox_{f^*/\gamma}^{K^*} ( -y_n -  z_n/\gamma), \\
	y_{n+1} &= \prox_{ g^*/\gamma } ( -K^*x_{n+1} -  z_n/\gamma ), \\
	z_{n+1} &= z_n + \gamma(y_{n+1}+K^*x_{n+1}).
\end{array}\right.
\end{equation}

The following proposition, which was initially proved in~\cite{Gabay83}, shows that applying the ADMM algorithm to the dual problem~\eqref{eq-addm-dr-primal2} is equivalent to solving the primal~\eqref{eq-addm-dr-primal} using the Douglas-Rachford algorithm. 
This result was further extended by~\cite{Eckstein1992} which shows the equivalence of ADMM with the proximal point algorithm and propose several generalizations.

\begin{proposition}\label{equiv_admmdr}
	For $f$ and $g$ convex, the Douglas-Rachford iterations~\eqref{Algodr_equiv}  are recovered from the ADMM iterations~\eqref{eq-admm-prox-iters} using
	\begin{equation}\label{change_var2}
 \left\{\begin{array}{ll}
		   z_n  &= -v_n,\\
		\gamma  y_n &= w_n-v_n, \\
		 -\gamma K^* x_{n+1} &= w_{n+1}-v_n. 
   \end{array}\right.  
	\end{equation}
\end{proposition}
To show this proposition, we first require the following result~\cite{papadakis2014optimal}.
\begin{proposition}\label{eq-proxB-proxD}
One has
\be	
		\prox_{f/\gamma}^A(x) = A^+ \pa{ x - \frac{1}{\gamma} \prox_{\gamma f^* \circ A^* }(\gamma x) },
\ee
where $A^+=(A^*A)^{-1}A^*$ is the pseudo-inverse of $A$. Note that in the case   $A=\id$, one recovers Moreau's identity~\eqref{eq-moreau-identity}.
\end{proposition}
\begin{proof}
	We denote $ \Uu=\partial f$ the set-valued maximal monotone operator. One has
	$\partial f^* =  \Uu^{-1}$, where $ \Uu^{-1}$ is the set-valued inverse operator. We thus have  $\partial (f^* \circ A^*) = A \circ \Uu^{-1} \circ A^*=\Vv $ and $\prox_{\gamma f^* \circ A^* } = (\id + \gamma  \Vv)^{-1}$, which is a single-valued operator. 
	Denoting  $x^* = \prox_{f/\gamma}^A(x)$, the optimality condition of~\eqref{eq-def-proxB} leads to
	\begin{align*}
		& 
		0 \in A^*( A x^* - x) + \frac{1}{\gamma}  \Uu (x^*) 
		\Leftrightarrow
		x^* \in  \Uu^{-1} \pa{ \gamma A^* A x^* - \gamma A^* x   }\\
		\Leftrightarrow& \, 
		\gamma A x^* \in \gamma \Vv  \pa{ \gamma x - \gamma A x^*    } 	 
		\Leftrightarrow
		\gamma A x^* \in ( \id + \gamma \Vv )\pa{ \gamma x - \gamma A x^*   } + \gamma A x^* - \gamma x \\
		\Leftrightarrow &\, 
		\gamma x \in ( \id + \gamma \Vv )\pa{ \gamma x - \gamma A x^*   } 
		\Leftrightarrow
		\gamma x - \gamma A x^* = ( \id + \gamma \Vv )^{-1}(\gamma x)\\ 
		\Leftrightarrow&\,  x^* = A^{+}\pa{ x - \frac{1}{\gamma} ( \id + \gamma \Vv )^{-1}(\gamma x) }		
	\end{align*}
	where we used the fact that $A$ is injective. 
\end{proof}

\begin{proof}[Proof of Proposition~\ref{equiv_admmdr}]
	Denoting $\bar x_n = -K^* x_n\in \text{Im}(-K^*)$ (recall that $K^*$ is injective) and using the result of Proposition~\ref{eq-proxB-proxD}, the iterates~\eqref{eq-admm-prox-iters} write
	\begin{equation}\label{eq-equiv-iter1}
		\left\{\begin{array}{ll}
		\bar x_{n+1} &= y_n+z_n/\gamma  + \frac{1}{\gamma} \prox_{\gamma f \circ K}\pa{ -\gamma y_n-z_n }, \\
		y_{n+1} &=  \bar x_{n+1}-z_n/\gamma  - \frac{1}{\gamma} \prox_{\gamma g}\pa{ \gamma \bar x_{n+1}-z_n )}, \\
		z_{n+1} &= z_n+\gamma (y_{n+1}-\bar x_{n+1}),
		\end{array}\right.
	\end{equation}
so that 
	\begin{equation}\label{eq-equiv-iter2}
		\left\{\begin{array}{ll}
		-z_n+\gamma(\bar x_{n+1}-y_n) &=  \prox_{\gamma f \circ K}\pa{ -\gamma y_n-z_n }, \\
		-z_n +\gamma (\bar x_{n+1}-y_{n+1}) &=   \prox_{\gamma g}\pa{ \gamma \bar x_{n+1}-z_n )}, \\
		z_{n+1} &= z_n+\gamma (y_{n+1}-\bar x_{n+1}).
		\end{array}\right.
	\end{equation}
Considering the following change of variables~\eqref{change_var2}:
\begin{align*}
\left\{\begin{array}{ll}
		   z_n  &=  -v_n,\\
		\gamma  y_n &= w_n-v_n, \\
		\gamma  \bar x_{n+1} &= w_{n+1}-v_n,
  \end{array}\right.\vspace{-0.1cm}
	\end{align*}
relations~\eqref{eq-equiv-iter2} become
\begin{equation}\label{eq-equiv-iter2a}
  		\left\{\begin{array}{ll}
		w_{n+1}&= w_n+\prox_{\gamma f\circ K}(2 v_n-w_n)-v_n, \\
		v_{n+1} &= \prox_{\gamma g} (w_{n+1}),\\
        0&=0,
		\end{array}\right.
	\end{equation}
that gives  the Douglas-Rachford iterations~\eqref{Algodr_equiv}.
\end{proof}

\begin{remark} \label{rmk:DRS_ADMM}
For $K=\id$, the previous result remains true for $f$ and $g$ nonconvex functions. Indeed, as already mentioned in Section~\ref{sec:ADMM}, for $K=\id$, there is a primal equivalence between ADMM~\eqref{eq:ADMM_simple} and DRS~\eqref{eqDR4}. 
\end{remark}

\section{Extension to nonconvex optimization}\label{sec:nonconvex_opt}
We now study the application of previously introduced gradient based algorithms to the optimization of nonconvex functionals. Without convexity, a lot of issues arise
\bi
\item A local minimum is no longer necessarily global.  The convergence  to a (local) minimum of the nonconvex function is not ensured. Iterative optimization algorithms rather target the convergence to critical points of the function.
\item The  subdifferential may be reduced to an empty set at many points.
\item The proximal operator may be 
undefined or multivalued.
\item The identity between a function 
$f$ and its biconjugate $f^{**}$ is no longer valid and can not be exploited in primal-dual formulations.
\ei
Despite these limits, we can still prove some convergence results. In particular, by finding decreasing Lyapunov functions, we can show convergence to zero of the norm between consecutive iterates. Then, assuming additional regularity on the optimized function around its critical points, such as the Kurdyka-Łojaseiwicz property, one can prove convergence of the iterates towards a critical point of the function. 

In this section, we first provide  preliminary results on the convergence of the gradient descent scheme in the nonconvex setting in section~\ref{sec:nonconvex_gd}. Next a detailed analysis of the Forward-Backward algorithm for nonconvex problems is given in section~\ref{sec:nonconvex_FB}. In section~\ref{sec:nonconvDR}, we present existing results relative to the use of Douglas-Rachford, ADMM and Chambolle-Pock algorithms in the nonconvex setting.

\subsection{From convex to nonconvex optimization, the case of gradient descent}\label{sec:nonconvex_gd}

A part of the convergence results that were proved in the convex setting are still valid on the nonconvex setting. As in Section~\ref{sec:smooth_opt}, we consider 
\begin{equation}\label{eq:min_ncvx}
\min_{x\in E}f(x)
\end{equation} 
with a nonconvex, $L$-smooth function $f: E \subseteq \R^d \to \R \cup \{+\infty\}$. The main difference with the convex case is that for minimizing a nonconvex function, there is usually no hope to target a global minimum. The gradient of $f$ is still defined and we instead look for a stationary point of $f$, i.e. a point $x$ such that $\nabla f(x) = 0$. Notice that in the nonconvex case, such a stationary point is not necessarily a global minimum.

We consider again the gradient-descent algorithm with stepsize $\gamma>0$,
\begin{equation}\label{algo:gd_ncvx}
x_{n+1}=x_n-\gamma \nabla f(x_n).
\end{equation} 
In the proof of convergence of the gradient-descent algorithm in the convex setting in Theorem~\ref{TheoGradPasFixe}, without making use of the convexity of $f$ and using the fact that $f$ is $L$-smooth, we proved the following sufficient decrease property:
\begin{equation} \label{eq:SDC_GD_ncvx}
f(x_{n+1})+\left(\frac{2-\gamma L}{2\gamma}\right)\norm{x_{n+1}-x_n}^2\leqslant f(x_n).
\end{equation}
Therefore, as soon as $\gamma < \frac{2}{L}$, the function $f$ decreases along the iterates. Assuming that $f$ is lower-bounded, the sequence of function values $f(x_n)$ then converges. Moreover, by summing~\eqref{eq:SDC_GD_ncvx}, between $n=0$ and $n=N>0$, we get that $\norm{x_{n+1}-x_n} = \norm{\nabla f(x_n}$ converges to $0$ with rate $\mathcal{O}(1/\sqrt{N})$. Eventually, assuming $f$ of class $\mathcal{C}^1$, we get that any cluster point $x^*$ of the sequence $(x_n)$ verifies $\nabla f(x^*)=0$,  i.e. $x^*$ is a critical points of the objective function $f$.
These results are summarized in the following Proposition.
\begin{proposition}[Nonconvex Gradient Descent]
\label{prop:nonconvex_PGD}
Assume $f$ proper, lsc, bounded from below and $L$-smooth. Then, for $\gamma < 2/L$, the iterates $(x_n)$ given by the Gradient Descent algorithm~\eqref{algo:gd_ncvx} verify 
\begin{itemize}
     \item[(i)] $f(x_n)$ is non-increasing and converges.
    \item[(ii)] The sequence $\norm{x_{n+1}-x_n}$ converges to $0$ at rate $\min_{n < N} \norm{x_{n+1}-x_n} = \mathcal{O}(1/\sqrt{N})$
    \item[(iii)] All cluster points of the sequence $(x_n)$ are critical points of $f$.
\end{itemize}
\end{proposition}

It is important to note that this result does not guarantee the convergence of the sequence $(x_n)_{n\in\N}$ itself, which could infinitely turn around a level set of $f+g$ that is not a local minima.  In the following section, we will explain that under an additional assumption on the geometry of the objective function around its critical points, we can prove single-point convergence of the sequence $(x_n)$ towards a critical point of~$f$.

\subsection{Single-point nonconvex convergence of Forward-Backward}\label{sec:nonconvex_FB}

Let's consider again the more general problem
\begin{equation}\label{pb:nonconvFB}
\min_{x\in \E} f(x)+g(x)
\end{equation}
under the main assumption of a $L$-smooth function $f$. 
Note that even if the functions $f$ and/or $g$ are nonconvex, the following Forward-Backward algorithm is well defined:
\begin{equation}
    \label{eq:PGD}
    x_{n+1} \in \prox_{\gamma g} \left(x_n- \gamma \nabla f(x_n)\right).
\end{equation}
Note that, due to the nonconvexity of $g$, the output of the proximal operator may not be unique.
We consider this algorithm because it is more general than the Gradient Descent algorithm. In particular, it allows to treat non-differentiable term $g$. \\

In this section, we study the point-wise convergence of the iterates of the Forward-Backward algorithm~\eqref{eq:PGD} in the nonconvex setting. As the objective function is not necessarily differentiable, we look for a zero of its subdifferential. We first need to define the right notion of nonconvex subdifferential. Then, we introduce the Kurdyka-Łojasiewicz property, a local property that characterizes the shape of a function around its critical points. Finally, we prove iterate convergence of the Forward-Backward algorithm~\eqref{eq:PGD} under this property.

\subsubsection{Nonconvex subdifferential and critical points}\label{sec:ncvx_subdif}
We first define the notion of optimum for the nonconvex problem~\eqref{pb:nonconvFB}. For a nonconvex function $f$, the usual notion of subdifferential from Definition~\ref{def:subdif} may not be informative enough. Indeed, with this notion, with Fermat's rule (Theorem~\ref{thm:fermat}), the critical points of $f$, i.e. the zeros of the subdifferential $\partial f$ corresponds to the minimizers of $f$. Therefore, targeting a critical point comes back to looking for a global minimum. We explained that this is out of reach when minimizing a general nonconvex function. Instead, one can hope to reach a point that is critical for some generalized notion of subdifferential.

To study the convergence of proximal iterative schemes for minimizing a nonconvex function, an adequate notion of subdifferential \cite{ABS13} is the \textit{limiting subdifferential}, also called \emph{general subgradient} in \cite{rt1998wets}:
\begin{equation} \label{eq:limiting_subdifferential}
   \partial^{lim} f(x) = \left\{\omega \in \E, \exists x_n \to x, f(x_n) \rightarrow f(x), \omega_n \rightarrow \omega, \omega_n \in \hat \partial f(x_n) \right\}
\end{equation}
with $\hat \partial f$ the  \textit{Fréchet subdifferential} of $f$ (also called \emph{regular subgradient} in \cite{rt1998wets}) defined as
\begin{align}
    \hat \partial f(x) &= \left\{ \omega \in \E, \liminf_{y \to x} \frac{f(y) - f(x) - \langle \omega, y-x \rangle }{\norm{x-y}} \geq 0 \right\} .
 \end{align}
The three introduced notions of subdifferential verify for $x \in \dom f$
\begin{equation} \label{eq:inclusion_subdiff}
    \partial f(x) \subset \hat \partial f(x) \subset \partial^{lim} f(x) .
\end{equation}
For $f$ \emph{convex}, the three subdifferentials coincide \cite[Proposition 8.12]{rt1998wets}. For $f$ of class $\mathcal{C}^1$ they also coincide with the usual concept of gradient $\hat \partial f(x) = \partial^{lim} f(x) = \{ \nabla f(x)\}$. 
This generalized notion of subdifferential gives birth to generalized notions of
\emph{critical point} or \emph{stationary point} \emph{i.e.} $x \in \E$ such that
\begin{equation}
    0 \in \partial^{lim} f(x).
\end{equation} 
A necessary (but not sufficient) condition for $x \in \E$ to be a local minimizer of a nonconvex function $f$ is  $0 \in \hat \partial f(x)$ and thus  $0 \in \partial^{lim}  f(x) $. 
Last but not least, the limiting subdifferential also verifies the sum rule.
\begin{proposition}[Sum rule {\cite[8.8(c)]{rt1998wets}}] \label{prop:subdif_sum}
    If $\mathcal{J} = f + g$ with $f$ of class $\mathcal{C}^1$ and $g$ proper. Then for $x \in \dom(g)$, 
    \begin{equation}
        \partial^{lim} \mathcal{J}(x) = \nabla f(x) + \partial^{lim} g(x).
    \end{equation}
\end{proposition}
In particular, thanks to this property, a fixed point of the Forward-Backward algorithm~\eqref{eq:PGD} is a critical point of $f+g$
\begin{align}
x^{*}\in \prox_{\gamma g}(x^{*}-\gamma \nabla f(x^{*}))  
&\Rightarrow 0 \in \gamma \partial^{lim} g(x^{*}) + (x^{*} - (x^{*}-\gamma \nabla f(x^{*}))) \\
&\Rightarrow 
-\nabla f(x^{*})\in\partial^{lim} g(x^{*}) \\
&\Rightarrow 0 \in \partial^{lim} (f + g)(x^{*})
\end{align}

For additional details on the notion of limiting subdifferential, we refer to \cite[Chapter 8]{rt1998wets}.
In the rest of this section, we will denote for simplicity $\partial f$ as the limiting subdifferential $\partial^{lim} f$.

\subsubsection{Kurdyka-Łojasiewicz (KŁ) property }\label{sec:KŁ}
To ensure the single-point convergence of the Forward-Backward algorithm towards a critical point of $\mathcal{J} = f + g$ (i.e. a zero of the limiting subdifferential of $\mathcal{J}$), we introduce the Kurdyka-Łojasiewicz (KŁ) property~\cite{bolte2007lojasiewicz}. It is a local property that characterizes the shape of the function $\mathcal{J}$ around its critical points.
Before entering into technical details, we  first present  a simpler condition given by Łojaseiwicz for smooth functions \cite{lojasiewicz1963propriete, lojasiewicz1982trajectoires}.
\begin{definition}[Łojaseiwicz condition]
Let $f: \E \to \R\cup \{+\infty\}$. We say that the function  $f$ satisfies the Łojasiewicz condition if, for any critical point $x^*$ of $f$, there exists a neighborhood $V$ of $x^*$, $C>0$,  and $\alpha\in[0,1[$
such that for all $x\in V$
 \begin{equation}\label{def:lowa}
 \norm{\nabla f(x)}\geqslant C|f(x)-f(x^*)|^{\alpha}.
 \end{equation}
\end{definition}

For a smooth function $f$, this condition is  verified at points $x$ such that \\ $\nabla f(x)\neq 0$. The Łojaseiwicz  condition thus ensures that a function 
is not too flat around its critical points. 
It is an essential property to ensure  convergence of optimization algorithms in the nonconvex setting. The
Łojasiewicz property is verified in particular by \textit{real analytic functions}.

The Kurdyka-Łojaseiwicz condition \cite{kurdyka1998gradients} generalizes the previous Łojaseiwicz condition to non\-smooth functions (see also~\cite{ABS13} for more details).

\begin{definition}[Kurdyka-Łojaseiwicz property]\label{defKŁ}
The function $f\,:\,\E\to \R\cup \{+\infty\}$ satisfies the Kurdyka-Łojaseiwicz property at point $x^*\in \dom\partial f$ if there exists $\eta>0$, a neighborhood  $V$ of $x^*$ and a function $\varphi\,:\,[0,\eta[\to \R_+$ such that:
\begin{enumerate}
	\item $\varphi(0)=0$,
	\item $\varphi$ is $C^1$ on $]0,\eta[$,
	\item For all $s\in]0,\eta[,\,\varphi(s)>0$,
	\item For all  $x\in V$ such that $f(x)\in]f(x^{*}),f(x^*)+\eta[\}$, the KŁ property is verified:
	\begin{equation}\label{def:kl}
	\varphi'(f(x)-f(x^{*}))\text{dist}(0,\partial f(x))\geqslant 1.
	\end{equation}
\end{enumerate}
Proper and l.s.c functions satisfying the KŁ property  $\forall x\in \dom\partial f$ are called KŁ functions. 
\end{definition} 
The KŁ condition can be interpreted as the fact that, up to a reparameterization, the function is locally sharp around its critical points, i.e. we can bound its subgradients away from 0. As before, this condition is verified at all points  $x$ such that $0\notin \partial f(x)$.

\paragraph{How to verify that a function verifies the Kurdyka–Łojasiewicz property?} 
In the next section, we will make use of the KŁ property to study the convergence of iterative algorithms for minimizing the sum of two functions $F = f + g$ with $f$ or $g$ nonconvex. \newpage

The first condition will be to have $F$ KŁ. However, the KŁ condition is not stable by sum. Therefore, we  need to introduce conditions on $f$ and $g$ such that $f+g$ is KŁ. Moreover, directly verifying that a function verifies the KŁ inequality~\eqref{def:kl} is not easy. We need to choose a subclass of KŁ functions that, on the one hand, is large enough to encompass most functions of interest and, on the other hand, has minimal stability properties so that inclusion to that set is easy to verify. 
\textit{Semialgebraic functions} \cite{coste2000introduction} is such a convenient class of functions and is used in most works on nonconvex optimization \cite{attouch2010proximal, ABS13, bolte2018first}. 

\begin{definition}[Semialgebraic function] \label{def:semialgebraic}~
\begin{itemize}
    \item A subset $S$ of $\mathbb{R}^n$ is a real semialgebraic set if there exists a finite number of real polynomial functions $P_{i,j}, Q_{i,j} : \mathbb{R}^n \to \mathbb{R}$ such that 
    \begin{equation} \label{eq:semialgebraic}
        S = \cup_{j=1}^p \cap_{i=1}^q \{ x \in \mathbb{R}^n, P_{i,j} = 0, Q_{i,j} <0 \}.
    \end{equation}
    \item A function $f : \mathbb{R}^n \to \mathbb{R} \cup \{+\infty \}$ (resp. $f : \mathbb{R}^n \to \mathbb{R}^m$) is called semialgebraic if its graph $\{(x,y) \in \mathbb{R}^n \times \mathbb{R}, y = f(x) \}$ (resp. $\{(x,y) \in \mathbb{R}^n \times \mathbb{R}^m, y = f(x) \}$) is a semialgebraic subset of $\mathbb{R}^n \times \mathbb{R}$ (resp. $\mathbb{R}^n \times \mathbb{R}^m$).
\end{itemize}
\end{definition}
A semialgebraic function satisfies the Kurdyka-Łojasiewicz property \cite{bolte2007lojasiewicz}.
From Definition~\ref{def:semialgebraic}, we first verify that polynomial functions are semialgebraic functions. Some other typical semialgebraic maps are the indicator function of a semialgebraic set or the Euclidean norm. What makes the semialgebraic set of functions useful is that it has strong stability properties. The main ones follow from the \textit{Tarski-Seidenberg principle}, which states that the image of a semialgebraic subset of $\R^{n+m}$ by projection on the first $n$ coordinates is a semialgebraic subset of $\R^{n}$. From this theorem, the sum, product, composition or derivative of a semialgebraic functions are semialgebraic. These stability properties are very useful to prove that a given function is KŁ. 

\begin{remark}[On the KŁ properties of neural networks]
One can  prove that neural networks with ReLU activations are semialgebraic functions, and thus KŁ. Indeed neural networks basically consists of composition and sums of linear maps and activation functions.  A linear map being semialgebraic, the main difficulty is to show that activation functions are semialgebraic. This is the case of the ReLU which can be expressed with a system of polynomial (in)equalities: 
\begin{equation*}
    \textrm{for }x,y \in \R\textrm{, }y = \textrm{ReLU}(x) = \max(0,x) \textrm{ if and only if }y(y-x) = 0\textrm{, }y \geq x, y \geq 0.
\end{equation*}
    
\end{remark}
\paragraph{Nonconvex convergence to a critical point with Kurdyka–Łojasiewicz property}

We now explain why Kurdyka–Łojasiewicz functions are useful for convergence in the nonconvex setting. Given a KŁ function $f$, we want to show convergence of a sequence $(x_n)_{n \in \mathbb{N}}$, produced by a certain iterative procedure, towards a critical point of $f$. The following lemma from~\cite{ABS13} gives sufficient conditions on $(x_n)$ such that this convergence is verified. This comprehensive lemma will be the basis of the nonconvex convergence analysis of first-order optimization algorithms like Forward-Backward. 

\begin{lemme}[{\cite[Theorem 2.9]{ABS13}}]
\label{lem:KŁ_convergence}
    Let $f: \E\to \mathbb{R} \cup \{ + \infty \}$ be a proper lsc function. Consider a sequence $(x_n)_{n \in \mathbb{N}}$ satisfying the following conditions: 
    \begin{itemize}
        \item \textbf{H1:\,Sufficient decrease condition} $\forall n \in \mathbb{N}$, \vspace{-.1cm}
        \begin{equation}
            f(x_{n+1}) + a \norm{x_{n+1}-x_n}^2 \leq f(x_n) .
        \end{equation}
        \item \textbf{H2:\,Relative error condition}
        $\forall n \in \mathbb{N}$, there exists $\omega_{n+1} \in \partial f(x_{n+1})$ such that \vspace{-.1cm}
        \begin{equation}
            \norm{\omega_{n+1}} \leq b\norm{x_{n+1}-x_n} .
        \end{equation}
        \item \textbf{H3:\,Continuity condition}
        There exists a subsequence $(x_{k_i})_{i\in\mathbb{N}}$ and $\hat x \in \E$ such that \vspace{-.1cm}
        \begin{equation}
           x_{n_i} \to \hat x \ \ \text{and} \ \ f(x_{n_i}) \to f(\hat x) \ \ \text{as} \ \ i \to +\infty .
        \end{equation}
    \end{itemize}
    If $f$ verifies the Kurdyka–Łojasiewicz property at the cluster point $\hat x$ specified in H3, then $(x_n)$ converges to $\hat x$ as $k \to+\infty$ and $\hat x$ is a critical point of $f$. 
\end{lemme}

The first condition H1 represents the descent of the objective function $f$ along the iterates. Note that if $f$ is lower-bounded, with H1, as $f(x_k)$ decreases, we have convergence of the function values $f(x_n)$, as well as, by telescopic sum, \\ $\sum \norm{x_{n+1}-x_n}^2 < + \infty$. If $f$ is also coercive, then $\{f(x) < f(x_0)\}$ is necessarily bounded and the iterates remain bounded. We can then extract from $(x_n)$ a subsequence converging towards $\hat x$.  With condition H2 and H3, we obtain $0 \in \partial f(\hat x)$ i.e. the limit point is a critical point of $f$. 

When dealing with a nonconvex function $f$, a sequence ($x_n$) satisfying the above conditions is not guaranteed to converge to a single point. It can be the case if $f$ is flat or highly oscillating around its critical points. 
The KŁ property is a general and flexible condition that prevents the above cases and ensures single-point convergence of any sequence satisfying H1, H2 and H3.

\subsubsection{Nonconvex Forward-Backward single-point convergence}\label{sec:ncvx_single_point}

We remind the Forward-Backward algorithm
\begin{equation}
    \label{eq:PGD_2}
    x_{n+1}\in\prox_{\gamma g} \left(x_n- \gamma \nabla f(x_n)\right).
\end{equation}
We can use the abstract nonconvex convergence result of Lemma~\ref{lem:KŁ_convergence}, with the Kurdyka–Łojasiewicz (KŁ) property, to prove that if $\mathcal{J} = f + g$ is KŁ and the sequence generated by the Forward-Backward algorithm is bounded, the sequence converges towards a critical point of the objective function.  

\begin{theoreme}[Single point convergence of nonconvex Forward-Backward.]
\label{thm:nonconvex_PGD}
Assume $f$ and $g$ proper, lsc, lower-bounded, with $f$ $L$-smooth and $\mathcal{J} = f + g$ KŁ. Then, for $\gamma < 1/L$, if the iterates $(x_n)$ given by the Forward-Backward algorithm~\eqref{eq:PGD_2} are bounded, then they converge towards a critical point of $\mathcal{J} = f + g$.
\end{theoreme}
\begin{proof}
    This result is a direct application of the general nonconvex convergence result from Lemma~\ref{lem:KŁ_convergence}. We need to verify its assumptions \textbf{H1}, \textbf{H2} and \textbf{H3}.
    \begin{itemize}[leftmargin=.5cm]
        \item \textbf{H1: Sufficient decrease condition}
    \end{itemize}
        
We first reformulate the Forward-Backward iterates as 
\begin{align} 
      x_{n+1} &\in \prox_{\gamma g} \circ \left( x_n - \gamma \nabla f(x_n) \right) \\
      \Leftrightarrow \ \ x_{n+1} &\in \argmin_x g(x) + \frac{1}{2\gamma}\norm{x -  \left( x_n - \gamma \nabla f(x_n) \right)}^2 \\
     \Leftrightarrow \ \  x_{n+1} &\in\argmin_x g(x) + \langle x-x_n, \nabla f(x_n) \rangle + \frac{1}{2\gamma}\norm{x-x_n}^2 .
      \label{eq:PGD_form2}
\end{align}  
Hence by evaluating the right-hand side at $x_{n+1}$ and $x_{n}$, and adding $f(x_n)$ on both sides, we get
\begin{equation} \label{eq:ineq_proof_NC_PGD}
f(x_n) + g(x_{n+1}) + \langle x_{n+1}-x_n, \nabla f(x_n) \rangle + \frac{1}{2 \gamma}\norm{x_{n+1}-x_n}^2 \leq f(x_n) + g(x_n) = \mathcal{J}(x_n) .
\end{equation} 
Then, using the descent lemma (Lemma~\ref{LemmeMajQuad}), $f$ being $L$-smooth, we obtain
\begin{align}
&\,f(x_n)+\langle x_{n+1}-x_n, \nabla f(x_n) \rangle + \frac{1}{2\gamma}\norm{x_{n+1}-x_n}^2 \nonumber\\
=   &\,f(x_n)+\langle x_{n+1}-x_n, \nabla f(x_n) \rangle + \frac{L}{2}\norm{x_{n+1}-x_n}^2 + \left(\frac{1}{2\gamma} - \frac{L}{2}\right)\norm{x_{n+1}-x_n}^2\nonumber \\ \label{eq:DL_applied_PGD}
\geq  &\,f(x_{n+1})  + \left(\frac{1}{2\gamma} - \frac{L}{2}\right)\norm{x_{n+1}-x_n}^2,
\end{align} 
leading to the sufficient decrease relation 
\begin{equation} \label{eq:sufficient_decrease_PGD}
\mathcal{J}(x_{n}) \geq \mathcal{J}(x_{n+1}) + \left(\frac{1}{2\gamma} - \frac{L}{2}\right)\norm{x_{n+1}-x_n}^2 .
\end{equation} 
\begin{itemize}[leftmargin=.5cm]
\item \textbf{H2: Relative error condition}\vspace{-.1cm}
\end{itemize}
By optimality of the proximal operator of $g$, wehave , for all $n \geq 0$
        \begin{equation} \label{eq:subdif_PGD}
                \frac{x_{n+1}-x_{n}}{\gamma} - \nabla f(x_{n+1}) \in \partial g(x_{n+1}),
        \end{equation}
        and thus $\omega_n = \frac{x_{n+1}-x_n}{\gamma}  \in \partial g(x_{n+1}) + \nabla f(x_{n+1}) = \partial \mathcal{J}(x_{n+1})$. 
        
        \begin{itemize}[leftmargin=.5cm]
         \item \textbf{H3: Continuity condition}\vspace{-.1cm}
         \end{itemize}
      The iterates $(x_n)$ are assumed bounded. Thus, there exists a subsequence $(x_{n_i})_{i \in \mathbb{N}}$ converging towards $\hat x \in \E$ as $i \to +\infty$. If we can show that for such a subsequence  $\lim_{i \to \infty} g(x_{n_i}) = g(\hat x)$, by continuity of $f$, we get $\lim_{i \to \infty} \mathcal{J}(x_{n_i}) = \mathcal{J}(\hat x)$ and the continuity condition is verified.
       Using the fact that $g$ is l.s.c we first have
        \begin{equation}
            \liminf_{i \to \infty} g(x_{n_i})  \geq g(x).
        \end{equation}
        On the other hand, by optimality of \eqref{eq:PGD_form2}, we obtain 
        \begin{equation} 
        \begin{split}
                    &g(x_{n_i+1}) + \langle x_{n_i+1}-x_{n_i}, \nabla f(x_{n_i}) \rangle + \frac{1}{2 \gamma}\norm{x_{n_i+1}-x_{n_i}}^2 
                    \\ \leq &\;g(x) + \langle x-x_{n_i}, \nabla f(x_{n_i}) \rangle + \frac{1}{2 \gamma}\norm{x-x_{n_i}}^2 .
        \end{split}
        \end{equation} 
        Using that $x_{n_i} \to x$ and $x_{n_i+1}-x_{n_i} \to 0$ when $i \to +\infty$, we get
        \begin{equation}
        \limsup_{i \to \infty} g(x_{n_i}) \leq g(x),
        \end{equation}
        and 
        \begin{equation}
        \lim_{i \to \infty} g(x_{n_i}) = g(x).
        \end{equation}

\end{proof}
\begin{remark} 
The boundedness of the iterates is verified as soon as the objective $\mathcal{J}$ is \textbf{coercive}. Indeed, it ensures that $\{ \mathcal{J}(x) \leq \mathcal{J}(x_0) \}$ is bounded and, since $\mathcal{J}(x_k)$ is non-increasing, the iterates remain bounded. 
\end{remark}
\begin{remark} 
When the function under the proximal operator (here~$g$) is assumed $\alpha$-weakly convex (i.e. $g(x)+\frac{\alpha}2||x||^2$ is convex), one can make the inequality \eqref{eq:ineq_proof_NC_PGD} sharper and finally relax the the condition on the stepsize in Theorem~\ref{thm:nonconvex_PGD} from \\ $\gamma < 1/L$ to $\gamma < 2/(L+\alpha)$. In particular, if $g$ is convex, the condition on the stepsize becomes $\gamma < 2 / L$. 
\end{remark}

\subsection{Other proximal splitting algorithms}\label{sec:nonconvDR}
We finally present some generalizations of Douglas-Rachford, ADMM and Primal-dual algorithms to the nonconvex setting.
\subsubsection{Nonconvex Douglas-Rachford and ADMM}

We remind the Douglas-Rachford (DR) algorithm introduced in Sections~\ref{sec:DR} 
\be \label{eq:ncvx_DR}
\forall n\in\N \left\{\begin{array}{ll}
y_n&\in\prox_{\gamma f}(x_n),\\
z_n&\in\prox_{\gamma g}(2y_n-x_n),\\
x_{n+1}&=x_n+\mu(z_n-y_n).
\end{array}\right.
\ee
As shown in Proposition~\ref{equiv_admmdr}, for $\mu=1$, the DR algorithm is equivalent to the ADMM algorithm introduced in Section~\ref{sec:ADMM}. In \cite{li2016douglas} and~\cite{themelis2020douglas}, convergence proofs of the DR algorithm are proposed for the minimization of the sum of two nonconvex functions $\mathcal{J}=f+g$, one of the two functions (here $f$) being $L$-smooth. The authors of \cite{themelis2020douglas} generalize the result from~\cite{li2016douglas} with a less restrictive stepsize condition. Both consider as Lyapunov function the \emph{Douglas-Rachford envelope}~\cite{themelis2020douglas} (or \emph{Douglas-Rachford merit function}~\cite{li2016douglas})
\begin{equation}
    \mathcal{J}^{DR}_{\gamma}(x,y,z) = f(y) + g(z) + \frac{1}{\gamma} \langle y-x, y-z \rangle + \frac{1}{2\gamma}\norm{y-z}^2.
 \label{eq:DRE}
\end{equation}

Similar to the previous result on Forward-Backward convergence, two convergence results can be derived in the nonconvex setting. First, without the KŁ hypothesis, one can show convergence of the Douglas-Rachford envelope $\mathcal{J}^{DR}_{\gamma}(x_n, y_n, z_n)$ along the iterates and convergence to $0$ of $\norm{x_{n+1}-x_n}$. Second, invoking the Kurdyka–Łojasiewicz (KŁ) property, we get convergence towards a critical point of the objective function. Both results are encompassed in the following theorem. 

\begin{theoreme}[\cite{li2016douglas, themelis2020douglas}] \label{thm:nonconvex_DRS}
    Assume that $f$ and $g$ are proper, lsc, lower-bounded and that $f$ is $L$-smooth and $\alpha$-weakly convex. Then, for a stepsize
   \begin{equation}\label{thm:stepsize}
   0 < \gamma < \min \left( \frac{2-\mu}{2\alpha},\frac{1}{L} \right),
   \end{equation}
   the sequence $(x_n,y_n,z_n)$ generated by the DR algorithm~\eqref{eq:ncvx_DR} verifies
    \begin{itemize}
        \item[(i)] $\mathcal{J}^{DR}_{\gamma}(x_{n-1},y_n,z_n)$ is non-increasing and converges \cite[Theorem 4.1]{themelis2020douglas}.
        \item[(ii)] $x_n - x_{n-1} = \frac{\mu}{2} (y_n - z_n)$ tends to $0$ with rate $\min_{n \leq N} \norm{y_n - z_n} = \mathcal{O}(\frac{1}{\sqrt{N}})$ \cite[Theorem 4.2]{themelis2020douglas}.
        \item[(iii)] $(y_n)$ and $(z_n)$ have the same cluster points, which are critical points of $\mathcal{J}$ \cite[Theorem 4.2]{themelis2020douglas}.
        \item[(iv)] If the sequence $(x_n,y_n,z_n)$ is bounded, and if $\mathcal{J}^{DR}_\gamma$ is KŁ, then the sequences $(y_n)$ and $(z_n)$ converge to the same critical point of $\mathcal{J}$ \cite[Theorem 2]{li2016douglas}.
    \end{itemize}
    \end{theoreme}

    \begin{remark}
        \begin{itemize}
            \item[(i)] We can assume $f$ weakly convex without loss of generality. Indeed, as $f$ is assumed $L$-smooth, $f$ is at least $L$-weakly convex and we have necessarily $\alpha \leq L$. 
            \item[(ii)] Contrary to what we had with Forward-Backward, for DR, the decreasing function is not $\mathcal{J}$ itself but the Lyapunov function  $\mathcal{J}^{DR}$. We can verify (\cite[Theorem 4]{li2016douglas}) that if $f$ or $g$ are coercive, $\mathcal{J}^{DR}$ is coercive and the iterates remain bounded. 
        \end{itemize}
    \end{remark}
We refer to \cite{wang2019global} for more specific results on the convergence of the ADMM in the nonconvex setting.
\subsubsection{Nonconvex primal version of  the  Chambolle-Pock algorithm}

The Chambolle-Pock algorithm~\eqref{eqAlgoPD} can be written in a fully primal version using Moreau's identity (Corollary~\ref{Moreau}). We then get rid of the convex conjugate and the algorithm is well-defined for \textbf{nonconvex} $f$ and $g$ functions, and writes
\begin{equation}\label{eq:CP_primal}
\left\{\begin{array}{lll}
z_{n+1}&\in \prox_{\frac{1}{\sigma} f}( \frac{1}{\sigma}y_n+ K\bar{x}_n) \\
y_{n+1}&= y_n + \sigma (K\bar{x}_k-z_{n+1}) \\
x_{n+1}&\in \prox_{\tau g}(x_n-\tau K^*y_{n+1})\\
\bar{x}_{n+1}&=x_{n+1} + \beta(x_{n+1} - x_n) .
\end{array}\right.
\end{equation}
    
The authors of \cite{mollenhoff2015primal} study the convergence of this algorithm for weakly convex $f$ and strongly convex $g$, also assuming that the strong convexity of $g$ dominates the weak convexity of $f$ in order to ensure that the overall objective $F = f(K.) + g$ is convex. In the fully nonconvex general case, the authors of \cite{sun2018precompact} study the convergence of the algorithm~\eqref{eq:CP_primal} with $\beta=0$, for nonconvex $f$ and $g$ functions. Similar to the DRS case in Theorem~\ref{thm:nonconvex_DRS}, they give a sufficient decrease condition on a particular Lyapunov function and prove convergence of the iterates only if these iterates are assumed to remain bounded. However, the coercivity of the functions $f$ and $g$ does not imply the coercivity of their proposed Lyapunov function and the boundedness of the iterates. 
The  boundedness of the iterates has next been proven in~\cite[Theorem 12]{hurault2023convergent} provided $f$ is differentiable with Lipschitz gradient.

\subsubsection{Nonlinear Chambolle-Pock algorithm}
We finally mention the problem
$$\min_x f(K(x))+g(x),$$ 
for convex functions $f$ and $g$ associated to a nonlinear operator $K$. As shown in~\cite{clason2017primal}, it is possible to study the nonconvex/concave  problem
$$\min_x\max_p \langle K(x),p\rangle -f^*(p)+g(x)$$
and show the local convergence of a linearized Chambolle-Pock algorithm inspired from~\eqref{eqAlgoPD2}:
\be\label{eqAlgoPD2nl}
\left\{\begin{array}{lll}
x_{n+1}&=&\prox_{\gamma g}(x_n-\gamma (\nabla_{x_n} K)^*y_{n+1})\\
\bar{x}_{n+1}&=&2x_{n+1}-x_n\\
y_{n+1}&=&\prox_{\sigma f^*}(y_n+\sigma K\bar{x}_{n+1}).
\end{array}
\right.
\ee   
\section{Various imaging problems, various optimization algorithms}
\label{sec:examples}

For any problem, several algorithms may be available. Choosing one algorithm among all possible existing ones may be a matter of personal taste and we do not provide any recommendation. In this section, we rather give practical imaging problems and detail which algorithms can be used to solve these problems. In particular cases, this will require a preliminary reformulation such as dualization or dimension extension.

In all the following examples, we want to minimize the sum of functions  defined in $\R^n$ and we assume that a minimizer always exists.
We will in particular focus on inverse problems in which we have a noisy signal $y=Ax^*+\epsilon$ corresponding to the observation of a clean signal $x^*$ measured through a linear operator $A$ and perturbed by an additive Gaussian noise $\epsilon$. 

Such a formulation includes problems such as deblurring (with a smoothing operator $A$), super-resolution (with the combination of a smoothing operator with a subsampling one) or inpainting (with a masking operator $A$).

\paragraph{Lasso}
The LASSO in statistics is an estimator that is computed by solving a problem of the  form:
\begin{equation}\label{eqLASSO}
\min_x \frac{1}{2}\norm{Ax-y}_2^2+\lambda\norm{x}_1 .
\end{equation}
The $L_1$-norm regularization term promotes sparse solutions.
Such a class of problems can be efficiently treated with the Forward-Backward algorithm, with a proximal step on the $L_1$-norm term. When the operator $A$ is positive semi-definite, the problem is even $||A||$-strongly convex and  accelerated schemes such as FISTA provide optimal convergence rates.
\paragraph{Orthogonal wavelet regularization}
When dealing with images, it is possible to regularize the problem using an orthogonal wavelet decomposition
$T^*T=TT^*=\id$:\begin{equation*}
\min_x \frac{1}{2}\norm{Ax-y}_2^2+\lambda\norm{Tx}_1
\end{equation*}
Thanks to the orthogonality of $T$, a closed form expression of the proximal operator of  $||Tx||_1$ is available.
Hence, we can consider all the algorithms dedicated to non smooth optimization, using the proximity operator of $g(x)=\norm{Tx}_1$ and an explicit or an implicit descent on $f(x)=\frac{1}{2}\norm{Ax-y}_2^2$. 
\paragraph{Total Variation Denoising}
A standard method to regularize ill-posed inverse problems is to promote piece-wise constant images through the minimization of  the total variation (TV):
\begin{equation}\label{ProxTV}
\min_x \frac{1}{2}\norm{x-y}_2^2+\lambda\norm{\nabla x}_1,
\end{equation}
where $\nabla=[\nabla_x, \nabla_y]$ is the $2D$ spatial gradient operator. 
The optimization of~\eqref{ProxTV} requires the  computation of the proximal operator of $g(x)=\norm{\nabla x}_1$ for which no closed-form expression exists. 
\noindent This issue can be treated in different ways:
\bi[leftmargin=.5cm]
\item Using Douglas-Rachford or PPXA algorithms (see sections~\ref{sec:DR} and~\ref{sec:PPXA}) thanks to the dimension extension:
\begin{equation*}
\hspace*{-0.5cm}\min_{x,z} \underbrace{\frac{1}{2}\norm{x-y}_2^2+\lambda\norm{z}_1}_{f(x,z)}+\underbrace{\iota_D(x,z)}_{g(x,z)}
\text{ with }D=\{(x,z)\text{ s.t. }\nabla z=x\}.\end{equation*}
The projection on $D$ can be computed using conjugate gradient (with inner loops) or in closed-form if the gradient is periodic, as the proximal operator can be computed in an explicit way in the Fourier domain.
\item Using Primal Dual algorithms (Chambolle-Pock and Condat) thanks to the saddle-point formulation:
\begin{equation}\label{ProxTV2}
\min_x \max_p\frac{1}{2}\norm{x-y}_2^2+\langle \nabla x,p\rangle-\iota_{B}(p),
\end{equation}
where $\iota_{B}$ is the convex conjugate of $||.||_1$, with $$B=\{p \text{ such that }\norm{p}_{\infty}\leqslant \lambda\}.$$
With respect to problem~\eqref{eqPD1}, this corresponds to  $f(.)=||.||_1$,  $K=\nabla$ and $g(x)=\frac12||x-y||^2$.
\newpage
\item Using the Forward-Backward algorithm~\eqref{algo:FB} on  the dual problem associated to~\eqref{ProxTV}:
\begin{equation}\label{ProxTVD}
\min_p \underbrace{\frac{1}{2}\norm{y+\nabla^*p}_2^2}_{f(p)}+\underbrace{\iota_{\mathcal{B}}(p)}_{g(p)}
\end{equation}
where $\mathcal{B}=\{p \text{ such that }\norm{p}_{\infty}\leqslant \lambda\}$. 
Indeed, as shown in \cite{Chambolle2004}, if $p^*$ is a minimizer of~\eqref{ProxTVD} then $x^*=y+\nabla^*p^*$ is a solution of \eqref{ProxTV}.
\ei

\paragraph{Total variation regularization} 
We now consider that an operator is included in the data fidelity term in the inverse problem regularized by TV
$$\min_x \frac{1}{2}\norm{Ax-y}_2^2+\lambda\norm{\nabla x}_1.$$
Without any assumptions on the operator $A$, the proximal operator of the data fidelity term $\frac{1}{2}\norm{Ax-y}_2^2$ has no closed-form expression. To avoid approximate computation of this proximal operator with inner loops, one can consider primal-dual strategies:
\begin{itemize}[leftmargin=.5cm]
    \item The Primal-dual algorithm~\eqref{algo:condat} by Condat can  be directly applied, with an  explicit descent on $f$ and a dualization of the TV term. Looking at the problem~\eqref{pb:condat}, it corresponds to  $f(x)=\frac12||Ax-y||^2$, $g=0$,  $h_1(.)=||.||_1$ and $L_1=\nabla$.

 \item The  Chambolle-Pock algorithm can also be considered with the following reformulation including two dualizations:
\begin{equation*}
\min_x\max_{p,q}\langle p, Ax\rangle +\langle q,\nabla x\rangle -\frac{1}{2}\norm{y+p}_2^2-\iota_{\mathcal{B}}(q)
\end{equation*}
where $\mathcal{B}=\{q \text{ such that }\norm{q}_{\infty}\leqslant \lambda\}$, $\iota_{B}$ is the convex conjugate of $||.||_1$ and $\frac12||.+y||^2$ is the convex conjugate of $\frac12||.-y||^2$.
The saddle-point problem~\eqref{eqPD3} is recovered with  $y=[p,q]$ $f^*(y)=\frac{1}{2}\norm{y+p}_2^2+\iota_{\mathcal{B}}(q)$, $K=[A,\nabla]$ and $g=0$.
\end{itemize}

\paragraph{TV Denoising with non smooth data fidelity} Problems involving non Gaussian noise (such as salt and pepper, speckle or Laplacian) may require non smooth data attachment terms such as:
\begin{equation*}
\min_x\frac{1}{2}\norm{x-y}_1+\lambda\norm{\nabla x}_1
\end{equation*}
Primal-dual algorithms can directly be applied to this problem, whereas Douglas-Rachford and PPXA algorithms require a dimension extension:
\begin{equation*}
\min_{x,z}\underbrace{\norm{x-y}_1+\lambda\norm{z}_1}_{f}+\underbrace{\iota_D(x,z)}_{g}\text{ with }D=\{(x,z)\text{ such that }\nabla z=x\},
\end{equation*}
and inner loops to project onto the set $D$.

\paragraph{Image fusion} We consider here the Poisson Image Editing problem proposed in~\cite{Perez03a} that consists in inserting in a smooth way a  source image $s$ into a target image $y$ within an area delimited by a mask $\Omega$. This problem can be solved by minimizing the following function:
\begin{equation}
\min_x\frac{1}{2}\norm{\nabla s-\nabla x}_{\Omega}^2+\iota_{D}(x)\quad \text{where }D=\{z\text{ such that }z_{|\overline{\Omega}}=y_{|\overline{\Omega}}\}.
\end{equation}
This optimization problem may be solved using the Forward-Backward algorithm, which here corresponds to a project gradient algorithm: an  explicit gradient descent on $\frac{1}{2}\norm{\nabla s-\nabla x}_{\Omega}^2$ is followed by a projection on the set $D$.

\paragraph{Regularization with an image denoiser}

Given an image denoiser operator $D_\sigma$ built to remove Gaussian noise with standard deviation $\sigma$ from an image, one can consider   regularization functions such as
\begin{equation}
    g_{\sigma}(x) = \langle x, x-D_\sigma(x) \rangle\hspace{.5cm}\textrm{or} \hspace{.5cm} g_{\sigma}(x) = \frac12|| x-D_\sigma(x)||^2,
\end{equation}
as respectively proposed in the  Regularization by Denoising (RED)~\cite{romano2017little} or gradient step Plug-and-Play~\cite{hurault2021gradient} frameworks.
These regularization are then used to solve a variety of inverse problems by minimizing
\begin{equation} \label{eq:obj_RED}
\min_x f(x)+\lambda g_{\sigma}(x),
\end{equation}
where $f$ is a data-fidelity term.

The denoiser $D_\sigma$ is typically a neural network which has been trained to denoise. With such denoiser, the RED regularization function $g_{\sigma}$ is typically smooth (i.e. with Lipschitz gradient) but nonconvex. Depending on the regularity of the data-fidelity term, different algorithms can be used to solve this nonconvex objective. If $f$ is smooth, one can use a standard (nonconvex) gradient descent; if $f$ is non-smooth but has closed-form proximal operator, one can use the forward-backward algorithm with proximal step on $f$ and gradient step on $g_\sigma$. Finally, it can happen that $f$ is non-smooth and has no closed form proximity operator, for example when the observation noise has a Poisson distribution. In such situation, it can be useful to change the optimization algorithm to adapt to the particular geometry of $f$, using Bregman optimization algorithms ~\cite{bolte2018first,hurault2024convergent}.

\section*{Acknowledgements}

This work was funded in part by the Agence nationale de la recherche (ANR), Grant ANR-23-
CE40-0017 and by the France 2030 program, with the reference ANR-23-PEIA-0004.

\appendix

\section{Convergence of Chambolle-Pock algorithm}\label{sec:app:CP}
We here reproduce the convergence proof of the Chambolle-Pock algorithm introduced in \cite{ChambollePock}. We first recall that this algorithm targets a saddle-point of the following minmax problem: \be\label{eqPD4}
\min_{x\in\E}\max_{y\in F}\,\ps{K^{*}y}{x}-f^{*}(y)+g(x),
\ee
and that we rely on the function:
\be\label{eqdefh2}
h(x,y)=\ps{Kx}{y}+g(x)-f^*(y).
\ee
related to the 
{\it partial primal-dual gap}:
\be
\mathcal{G}_{B_1\times B_2}(x,y)=\max_{y'\in B_2}\,\ps{y'}{Kx}-f^{*}(y')+g(x)-\min_{x'\in B_1}\,\ps{y}{Kx'}+g(x')-f^*(y)\geq 0.
\ee
\begin{theoreme}[Chambolle-Pock Algorithm]
Let $f$ and $g$ be two convex, proper, lower semi-continuous functions, with $f$ defined from $F$ to $[-\infty,+\infty]$, $g$ defined from $\E$ to $[-\infty,+\infty]$ and $K$ a linear operator from $F$ to $\E$. We suppose that the problem \eqref{eqPD3} admits a solution $(\hat x,\hat y)$. Let us denote $L=\norm{K}=\sqrt{||K^*K||}$, and take $\sigma$ and $\tau$ such that $\tau \sigma L^2<1$. We choose $(x_0,y_0)\in\E\times F$ and set $\bar{x}_0=x_0$. We define the sequences $(x_n)_{n\in\N},\,(y_n)_{n\in\N},\,\text{and }(\bar{x}_n)_{n\in\N}$ by
\be\label{eqAlgoPD0}
\left\{\begin{array}{lll}
y_{n+1}&=&\prox_{\sigma f^*}(y_n+\sigma K\bar{x}_n)\\
x_{n+1}&=&\prox_{\tau g}(x_n-\tau K^*y_{n+1})\\
\bar{x}_{n+1}&=&2x_{n+1}-x_n\\
\end{array}
\right.
\ee   
Then
\begin{enumerate}
\item For all $n\in\N$,
\be
\frac{\norm{y_n-\hat y}^2}{2\sigma}+\frac{\norm{x_n-\hat x}^2}{2\tau}\leqslant 
(1-\tau\sigma L^2)^{-1}\left(\frac{\norm{y_0-\hat y}^2}{2\sigma}+\frac{\norm{x_0-\hat x}^2}{2\tau}\right)
\ee
\item If we set $x^N=(\sum_{n=1}^N x_n)/N$ and $y^N=(\sum_{n=1}^N y_n)/N$, then for each bounded set $(B_1\times B_2)\subset \E\times F$, the {\it partial primal-dual gap} satisfies
$$\mathcal{G}_{B_1\times B_2}(x^N,y^N)\leqslant \frac{D(B_1,B_2)}{N}$$
where
$$D(B_1,B_2)=\underset{(x,y)\in B_1\times B_2}{\sup}\frac{\norm{x-x_0}^2}{2\tau}+\frac{\norm{y-y_0}^2}{2\sigma}.$$
\item The sequence $(x_n,y_n)_{n\in\N}$ converges to a solution $(x^*,y^*)$ of \eqref{eqPD3}.
\end{enumerate}
\end{theoreme}

\begin{proof}
  
In order to prove this theorem we will rewrite the iterations \eqref{eqAlgoPD0} on the following form:
\be\label{eqAlgoPD3}
\left\{\begin{array}{ccc}
y_{n+1}&=&\prox_{\sigma f^*}(y_n+\sigma K\bar{x})\\
x_{n+1}&=&\prox_{\tau g}(x_n-\tau K^*\bar{y})\\
\end{array}
\right.
\ee
The fundamental properties of the proximity operators ensure that
\begin{align*}
K\bar{x}+\frac{y_n-y_{n+1}}{\sigma}\in\partial f^{*}(y_{n+1})\\
-K^*\bar{y}+\frac{x_n-x_{n+1}}{\sigma}\in\partial g(x_{n+1})
\end{align*}
and thus, for all $(x,y)\in \E\times F$,
\begin{align*}
f^{*}(y)\geqslant f^{*}(y_{n+1})+\ps{\frac{y_n-y_{n+1}}{\sigma}}{y-y_{n+1}}+\ps{K\bar{x}}{y-y_{n+1}}\\
g(x)\geqslant g(x_{n+1})+\ps{\frac{x_n-x_{n+1}}{\tau}}{x-x_{n+1}}-\ps{K(x-x_{n+1})}{\bar{y}}
\end{align*}
By adding these two inequalities we obtain:
\be\label{eq12}
\begin{split}
h(x_{n+1},y)-h(x,y_{n+1})\hspace{8cm}\\
+\frac{\norm{y-y_{n+1}}^2}{2\sigma}+\frac{\norm{x-x_{n+1}}^2}{2\tau}+
\frac{\norm{y_n-y_{n+1}}^2}{2\sigma}+\frac{\norm{x_n-x_{n+1}}^2}{2\tau}\\
+\ps{K(x_{n+1}-\bar{x})}{y_{n+1}-y}-\ps{K(x_{n+1}-x)}{y_{n+1}-\bar{y}}\\
\leqslant \frac{\norm{y-y_n}^2}{2\sigma}+\frac{\norm{x-x_n}^2}{2\tau}
\end{split}
\ee
Chambolle and Pock propose to choose $\bar{y}=y_{n+1}$ and $\bar{x}=2x_n-x_{n-1}$. Thus the second last line of the previous inequality can be written
\be
\begin{split}
&\,\ps{K(x_{n+1}-\bar{x})}{y_{n+1}-y}-\ps{K(x_{n+1}-x)}{y_{n+1}-\bar{y}}\\
=&\,\ps{K((x_{n+1}-x_n)-(x_n-x_{n-1}))}{y_{n+1}-y}\\
=&\,\ps{K(x_{n+1}-x_n)}{y_{n+1}-y}-\ps{K(x_{n}-x_{n-1})}{y_{n}-y}\\
&\,-\ps{K(x_{n}-x_{n-1})}{y_{n+1}-y_n}\\
\geqslant&\, \ps{K(x_{n+1}-x_n)}{y_{n+1}-y}-\ps{K(x_{n}-x_{n-1})}{y_{n}-y}\\
&\,-L\norm{x_n-x_{n-1}}\norm{y_{n+1}-y_n}.
\end{split}
\ee

By using the fact that for all $\alpha>0$ we have $2ab\leqslant \alpha a^2+\frac{b^2}{\alpha}$, we obtain
\be
L\norm{x_n-x_{n-1}}\norm{y_{n+1}-y_n}\leqslant \frac{L\alpha \tau}{2\tau}\norm{x_n-x_{n-1}}^2+\frac{L\sigma}{2\alpha\sigma}
\norm{y_{n+1}-y_n}^2
\ee
By combining~\eqref{eq12} with the previous inequality taken with $\alpha=\sqrt{\frac{\sigma}{\tau}}$, we deduce that for all $(x,y)\in\E\times F$,
\be\label{eq14}
\begin{split}
h(x_{n+1},y)-h(x,y_{n+1})\hspace{8cm}\\
+\frac{\norm{y-y_{n+1}}^2}{2\sigma}+\frac{\norm{x-x_{n+1}}^2}{2\tau}+(1-\sqrt{\sigma \tau}L)\frac{\norm{y_n-y_{n+1}}^2}{2\sigma}\\
+\frac{\norm{x_n-x_{n+1}}^2}{2\tau}-\sqrt{\sigma \tau}L\frac{\norm{x_{n-1}-x_n}^2}{2\tau}\\
+\ps{K(x_{n+1}-x_n)}{y_{n+1}-y}-\ps{K(x_{n}-x_{n-1})}{y_n-y}\\
\leqslant \frac{\norm{y-y_n}^2}{2\sigma}+\frac{\norm{x-x_n}^2}{2\tau}.
\end{split}
\ee
Summing the previous inequalities from $n=0$ to $N-1$, and taking $x_{-1}=x_0$, we obtain:
\begin{equation*}
\begin{split}
\sum_{n=1}^N h(x_{n},y)-h(x,y_{n})\hspace{7cm}\\
+\frac{\norm{y-y_{N}}^2}{2\sigma}+\frac{\norm{x-x_{N}}^2}{2\tau}+(1-\sqrt{\sigma \tau}L)\sum_{n=1}^N\frac{\norm{y_n-y_{n-1}}^2}{2\sigma}\\
+(1-\sqrt{\sigma \tau}L)\sum_{n=1}^{N-1}\frac{\norm{x_n-x_{n-1}}^2}{2\tau}+\frac{\norm{x_{N}-x_{N-1}}^2}{2\tau}\\
\leqslant \frac{\norm{y-y_0}^2}{2\sigma}+\frac{\norm{x-x_0}^2}{2\tau}
+\ps{K(x_{N}-x_{N-1})}{y_N-y}.
\end{split}
\end{equation*}
We use again the upper bound $2ab\leqslant \alpha a^2+\frac{b^2}{\alpha}$, now  with $\alpha=1/(\sigma\tau)$:
$$\ps{K(x_N-x_{N-1})}{y_N-y}\leqslant 
\norm{x_N-x_{N-1}}^2/(2\tau)+(\sigma\tau L^2)\norm{y-y_N}^2/(2\sigma)
$$
and we obtain the following relation:
\be\label{eqPDSplit}
\begin{split}
\sum_{n=1}^N h(x_{n},y)-h(x,y_{n})\hspace{7cm}\\
+(1-\sigma \tau L^2)\frac{\norm{y-y_N}^2}{2\sigma}+\frac{\norm{x-x_N}^2}{2\tau}
+(1-\sqrt{\sigma \tau}L)\sum_{n=1}^N\frac{\norm{y_n-y_{n-1}}^2}{2\sigma}\\
+(1-\sqrt{\sigma \tau}L)\sum_{n=1}^N\frac{\norm{x_n-x_{n-1}}^2}{2\tau}\leqslant 
\frac{\norm{y-y_0}^2}{2\sigma}+\frac{\norm{x-x_0}^2}{2\tau}
\end{split}
\ee

We apply this inequality to a saddle-point $(x,y)=(\hat x,\hat y)$ of \eqref{eqPD4}. The first line of \eqref{eqPDSplit} is the sum of {\it partial primal-dual gaps} with $(x,y)$ solutions to \eqref{eqPD4}. As we have noted previously, this implies that all the terms of the sum from this first line are positive. We deduce the first result of the theorem from the fact that $\tau\sigma L^2<1$.\\
We now consider an arbitrary pair $(x,y)\in B_1\times B_2$. As
\be
\sum_{n=1}^N h(x_n,y)-h(x,y_n)\leqslant 
\frac{\norm{y-y_0}^2}{2\sigma}+\frac{\norm{x-x_0}^2}{2\tau}.
\ee
As $f^*$ and $g$ are convex, we deduce that 
$$
h(x^N,y)-h(x,y^N)\leqslant \frac{1}{N}\left(
\frac{\norm{y-y_0}^2}{2\sigma}+\frac{\norm{x-x_0}^2}{2\tau}\right)
$$
By taking the supremum on $(x,y)\in B_1\times B_2$ of the two sides of the inequality, we obtain the second point of the Lemma.\\ 

Using the first point of the theorem we know that the sequences $(x_n)_{n\in\N}$ and $(y_n)_{n\in\N}$ are bounded. As the dimension of $E$ is finite, from each of these sequences we can extract subsequences that converge to some limits $x^*$ et $y^*$. By taking the limit when $N$ goes to $+\infty$ in the previous inequality we obtain
$$
\ps{Kx^*}{y}-f^{*}(y)+g(x^*)-(\ps{Kx}{y^*}-f^{*}(y^*)+g(x))\leqslant 0
$$
for all $(x,y)\in B_1\times B_2$. By taking the supremum over the pairs $(x,y)\in B_1\times B_2$, we deduce that the {\it partial primal-dual gap} in the point $(x^*,y^*)$ is negative and thus zero, and thus $(x^*,y^*)$ is a solution to \eqref{eqPD4}.\\

As the sequences  $(x_n)_{n\in\N}$ and $(y_n)_{n\in\N}$ are bounded we can extract subsequences $(x_{n_k})_{nk\in\N}$ and $(y_{n_k})_{k\in\N}$ that converge to some points $x^{\infty}$ et $y^{\infty}$. As the series having the general term $\norm{x_n-x_{n-1}}^2$ and $\norm{y_n-y_{n-1}}^2$ are convergent we deduce that the sequences $(\norm{x_n-x_{n-1}}^2)_{n\in\N}$ and $(\norm{y_n-y_{n-1}}^2)_{n\in\N}$ go to zero when $n$ goes to $+\infty$. We deduce that the sequences $(x_{n_k+1})_{nk\in\N}$ et $(y_{n_k+1})_{k\in\N}$ also converge to $x^{\infty}$ et $y^{\infty}$ and thus $x^{\infty}$ et $y^{\infty}$ are fixed points of \eqref{eqAlgoPD0}. 
Fixed points of \eqref{eqAlgoPD0} satisfy  relations \eqref{eq:saddlePD}. Thus they are saddle-points,  that is to say solutions of~\eqref{eqPD4}.\\
By adding the inequalities \eqref{eq14} for $n=n_k$ to $N$ we obtain
\be
\begin{split}
\frac{\norm{y^*-y_N}^2}{2\sigma}+\frac{\norm{x^*-x_N}^2}{2\tau}
+(1-\sqrt{\sigma \tau}L)\sum_{n=n_k+1}^N\frac{\norm{y_n-y_{n-1}}^2}{2\sigma}\hspace{1.5cm}\\
-\frac{\norm{x_{n_k}-x_{n_k-1}}^2}{2\tau}+
(1-\sqrt{\sigma \tau}L)\sum_{n=n_k}^{N-1}\frac{\norm{x_n-x_{n-1}}^2}{2\tau}+
\frac{\norm{x_N-x_{N-1}}}{2\tau}\\
+\ps{K(x_N-x_{N-1})}{y_N-y^*}-\ps{K(x_{n_k}-x_{n_k-1})}{y_{n_k}-y^*}\\
\leqslant \frac{\norm{y^*-y_{n_k}}^2}{2\sigma}+\frac{\norm{x^*-x_{n_k}}^2}{2\tau}.
\end{split}
\ee
By using the fact that  $\underset{n\to+\infty}{\lim}\norm{x_{n}-x_{n-1}}=\underset{n\to+\infty}{\lim}\norm{y_{n}-y_{n-1}}=0$, we deduce that the sequence $(x_N,y_N)_{N\in\N}$ goes to $(x^*,y^*)$, which gives the third point of the theorem.
\end{proof}

\newpage
\bibliographystyle{abbrv}

\bibliographystyle{abbrv}

\begin{thebibliography}{10}

\bibitem{attouch2010proximal}
H.~Attouch, J.~Bolte, P.~Redont, and A.~Soubeyran.
\newblock Proximal alternating minimization and projection methods for
  nonconvex problems: An approach based on the kurdyka-{\l}ojasiewicz
  inequality.
\newblock {\em Mathematics of operations research}, 35(2):438--457, 2010.

\bibitem{ABS13}
H.~Attouch, J.~Bolte, and B.~Svaiter.
\newblock Convergence of descent methods for semi-algebraic and tame problems:
  proximal algorithms, forward-backward splitting, and regularized
  {G}auss-{S}eidel methods.
\newblock {\em Mathematical Programming}, 137(1-2, Ser. A):91--129, 2013.

\bibitem{aujol2023fista}
J.-F. Aujol, C.~Dossal, and A.~Rondepierre.
\newblock Fista is an automatic geometrically optimized algorithm for strongly
  convex functions.
\newblock {\em Mathematical Programming}, pages 1--43, 2023.

\bibitem{BauschkeCombettes}
H.~H. Bauschke and P.~Combettes.
\newblock {\em Convex Analysis and Monotone Operator Theory in Hilbert Spaces}.
\newblock Springer, 2011.

\bibitem{bauschke2012firmly}
H.~H. Bauschke, S.~M. Moffat, and X.~Wang.
\newblock Firmly nonexpansive mappings and maximally monotone operators:
  correspondence and duality.
\newblock {\em Set-Valued and Variational Analysis}, 20(1):131--153, 2012.

\bibitem{beck2017first}
A.~Beck.
\newblock {\em First-order methods in optimization}.
\newblock SIAM, 2017.

\bibitem{BeckTeboulle}
A.~Beck and M.~Teboulle.
\newblock A fast iterative shrinkage-thresholding algorithm for linear inverse
  problems.
\newblock {\em SIAM Journal on Imaging Sciences}, 2(1):183--202, 2009.

\bibitem{bolte2007lojasiewicz}
J.~Bolte, A.~Daniilidis, and A.~Lewis.
\newblock The {\l}ojasiewicz inequality for nonsmooth subanalytic functions
  with applications to subgradient dynamical systems.
\newblock {\em SIAM Journal on Optimization}, 17(4):1205--1223, 2007.

\bibitem{bolte2018first}
J.~Bolte, S.~Sabach, M.~Teboulle, and Y.~Vaisbourd.
\newblock First order methods beyond convexity and lipschitz gradient
  continuity with applications to quadratic inverse problems.
\newblock {\em SIAM Journal on Optimization}, 28(3):2131--2151, 2018.

\bibitem{Chambolle2004}
A.~Chambolle.
\newblock An algorithm for total variation minimization and applications.
\newblock {\em Journal of Mathematical Imaging and Vision}, 20(1--2):89--97,
  2004.

\bibitem{chambolle2015convergence}
A.~Chambolle and C.~Dossal.
\newblock On the convergence of the iterates of the “fast iterative
  shrinkage/thresholding algorithm”.
\newblock {\em Journal of Optimization theory and Applications}, 166:968--982,
  2015.

\bibitem{ChambollePock}
A.~Chambolle and T.~Pock.
\newblock A first-order primal-dual algorithm for convex problems with
  applications to imaging.
\newblock {\em Journal of Mathemtical Imaging and Vision 40}, 2011.

\bibitem{chambolle2016ergodic}
A.~Chambolle and T.~Pock.
\newblock On the ergodic convergence rates of a first-order primal--dual
  algorithm.
\newblock {\em Mathematical Programming}, 159(1-2):253--287, 2016.

\bibitem{clason2017primal}
C.~Clason and T.~Valkonen.
\newblock Primal-dual extragradient methods for nonlinear nonsmooth
  pde-constrained optimization.
\newblock {\em SIAM Journal on Optimization}, 27(3):1314--1339, 2017.

\bibitem{CombettesPesquet}
P.~Combettes and J.-C. Pesquet.
\newblock Proximal splitting methods in signal processing.
\newblock {\em Fixed-Point Algorithms for Inverse Problems in Science and
  Engineering}, 2011.

\bibitem{Condat13}
L.~Condat.
\newblock A primal-dual splitting method for convex optimization involving
  lipschitzian, proximable and linear composite terms.
\newblock {\em Journal of Optimization Theory and Applications},
  158(2):460--479, 2013.

\bibitem{condat2023proximal}
L.~Condat, D.~Kitahara, A.~Contreras, and A.~Hirabayashi.
\newblock Proximal splitting algorithms for convex optimization: A tour of
  recent advances, with new twists.
\newblock {\em SIAM Review}, 65(2):375--435, 2023.

\bibitem{coste2000introduction}
M.~Coste.
\newblock An introduction to semialgebraic geometry, 2000.

\bibitem{Eckstein1992}
J.~Eckstein and D.~P. Bertsekas.
\newblock On the {D}ouglas-{R}achford splitting method and the proximal point
  algorithm for maximal monotone operators.
\newblock {\em Mathematical Programming}, 55:293--318, 1992.

\bibitem{Gabay83}
D.~Gabay.
\newblock Chapter ix applications of the method of multipliers to variational
  inequalities.
\newblock In {\em Studies in mathematics and its applications}, volume~15,
  pages 299--331. Elsevier, 1983.

\bibitem{GabayMercier}
D.~Gabay and B.~Mercier.
\newblock A dual algorithm for the solution of nonlinear variational problems
  via finite element approximation.
\newblock {\em Computers \& Mathematics with Applications}, 2(1):17 -- 40,
  1976.

\bibitem{garrigos2023handbook}
G.~Garrigos and R.~M. Gower.
\newblock Handbook of convergence theorems for (stochastic) gradient methods.
\newblock {\em arXiv preprint arXiv:2301.11235}, 2023.

\bibitem{GlowinskiMarroco}
R.~Glowinski and A.~Marroco.
\newblock Sur l'approximation, par éléments finis d'ordre un, et la
  résolution, par pénalisation-dualité d'une classe de problèmes de
  dirichlet non linéaires.
\newblock {\em ESAIM: Mathematical Modelling and Numerical Analysis},
  9(R2):41--76, 1975.

\bibitem{he2012convergence}
B.~He and X.~Yuan.
\newblock Convergence analysis of primal-dual algorithms for a saddle-point
  problem: from contraction perspective.
\newblock {\em SIAM Journal on Imaging Sciences}, 5(1):119--149, 2012.

\bibitem{hurault2023convergent}
S.~Hurault.
\newblock {\em Convergent plug-and-play methods for image inverse problems with
  explicit and nonconvex deep regularization}.
\newblock PhD thesis, Universit{\'e} de Bordeaux, 2023.

\bibitem{hurault2024convergent}
S.~Hurault, U.~Kamilov, A.~Leclaire, and N.~Papadakis.
\newblock Convergent bregman plug-and-play image restoration for poisson
  inverse problems.
\newblock {\em Advances in Neural Information Processing Systems}, 36, 2024.

\bibitem{hurault2021gradient}
S.~Hurault, A.~Leclaire, and N.~Papadakis.
\newblock Gradient step denoiser for convergent plug-and-play.
\newblock In {\em International Conference on Learning Representations}, 2022.

\bibitem{kurdyka1998gradients}
K.~Kurdyka.
\newblock On gradients of functions definable in o-minimal structures.
\newblock {\em Annales de l'institut Fourier}, 48(3):769--783, 1998.

\bibitem{li2016douglas}
G.~Li and T.~K. Pong.
\newblock Douglas--rachford splitting for nonconvex optimization with
  application to nonconvex feasibility problems.
\newblock {\em Math. Progr.}, 159:371--401, 2016.

\bibitem{lojasiewicz1963propriete}
S.~Lojasiewicz.
\newblock Une propri{\'e}t{\'e} topologique des sous-ensembles analytiques
  r{\'e}els.
\newblock {\em Les {\'e}quations aux d{\'e}riv{\'e}es partielles}, 117:87--89,
  1963.

\bibitem{lojasiewicz1982trajectoires}
S.~Lojasiewicz.
\newblock Sur les trajectoires du gradient d’une fonction analytique.
\newblock {\em Seminari di geometria}, 1983:115--117, 1982.

\bibitem{lu2023unified}
H.~Lu and J.~Yang.
\newblock On a unified and simplified proof for the ergodic convergence rates
  of ppm, pdhg and admm.
\newblock {\em arXiv preprint arXiv:2305.02165}, 2023.

\bibitem{Moreau1965}
J.-J. Moreau.
\newblock Proximit\'e et dualit\'e dans un espace hilbertien.
\newblock {\em Bulletin de la Soci\'et\'e Math\'ematique de France},
  93:273--299, 1965.

\bibitem{mollenhoff2015primal}
T.~Möllenhoff, E.~Strekalovskiy, M.~Moeller, and D.~Cremers.
\newblock The primal-dual hybrid gradient method for semiconvex splittings.
\newblock {\em SIAM Journal on Imaging Sciences}, 8(2):827--857, 2015.

\bibitem{Nes83}
Y.~Nesterov.
\newblock A method for solving the convex programming problem with convergence
  rate {$O(1/k^{2})$}.
\newblock {\em Doklady Akademii Nauk SSSR}, 269(3):543--547, 1983.

\bibitem{papadakis2014optimal}
N.~Papadakis, G.~Peyr{\'e}, and E.~Oudet.
\newblock Optimal transport with proximal splitting.
\newblock {\em SIAM Journal on Imaging Sciences}, 7(1):212--238, 2014.

\bibitem{Perez03a}
P.~P{\'e}rez, M.~Gangnet, and A.~Blake.
\newblock Poisson image editing.
\newblock In {\em ACM SIGGRAPH 2003 Papers}, pages 313--318, 2003.

\bibitem{polyak1964some}
B.~T. Polyak.
\newblock Some methods of speeding up the convergence of iteration methods.
\newblock {\em Ussr computational mathematics and mathematical physics},
  4(5):1--17, 1964.

\bibitem{RockOnTheMax}
R.~Rockafellar.
\newblock On the maximality of sums of nonlinear monotone operators.
\newblock {\em Transactions of the American Mathematical Society.}, 149:75--88,
  1970.

\bibitem{rockafellar1997convex}
R.~Rockafellar.
\newblock {\em Convex Analysis}.
\newblock Convex Analysis. Princeton University Press, 1997.

\bibitem{rockafellar1976monotone}
R.~T. Rockafellar.
\newblock Monotone operators and the proximal point algorithm.
\newblock {\em SIAM journal on control and optimization}, 14(5):877--898, 1976.

\bibitem{rt1998wets}
R.~T. Rockafellar and R.~J.-B. Wets.
\newblock {\em Variational analysis}, volume 317.
\newblock Springer Science \& Business Media, 2009.

\bibitem{romano2017little}
Y.~Romano, M.~Elad, and P.~Milanfar.
\newblock The little engine that could: Regularization by denoising (red).
\newblock {\em SIAM Journal on Imaging Sciences}, 10(4):1804--1844, 2017.

\bibitem{rondepierre2017methodes}
A.~Rondepierre.
\newblock M{\'e}thodes num{\'e}riques pour l’optimisation non linaires
  d{\'e}terministe 4eme ann{\'e}e.
\newblock {\em INSA Toulouse}, 2017.

\bibitem{su2016differential}
W.~Su, S.~Boyd, and E.~J. Candes.
\newblock A differential equation for modeling nesterov's accelerated gradient
  method: Theory and insights.
\newblock {\em Journal of Machine Learning Research}, 17(153):1--43, 2016.

\bibitem{sun2018precompact}
T.~Sun, R.~Barrio, L.~Cheng, and H.~Jiang.
\newblock Precompact convergence of the nonconvex primal--dual hybrid gradient
  algorithm.
\newblock {\em Journal of Comp. Appl. Math.}, 330:15--27, 2018.

\bibitem{themelis2020douglas}
A.~Themelis and P.~Patrinos.
\newblock Douglas--rachford splitting and admm for nonconvex optimization:
  Tight convergence results.
\newblock {\em SIAM Journal on Optimization}, 30(1):149--181, 2020.

\bibitem{wang2019global}
Y.~Wang, W.~Yin, and J.~Zeng.
\newblock Global convergence of admm in nonconvex nonsmooth optimization.
\newblock {\em Journal of Scientific Computing}, 78:29--63, 2019.

\bibitem{weiss2015elements}
P.~Weiss.
\newblock \'el\'ements d'analyse et d'optimisation convexe.
\newblock {\em INSA Toulouse}, 2015.

\end{thebibliography}

\end{document}